\documentclass{amsart}

\usepackage[utf8]{inputenc}
\usepackage{amsmath, amssymb}

\usepackage{xcolor}
\usepackage[colorlinks,
    linkcolor={red!50!black},
    citecolor={blue!50!black},
    urlcolor={blue!80!black}]{hyperref}

\usepackage{enumerate}
\usepackage{tikz}
\usetikzlibrary{calc, matrix, arrows, cd, decorations.markings, knots}
\usetikzlibrary{decorations.pathmorphing}
\usetikzlibrary{decorations.pathreplacing}

\theoremstyle{plain}
\newtheorem{proposition}{Proposition}[section]
\newtheorem{theorem}[proposition]{Theorem}
\newtheorem{corollary}[proposition]{Corollary}
\newtheorem{lemma}[proposition]{Lemma}

\theoremstyle{definition}
\newtheorem{definition}[proposition]{Definition}

\newtheorem{example}[proposition]{Example}

\theoremstyle{remark}
\newtheorem{remark}[proposition]{Remark}
\newtheorem{const}[proposition]{Construction}
\newtheorem*{convention}{Conventions}
\newtheorem*{acknowledgements}{Acknowledgements}

\newcommand{\lk}{\operatorname{lk}}
\newcommand{\sign}{\operatorname{sign}}

\newcommand{\wt}{\widetilde}
\newcommand{\wh}{\widehat}
\newcommand{\ol}{\overline}
\newcommand{\sm}{\setminus}

\newcommand{\Z}{\mathbb{Z}}

\newcommand{\R}{\mathbb{R}}

\newcommand{\Id}{\operatorname{Id}}

\newcommand{\Hom}{\operatorname{Hom}}

\newcommand{\pt}{\operatorname{pt}}
\newcommand{\Int}{\operatorname{Int}}

\newcommand{\vspan}{\operatorname{span}}
\newcommand{\Pow}{\mathcal{M}}

\begin{document}
\title{Triple linking numbers and surface systems}

\author[C. W. Davis]{Christopher W.\ Davis}
\address{Department of Mathematics, University of Wisconsin--Eau Claire, USA}
\email{daviscw@uwec.edu}
\urladdr{people.uwec.edu/daviscw}

\author[M. Nagel]{Matthias Nagel}
\address{Department of Mathematics \& Statistics, McMaster University, Canada}
\email{nagel@cirget.ca}

\author[P. Orson]{Patrick Orson}
\address{Department of Mathematics, Boston College, USA}
\email{patrick.orson@bc.edu}

\author[M. Powell]{Mark Powell}
\address{D\'{e}partement de Math\'{e}matiques, Universit\'{e} du Qu\'{e}bec  \`{a} Montr\'{e}al, Canada}
\email{mark@cirget.ca}

\begin{abstract}
We give a refined value group for the collection of triple linking numbers of links in the $3$--sphere. Given two links with the same pairwise linking numbers we show that they have the same refined  triple linking number collection if and only if the links admit homeomorphic surface systems. Moreover these two conditions hold if and only if the link exteriors are bordant over~$B\Z^n$, and if and only if the third lower central series quotients $\pi/\pi_3$ of the link groups are isomorphic preserving meridians and longitudes.   We also show that these conditions imply that the link groups have isomorphic fourth lower central series quotients~$\pi/\pi_4$, preserving meridians.
\end{abstract}

\maketitle
\section{Introduction}

In this article, all links $L\subset S^3$ are $n$--component, ordered and oriented. The triple linking numbers $\overline{\mu}_L(ijk)$ are a measure of higher linking for $L$, introduced by Milnor~\cite{Milnor:1957-1}. Any given triple linking number $\overline{\mu}_L(ijk)$ is an integer, well-defined only up to the greatest common divisor $\Delta_L(ijk)$ of the pairwise linking numbers involving the link components labelled $i$, $j$ and $k$. We define a \emph{surface system} for the link~$L = K_1 \cup \cdots \cup K_n$ as a collection $\Sigma$ of
embedded, oriented (possibly disconnected) surfaces~$\Sigma_i = \Sigma_{K_i}$ in $S^3$ with~$\partial \Sigma_i = K_i$, that intersect one another transversally and in at most triple points.
Mellor and Melvin~\cite{Mellor03} derived a geometric method for computing the triple linking numbers as a difference of auxiliary quantities $m_{ijk}(\Sigma), t_{ijk}(\Sigma)\in \Z$.

Given a surface system $\Sigma$, for $1\leq i<j<k\leq n$ we collect the differences $m_{ijk}(\Sigma)-t_{ijk}(\Sigma)$ into an $\binom{n}{3}$--tuple $\mu(L)$.  We construct a certain quotient $\Pow$ of~$\Z^{\binom{n}{3}}$, called the \emph{total Milnor quotient}  (Definition \ref{defn:total-Milnor-invariant}), only depending on the pairwise linking numbers.
In this quotient the invariant $\mu(L)$ is defined independently of $\Sigma$ (Theorem~\ref{thm:MilnorWellDef}). Moreover, taking $\mu(L)$ in $\Pow$, rather than the classical indeterminacy group $\prod_{i<j<k}\Z/\Delta_L(ijk)$, in general strictly refines the Milnor invariants.
The main result of this paper is the following series of characterisations.

\begin{theorem}\label{thm:Main-intro}
Suppose that~$L$ and~$L'$ have the same pairwise linking numbers.  Then the following are equivalent:
\begin{enumerate}[(a)]
\item \label{item:main-thm-1} The links $L$ and $L'$ admit homeomorphic surface systems.
\item\label{item:main-thm-3} The collections of triple linking numbers $\mu(L)$ and $\mu(L')$ are equal in the total Milnor quotient~$\mathcal{M}$.
\item\label{item:main-thm-2} The link exteriors~$X_L$ and~$X_{L'}$, together with their canonical homotopy classes of maps to~$B\Z^n$, are bordant rel.\ boundary over~$B\Z^n$.
\item\label{item:main-thm-4} There exist basings for $L$ and $L'$ and
an isomorphism~$\pi_1(X_L)/\pi_1(X_L)_3 \xrightarrow{\simeq} \pi_1(X_{L'})/\pi_1(X_{L'})_3$ of the lower central series quotients that sends the ordered, oriented meridians of $L$ to those of $L'$,
and the ordered, oriented zero-framed longitudes of $L$ to those of $L'$.
\end{enumerate}
\end{theorem}

It is a direct consequence of the geometric calculation method of~\cite{Mellor03} that links admitting homeomorphic surface systems have the same pairwise linking and triple linking numbers. The equivalence of (\ref{item:main-thm-1}) and (\ref{item:main-thm-3}) can be thought of as confirming the converse when the Milnor invariants are taken in the refined value group $\Pow$. We note that it is not an original observation that Milnor invariants, when collected together, are well-defined in a more subtle value group than $\prod_{i<j<k}\Z/\Delta_L(ijk)$. Larger value groups were derived in~\cite{Levine-homotopy-classn} and also~\cite{HL90,HL98}, as we discuss below.
The precise notion of bordism rel.\ boundary over $B\Z^n$ is introduced in Section~\ref{section:bordism-rel-boundary}.

As a result of Theorem \ref{thm:Main-intro} (\ref{item:main-thm-4}), we also obtain a statement about the fourth lower central series quotients:

\begin{theorem}\label{theorem:lower-central-series-intro}
Suppose that $L$ and $L'$ satisfy the conditions in Theorem~\ref{thm:Main-intro}.
Then there is an isomorphism~$\pi_1(X_L)/\pi_1(X_L)_4 \cong \pi_1(X_{L'})/\pi_1(X_{L'})_4$ between the lower central series quotients that preserves the free homotopy classes of the oriented, ordered meridians.
\end{theorem}

This theorem is directly analogous to the result, which follows from a well-known argument of Milnor~\cite[proof of Theorem 4]{Milnor:1957-1}, that equality of pairwise linking numbers implies an isomorphism between the lower central series quotients~$\pi_1(X_L)/\pi_1(X_L)_3 \cong \pi_1(X_{L'})/\pi_1(X_{L'})_3$ that preserves the free homotopy classes of the oriented, ordered meridians. Details of all these arguments are given in Section~\ref{section:lower-central-series-quotients}.

\subsection{A refinement for the collection of triple linking numbers}

Let~$\Sigma$ be a surface system for a link~$L=K_1\cup\cdots\cup K_n$. The two integers~$m_{ijk}(\Sigma)$ and $t_{ijk}(\Sigma)$ defined by Mellor and Melvin depend both on a choice of surface system for the link and a choice of base point for each link component.

The integers~$t_{ijk}(\Sigma)$ are the signed count of triple intersection points
in the surfaces for the link components $K_i$, $K_j$ and $K_k$. The integers $m_{ijk}(\Sigma)$ are determined by the \emph{clasp-words}.
These are words, one for each
component~$K_m$ of $L$, in the labels of the link components, that record the order in which the component $K_m$ intersects surfaces in a surface system for the link $L$, starting from some chosen base point of~$K_m$.
Precise details, including how to produce the integers $m_{ijk}(\Sigma)$ from the clasp-words, are given in Section~\ref{section:milnor-numbers}.

Indeterminacy in the differences $m_{ijk}(\Sigma)-t_{ijk}(\Sigma)$ arises from two sources.
Firstly, the choice of  surface system for the link $L$, which we deal with in Section~\ref{section:indet-surface-systems}.
Secondly, the choice of base points used to read off the clasp-words in the computation of the $m_{ijk}(\Sigma)$, examined in Section~\ref{section:basepoints}.
A change in choice of surfaces, or a change in the choice of base points, produces a change on several of the integers $m_{ijk}(\Sigma) - t_{ijk}(\Sigma)$ simultaneously.
This led us to look at the \emph{$\binom{n}{3}$--tuple} of integers $\{m_{ijk}(\Sigma) - t_{ijk}(\Sigma)\}_{i<j<k}$. We take its image in the quotient $\Pow$ of $\Z^{\binom{n}{3}}$ by linear combinations of the \emph{indeterminacy elements}, defined in Lemma~\ref{lem:IndetermineVectors}, which are geometrically motivated and depend on the linking numbers.

In Example~\ref{remark:total-quotient-is-a-massive-emphatic-awesome-refinement} we show the following.

\begin{proposition}\label{prop:improved}For $4$--component links with all linking numbers equal to 1, there is an isomorphism $\Pow \cong \Z$ and every integer in this quotient is realised by a link.
\end{proposition}

In contrast to this, when all linking numbers are 1, the classical indeterminacies~$\Delta_L(ijk)=1$ for all triple indices $ijk$ so the classical value group is trivial. Thus Proposition~\ref{prop:improved} shows the total Milnor invariant is in general a refinement for the classical indeterminacy.

In the case of $4$--component links and non repeating Milnor invariants of length up to and including $4$, the type of refined indeterminacy captured by our $\Pow$ was previously considered by Levine~\cite{Levine-homotopy-classn}. In this special case of $4$--component links, our indeterminacy elements recover the indeterminacy given by Levine's automorphisms $\phi_{r,s}$~\cite[p.~373]{Levine-homotopy-classn}; cf.\ the first three columns of \cite[Table~1]{Levine-homotopy-classn}.
It also seems likely that our refinement could also be extracted from the \emph{universal Milnor invariant} of Habegger and Lin~\cite{HL90,HL98}.   However, in their own words, the ``complicated nature'' of certain features of their algebraic approach ``conspire to make it difficult, if not impossible, to find a complete set of invariants'' for their value group~\cite[p.~414]{HL90}.

We suggest the reader attempts a calculation of triple linking numbers using the Mellor-Melvin formulation, as in Example~\ref{example:example-link}, in order to appreciate the ease with which the triple linking numbers can be computed by constructing a surface system with double intersections only, and reading off clasp-words. Such a surface system always exists; see Section \ref{section:C-complexes}. Trying to apply the Mellor-Melvin formulation to a generic surface system that includes triple intersection points can be a task for the more intrepid geometric topologist. But, while harder to actually calculate, the triple intersection numbers are more obviously related to part (\ref{item:main-thm-2}) of Theorem~\ref{thm:Main-intro}, the bordism side of the problem.

\subsection{Bordism rel.\ boundary over \texorpdfstring{$B\Z^n$}{BZn}}\label{section:bordism-rel-boundary}

Consider two $3$--manifolds $X_1,X_2$ with homeomorphisms $g_i \colon \coprod_n S^1 \times S^1 \xrightarrow{\cong} \partial X_i$, for $i=1,2$, and homotopy classes of maps $f_i \in [X_i,B]$, for some space $B$ and for~$i=1,2$.  For the exterior~$X_i$ of an oriented, ordered $n$--component link and~$B=B\Z^n$, the oriented meridians and the zero-framed longitudes determine $g_i$ up to isotopy and $f_i$ up to homotopy. The pairs~$(X_1,f_1)$ and $(X_2,f_2)$ are said to be \emph{bordant rel.\ boundary over $B$} if there exists a 4--manifold $W$ with boundary~$M := - X_1\cup_{g_2 \circ g_1^{-1}} X_2$ and a map~$F \colon W \to B$ such that~$F|_{X_i} \sim f_i$ for $i=1,2$.

To characterise when two knot exterior pairs~$(X_1,f_1)$ and~$(X_2,f_2)$ are bordant rel.\ boundary over $B\Z^n$, we first use the $g_i$ to create the closed~$3$--manifold~$M = -X_1 \cup X_2$, as above, and then attempt to glue the maps $f_1$ and $f_2$ accordingly, in order to analyse the pair~$(M,f_1\cup f_2)$ in the bordism group~$\Omega_3(B\Z^n)$. However, while some choice of map-gluing can always be made, the homotopy classes of~$f_1$ and~$f_2$ do not determine a \emph{unique} homotopy class of a map $f_1\cup f_2\in[M,B]$. This subtlety is closely related to the indeterminacies in the triple linking numbers, so next we indicate the extra structure required to glue the maps in a well-defined way.

Let $X$ be a $3$--manifold with boundary~$\Sigma$.  Fix some space~$B$ and a continuous map~$\phi \colon \Sigma \to B$, and suppose that~$X$ comes equipped with a parametrisation of its boundary,
namely a homeomorphism~$g \colon \Sigma \to \partial X$. A \emph{bordered $B$--structure} on~$(X,\Sigma, g,\phi)$ is
a map~$f \colon X \to B$ together with a
homotopy~$H \colon f \vert_{\partial X} \circ g \sim \phi$, recording the fact that the diagram below commutes up to homotopy:
\[
\begin{tikzcd}
 \partial X \ar[r]  & X \ar[d,"f"] \\
\Sigma \ar[r, "\phi"] \ar{u}{g}[swap]{\cong} & B.
\end{tikzcd}
\]

It is the choice of $H$ in the bordered $B$--structure that enables us to glue homotopy classes of maps together in a well-defined fashion. More precisely we have the following. We say that two bordered $B$--structures~$(f, H)$ and $(f', H')$ are \emph{homotopic}, if there exists a homotopy~$F \colon f \sim f'$, and a
homotopy~$\Phi \colon F \vert_{\partial X \times I} \circ (g\times \Id) \sim \phi \circ \operatorname{pr}_{\Sigma}$ between the two maps~$\Sigma \times I \to B$
such that $\Phi \vert_{ (\Sigma \times \{0\}) \times I } = H$ and
$\Phi \vert_{ (\Sigma \times \{1\}) \times I } = H' $.
Given two $3$--manifolds~$X_1, X_2$ with
bordered $B$--structures~$(f_1, H_1), (f_2,H_2)$,
we can construct~$M = - X_1 \cup_{g_1 \circ g_2^{-1}} X_2$ and a
map~$f = f_1 \cup f_2$. We now have enough structure so that the homotopy class of~$f$ only depends
on the homotopy classes of bordered $B$--structures~$(f_1, H_1), (f_2,H_2)$.
Furthermore, if we restrict the new map~$f$, we recover~$f \vert_{X_i} \sim f_i$ the former maps~$f_i$, for~$i=1,2$.

We have already noted that a link exterior $X_L$ comes equipped with canonical data $(X_L,\Sigma,g)$ and $f\in[X_L,B\Z^n]$, and in fact the map $\phi\colon \Sigma\to B\Z^n=T^n$ is also canonically determined, by the linking numbers of $L$. So we see that to equip a link complement~$X_L$ with a bordered $B\Z^n$--structure, we need only choose the homotopy~$H \colon f \vert_{\partial X_L} \circ g \sim \phi$. However, there is no preferred choice. Understanding the relationship between this choice and the triple linking numbers was a key step in proving Theorem~\ref{thm:Main-intro}.

We contrast this with the case of $3$--manifolds with empty boundary.
Here the gluing indeterminacy is not a feature, and
a result similar to Theorem~\ref{thm:Main-intro} was already obtained
by Cochran, Gerges and Orr~\cite[Theorem 3.1]{Cochran01}.
One might be tempted to try and directly relate our result to theirs by closing up the link exteriors with solid tori. However, for a link~$L$ with non-vanishing linking numbers, the canonical
map~$X_L \to B\Z^n$ does not extend over any filling of
the boundary tori with solid tori, so the results are not related in this way.

\subsection{Lower central series quotients}\label{section:lowercentral}

Recall that the lower central series of a group $G$ is a descending sequence of subgroups defined iteratively by $G_1 := G$ and~$G_n := [G,G_{n-1}]$. In Section \ref{section:lower-central-series-quotients}, we consider the lower central series of the link group~$\pi_1(X_L)$. We will recall two well-known results about lower central series quotients and pairwise linking numbers, and show how to prove the analogous results one lever further down the series using triple linking numbers.

The first well-known result is that the pairwise linking numbers of two links $L$ and $L'$ are the same if and only if the lower central series quotients~$\pi(L)/\pi(L)_2$ and~$\pi(L')/\pi(L')_2$ are isomorphic via an isomorphism that sends meridians to meridians and longitudes to longitudes. When there is equality of pairwise linking numbers, the characterisation Theorem \ref{thm:Main-intro} (\ref{item:main-thm-4}) says that the precisely analogous isomorphism of the third lower central series quotients holds if and only if the refined triple linking numbers agree.

The second well-known result (which follows from an argument recalled in Theorem \ref{thm:milnor-thm-4}) is that, given equality of pairwise linking numbers and a choice of oriented, ordered meridians for $L$ and $L'$, the lower central series quotients~$\pi(L)/\pi(L)_3$ and~$\pi(L')/\pi(L')_3$ are isomorphic, via an isomorphism that preserves the meridians. Theorem \ref{theorem:lower-central-series-intro} is a consequence of the appearance of (\ref{item:main-thm-4}) in Theorem~\ref{thm:Main-intro}, and proves the analogue for the refined triple linking numbers.

\subsection{C-complexes}\label{section:C-complexes}

An important concept motivating this article, which does not appear in the statement of Theorem~\ref{thm:Main-intro}, is that of a \emph{C-complex}. A C-complex is a surface system that consists of Seifert surfaces and only has \emph{clasp} intersections~\cite{Cooper82,Cimasoni04,Cimasoni-Florens}.  A clasp is a double point arc that has end points on distinct link components, shown in Figure~\ref{fig:ClaspIntersection}.  More details are given in Section~\ref{section:surface-systems}.  As mentioned above, C-complexes always exist and are often a very useful computational tool; see e.g.~\cite{Cimasoni-Florens, Mellor03}.

In the special cases that the linking numbers of a link $L$ are zero, or that $n=2$, the triple linking numbers $\overline{\mu}_L(ijk)$ are well-defined as integers. In these cases, it was proven by Davis and Roth~\cite{Davis16} that two links admit homeomorphic C-complexes if and only if their linking and triple linking numbers agree. They then asked~\cite[Question~1]{Davis16} about the generalisation to links with nonzero linking numbers,  which the following corollary to Theorem~\ref{thm:Main-intro} answers.

\begin{corollary}\label{cor-main-C-complexes}
  Suppose that~$L$ and~$L'$ have the same pairwise linking numbers. Then the links $L$ and $L'$ admit homeomorphic C-complexes if and only if the collections of triple linking numbers $\mu(L)$ and $\mu(L')$ are equal in the total Milnor quotient~$\mathcal{M}$.
\end{corollary}

\begin{proof}
If two links have homeomorphic C-complexes, then they trivially have homeomorphic surface systems.  If two links admit homeomorphic surfaces systems, then their pairwise linking numbers coincide, and by Theorem~\ref{thm:Main-intro} they are bordant over~$B\Z^n$.
The proof of Theorem~\ref{thm:2->1}, that one can isotope a surface system through that bordism
from one exterior to the other, also works for C-complexes; see Remark~\ref{rem:CSweep}.
\end{proof}

\subsection{Outline of the proof of Theorem~\ref{thm:Main-intro}}

We give a summary of our strategy in the proof of Theorem~\ref{thm:Main-intro}. The equivalences are proved as (\ref{item:main-thm-1})$\implies$(\ref{item:main-thm-3})$\implies$(\ref{item:main-thm-2})$\implies$(\ref{item:main-thm-1})$\implies$(\ref{item:main-thm-4}) $\implies$(\ref{item:main-thm-3}). We will also explain how to directly obtain (\ref{item:main-thm-1})$\implies$(\ref{item:main-thm-2}), as this explanation helps in understanding the other stages.

First we consider the implication (\ref{item:main-thm-1})$\implies$(\ref{item:main-thm-3}).
Suppose that $L$ and $L'$ bound homeomorphic surface systems $\Sigma$ and $\Sigma'$.  The linking numbers, the clasp-words and the count of triple intersection points, are all preserved by the homeomorphism between the surface systems.   Thus the total Milnor invariant $\mu(L)$ of $L$ agrees with the total Milnor invariant $\mu(L')$ of $L'$, and we see that  (\ref{item:main-thm-1})$\implies$(\ref{item:main-thm-3}) follows fairly easily from the definitions.

Next we explain why (\ref{item:main-thm-1})$\implies$(\ref{item:main-thm-2}).
Elements $(M,f)$ of the bordism group~$\Omega_3(B\Z^n)$ are detected by taking preimages under~$f$ of codimension $3$ sub-tori in the model $(S^1)^n \simeq B\Z^n$. The preimages are points, and the algebraic count of these points gives rise to $\binom{n}{3}$ integers that determine whether two 3--manifolds are $B\Z^n$--bordant, as proven in Theorem~\ref{thm:FillableIntersection}.  Let $X_L := S^3 \sm \nu L$ be the exterior of $L$, that is the complement of a regular neighbourhood $\nu L$ of~$L$.
A surface system gives rise to a map $X_L \to B\Z^n$,
produced in Construction~\ref{const:CollapseMap}, which follows the Pontryagin-Thom collapse construction.
If two links $L$ and $L'$ have homeomorphic surface systems, then after an isotopy of the surface systems, the resulting maps $X_L \to B\Z^n$ and $X_{L'} \to B\Z^n$ agree on the boundaries.  Thus the link exteriors can be glued together over $B\Z^n$.  This glues the surface systems together too.   The preimages detecting $\Omega_3(B\Z^n)$ are the triple intersection points of the surface system.  The triple intersection points cancel algebraically, because the gluing reverses orientations on one of these systems, so the link exteriors of $L$ and $L'$ are bordant rel.\ boundary.

Now we consider the converse, namely (\ref{item:main-thm-2})$\implies$(\ref{item:main-thm-1}).  This is proved in Theorem~\ref{thm:2->1}.
We recall in Lemma~\ref{lemma-W-only-2-handles} how to modify a bordism rel.\ boundary from $X_L$ to $X_{L'}$ over~$B\Z^n$ to a different bordism by replacing $1$--handles
with $2$--handles as in~\cite[Proof of 4.2]{Cochran01}.
For a bordism arising from $2$--handle attachments only, there exists a stabilised surface system for $L$ that isotopes through the bordism unchanged to give rise to a surface system for $L'$.

Next we discuss the implication (\ref{item:main-thm-3})$\implies$(\ref{item:main-thm-2}).  As noted in the discussion of (\ref{item:main-thm-1})$\implies$(\ref{item:main-thm-2}) above, in order to show that two $3$--manifolds are bordant over $B\Z^m$, we have to arrange that the maps agree on the boundary, and that the triple intersection numbers of the surface systems $\Sigma$ and $\Sigma'$  arising as the preimages of the maps to $B\Z^n$ agree (Theorem~\ref{thm:FillableIntersection}).  In order to achieve this, the key geometric move (Lemma~\ref{lem:MakeBoundaryAlike}) switches two clasps, modifying $m_{ijk}(\Sigma)$ and $t_{ijk}(\Sigma)$ in the same way, thus preserving their difference.  Repeated application of this move, together with a tubing operation (Figure~\ref{fig:tubing}) that removes adjacent algebraically cancelling intersection points, arranges that the clasp-words of both links agree, and therefore the terms $m_{ijk}$ agree.  Moreover, as above the surface systems can be isotoped so that the maps to $B\Z^n$ determined by the resulting systems agree on the boundaries of $X_L$ and $X_{L'}$.  After this, we alter the surface systems using
the \emph{torus sum} operation,  given in Construction~\ref{construction:boundary sum}, to arrange that the tuples $\{m_{ijk}(\Sigma) - t_{ijk}(\Sigma)\}_{i<j<k}$ and $\{m_{ijk}(\Sigma') - t_{ijk}(\Sigma')\}_{i<j<k}$ agree in $\Z^{\binom{n}{3}}$, and not just in the total Milnor quotient.  The torus sum operation fixes $m_{ijk}(\Sigma)$.  It will follows that the terms $t_{ijk}$ agree for both link exteriors.  Since these integers detect whether the link exteriors are bordant rel.\ boundary over $B\Z^n$, for the maps to $B\Z^n$ determined by the surface systems, this will complete the proof that (\ref{item:main-thm-3})$\implies$(\ref{item:main-thm-2}).

Finally, we consider the implications involving (\ref{item:main-thm-4}).  To show that (\ref{item:main-thm-1})$\implies$(\ref{item:main-thm-4}), we show that the longitudes of the link components, as elements of the lower central series quotient $\pi_1(X_L)/\pi_1(X_L)_3$, can be read off from the combinatorial data of the position of the clasps in a C-complex.  Note that the longitudes contain more information than the clasp-words: different occurrences of the same index in a clasp-word might appear in the longitude word with different conjugations.  To show (\ref{item:main-thm-4})$\implies$(\ref{item:main-thm-3}), we prove that the longitudes, considered as elements of the quotient $\pi_1(X_L)/\pi_1(X_L)_3$, determine the total Milnor invariant of~$L$.

\begin{convention}
All links are oriented, ordered and have $n \geq 3$ components.
Mathematical objects indexed by a knot component~$K_i$
may equivalently be addressed simply by the natural number~$i$, for brevity.
\end{convention}

\begin{acknowledgements}
We thank Anthony Conway, Tye Lidman, Kent Orr and Carolyn Otto for helpful discussions.
MN and PO were supported by CIRGET postdoctoral fellowships. MP was supported by an NSERC Discovery grant.
\end{acknowledgements}

\section{Surface systems}\label{section:surface-systems}

\begin{definition}\label{defn:surface-system}
A \emph{surface system} for the link~$L = K_1 \cup \cdots \cup K_n$ is a collection of
embedded, oriented (possibly disconnected) surfaces~$\Sigma_i = \Sigma_{K_i}$ in~$S^3$
with~$\partial \Sigma_i = K_i$, that intersect one another transversally and in at most triple points.
\end{definition}

A pair of two surfaces in a surface system potentially intersect each other
in circles, ribbons, and clasps; see e.g.~\cite{Cimasoni04}.

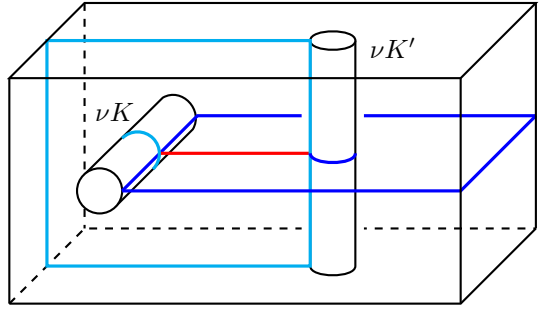
\begin{figure}
\begin{tikzpicture}

\draw[thick, black, dashed] (0,0) -- (1,1) -- (7,1);
\draw[thick, black] (7,1) -- (6,0) -- (0,0);
\draw[thick, black, dashed] (1,4) -- (1,1);

\draw[thick, black] (1.41,1.29)
to [out=45, in=-135] (2.44,2.32)
to [out=60, in=30, looseness=1.3] (2.04, 2.74)
to [out=-135,in=45] (0.99,1.71);
\draw[thick, black, fill=white] (1.2, 1.5) circle [radius=0.3];

\draw[very thick, blue] (1.5, 1.5) -- (2.5,2.5) -- (7,2.5);


\draw[thick, white, fill=white] (3.9,3.5)
to [out=down, in=up] (3.9,0.5)
to [out=down, in=down, looseness=0.7] (4.7, 0.5)
to [out=up, in=down] (4.7, 3.5);

\draw[thick, black, fill=white] (4,3.5)
to [out=down, in=up] (4,0.5)
to [out=down, in=down, looseness=0.7] (4.6, 0.5)
to [out=up, in=down] (4.6, 3.5);

\draw[thick, black, fill=white] (4,3.5)
to [out=up, in=up, looseness=0.7] (4.6,3.5)
to [out=down, in=down, looseness=0.7] (4,3.5);


\draw[very thick, red] (4,2) -- (2,2);

\draw[very thick, cyan] (4,3.5) -- (0.5, 3.5) -- (0.5, 0.5) -- (4,0.5) -- cycle;

\draw [very thick, cyan] (1.5,2.2) arc (135:-45:0.29);

\draw [very thick, blue] (4,2)
to [out=down, in=down, looseness=0.7] (4.6, 2);

\draw[very thick, blue] (7,2.5) -- (6,1.5) -- (1.5,1.5);

\draw[thick, black] (0,3) -- (1,4) -- (7,4) -- (6,3) --cycle;
\draw[thick, black] (6,3) -- (6,0);
\draw[thick, black] (7,1) -- (7,4);
\draw[thick, black] (0,0) -- (0,3);

\node at (1.4,2.55) {$\nu K$};
\node at (5.1,3.4) {$\nu K'$};

\end{tikzpicture}
\caption{The exterior of a link near a clasp intersection of the components~$K$ and $K'$.
}\label{fig:ClaspIntersection}
\end{figure}

We slightly modify the definition of clasp-words given in~\cite{Davis16}.
Given a surface system~$\Sigma$, we equip each component~$K$ of $L$ with the following data:
write $I_{K, i} \subset K$ for the set of intersection points of $K$
with $\Sigma_i$ and write $I_K = \bigcup_i I_{K,i}$ for their union.

Given a point~$x \in I_{K,i}$, we consider the
sign~$\varepsilon_x$ of the intersection point~$x$
between the two oriented submanifolds~$K$
and $\Sigma_{i}$, and assign to~$x$ the pair~$(i, \varepsilon_x)$.
This gives a map~$w_K \colon I_K \to \{1, \ldots, n\} \times \{\pm 1\}$.
Often, we abbreviate the tuple~$(i, \varepsilon)$ to~$i^\varepsilon$.

\begin{definition}\label{defn:cyclic-word}
Let $E$ be a set.
A \emph{cyclic word}~$w$ in the letters~$E$ is a map~$w \colon I \to E$
for $I \subset S^1$ a finite subset. Two cyclic words~$w$ and $v$ are considered to be equivalent if there exists
an orientation preserving diffeomorphism~$f \colon S^1 \to S^1$ such that
$v = w \circ f$. Such a map~$f$ is called an \emph{alignment} between $v$ and $w$.
\end{definition}

\begin{remark}
Given a cyclic word and a starting point in $S^1 \setminus I$,
we can read off the letters in positive direction and obtain a
(linear) word~$\wt w$. Two cyclic words~$w_0$ and $w_1$ are equivalent
if $\wt w_0$ can be obtained from $\wt w_1$ by a cyclic permutation
of the letters.
Note that there is no cancellation of letters at this point.
\end{remark}

\begin{definition}\label{defn:clasp-word}
Let $L$ be a link with a surface system. For each component~$K$
of $L$, the map~$w_K$ defines a cyclic word in the letters~$\{1, \ldots, n\} \times \{\pm 1\}$,
called the \emph{clasp-word}.
\end{definition}

We say that a boundary collar~$\nu \partial X_L = (-\varepsilon, 0] \times L \times S^1$ is \emph{adapted}
to a surface system~$\Sigma$, if we have the following two conditions:
\begin{enumerate}
\item $\Sigma_i \cap (-\varepsilon, 0] \times K_i \times S^1
= (-\varepsilon, 0] \times K_i \times \{1\}$, and
\item for $i \neq j$
we have~$\Sigma_j \cap (-\varepsilon, 0] \times K_i \times S^1 =(-\varepsilon, 0] \times I_{K_i, j} \times S^1$.
\end{enumerate}
We use these collars to glue two link exteriors with their surface systems together.

\begin{const}[Double exterior]\label{const:DoubleExt}
Let $L, L'$ be two links with surface systems~$\Sigma_L$ and $\Sigma_{L'}$ whose clasp-words agree.
Let $f_i \colon K_i \to K'_i$ be an alignment between the clasp-words~$w_i$ and $w'_i$.
Pick two adapted collars~$\nu (\partial X_L) =(-\varepsilon, 0] \times L \times S^1$
and $\nu (\partial X_{L'}) =(-\varepsilon, 0] \times L' \times S^1$. Remove the boundaries
of the exteriors~$X_L$
and $X_{L'}$ and glue them together via the following map:
\begin{align*}
f \colon (-\varepsilon, 0) \times L \times S^1 &\to (-\varepsilon, 0) \times L' \times S^1\\
(t, x, z) &\mapsto (-\varepsilon - t, f_i(x), z)
\end{align*}
for $x \in K_i$.

This defines a closed $3$--manifold~$M = -X_{L} \cup_f X_{L'}$, the \emph{double
exterior}. Inside~$M$, for each $i=1, \ldots, n$, build closed embedded oriented surfaces
\[ F_i = -(\Sigma_{L, i} \cap X_L) \cup (\Sigma_{L', i} \cap X_{L'}). \]
These surfaces intersect each other in circles and triple intersection points.
\end{const}

\begin{remark}
The diffeomorphism type of $M$ does not depend on the choice of alignment.
The isotopy type of the surfaces~$F_i$ does not depend on the choice
of adapted collar. On the other hand, different choices of alignments
can result in different surfaces~$F_i$.
\end{remark}

Recall that a surface system~$\Sigma = \Sigma_1 \cup \cdots \cup \Sigma_n$ is called a \emph{C-complex}
if each $\Sigma_i$ is a Seifert
surface, the only pairwise intersections are clasps (Figure~\ref{fig:ClaspIntersection}), and
there are no triple intersection points~\cite{Cooper82}.
The terminology clasp-word is entirely appropriate for C-complexes. For general surface systems, intersection points that belong to ribbon intersection arcs also contribute to clasp-words.

Consider a surface system $\Sigma = \Sigma_1\cup\cdots\cup \Sigma_n$. For every $i<j$, the \emph{pairwise intersection submanifold} $\Sigma_i\cap \Sigma_j$ is an oriented 1--dimensional submanifold of both~$\Sigma_i$ and $\Sigma_j$, with possibly nonempty boundary. Similarly, for $i<j<k$, the \emph{triple intersection submanifold} $\Sigma_i\cap\Sigma_j\cap\Sigma_k$ is an oriented 0--dimensional submanifold.

\begin{definition}\label{defn:homeo-surface-systems}
Call two surface systems $\Sigma = \Sigma_1\cup\cdots\cup \Sigma_n$ and $\Sigma' = \Sigma_1'\cup\cdots\cup \Sigma_n'$ \emph{homeomorphic} if there exists a homeomorphism $F \colon \Sigma\to \Sigma'$ that restricts to orientation preserving homeomorphisms $F|_{\Sigma_i}\colon \Sigma_i\to \Sigma_i'$ for $i=1,\dots,n$,
and preserves the orientations of each pairwise intersection submanifold, and each triple intersection submanifold.
\end{definition}

\begin{remark}\label{remark:clasp-orientations}
For a surface system $\Sigma=\Sigma_1\cup\cdots\cup\Sigma_n$, if there is a clasp intersection between $\Sigma_i$ and $\Sigma_j$, the \emph{sign} of the clasp is defined to be the sign of the intersection points $K_i\cap \Sigma_j$ and $K_j\cap \Sigma_i$. This sign is also determined by the orientation on the intersection arc, as follows. The clasp is positive if, for $i<j$, the arc in $\Sigma_i\cap \Sigma_j$ points from $K_i$ to K$_j$, whereas the clasp is negative if the arc points from~$K_j$ to~$K_i$.
\end{remark}

\section{Fillings of link complements}\label{section:fillings-of-link-complements}

Let $L = K_1 \cup \cdots \cup K_n$ be an $n$--component oriented, ordered link in $S^3$.
Consider its exterior~$X_L := S^3 \sm \nu L$ and recall that the first homology group~$H_1(X_L;\Z)$
is freely generated by the oriented meridians~$\mu_i$ of $K_i$, so inherits a preferred
isomorphism~$H_1(X_L;\Z) = \Z^n$. By the
identifications
\[ \begin{aligned}
\operatorname{Hom}(H_1(X_L; \Z), \Z^n) &\xrightarrow{\cong}& H^1(X_L; \Z^n)
&\xrightarrow{\cong}& [X_L ; B\Z^n]\\
\big(\mu_i \mapsto e_i \big)& \mapsto &\alpha_L &\mapsto& f_L,
\end{aligned} \]
we obtain a homotopy class of maps~$f_L \colon X_L \to B\Z^n$.
The class~$\alpha_L$ is the unique cohomology class
that evaluates to $\alpha(\mu_i) = e_i \in \Z^n$ on each meridian~$\mu_i$.
Given a surface system~$\Sigma$ for~$L$, the preimage of $\alpha_L$ in $\operatorname{Hom}(H_1(X_L; \Z), \Z^n)$ is given
geometrically by
\[ \alpha_L( \gamma) = \sum_{i=1}^n \big( \gamma \cdot \Sigma_i \big) e_i \]
with $\gamma \in H_1(X_L; \Z)$.

Following Cochran~\cite[p.~54]{Cochran90},
given a collection of closed, oriented surfaces~$F = \{F_i\}$
in a closed, oriented $3$--manifold~$M$, we construct a map~$p_F \colon M \to B\Z^n$.

\begin{const}\label{const:CollapseMap}
Let $\{F_i\}$ be a collection of closed oriented surfaces in
the closed oriented $3$--manifold~$M$.
Consider a tubular neighbourhood~$\nu F_i = F_i \times [-\pi, \pi]$
of $F_i$.
Define the map~$p_i\colon M \to S^1$ to be the composition
\[  F_i \times [-\pi,\pi] \xrightarrow{\operatorname{pr}} [-\pi,\pi] \xrightarrow{\exp} S^1 \]
in the neighbourhood~$\nu F_i$, and~$p_i(x) =1$ for all $x \notin \nu F$.
Here, $\operatorname{pr}$ denotes the projection and $\exp$ denotes $\theta\mapsto \exp( i\theta)$.
Recall that $B\Z^m$ is represented by an
$n$--torus~$T^n = S^1 \times \cdots \times S^1$.
Define the map~$p_F$ as the product
\[\begin{array}{rcl}
p_F \colon M &\to & T^n \\
x &\mapsto &  \big( p_1(x), \ldots, p_n(x)  \big).
\end{array}\]
\end{const}

Equip $T^n$ with the product CW-structure, where $S^1$ has the standard CW-decomposition with a single $0$-- and $1$--cell.
We see that $T^n$ has~$n$ $1$--cells~$S^1\langle i\rangle$,
each of which give rise to a generator of $\pi_1(T^n) \cong \Z^n$.
For each pair $1 \leq i < j \leq n$, there is a two cell~$D_{ij}$,
whose attaching map is the commutator~$[i,j] = iji^{-1}j^{-1}$
on~$S^1\langle i\rangle \vee S^1 \langle j \rangle\subset (T^n)^{(1)}$.
For each triple~$1 \leq i < j < k \leq n$, there is a single $3$--cell~$D_{ijk}$
filling the cube with sides~$D_{ij}, D_{jk}, D_{ik}$,
as illustrated in Figure~\ref{fig:construction-of-S}. Observe that in a cross section of a neighbourhood of $F_i \cap F_j$,
and away from the triple intersection points, the map~$p_F$
is described in Figure~\ref{fig:map-to-S}.

\begin{figure}
\begin{tikzpicture}[scale=0.7]
  \tikzset{->-/.style={decoration={
  markings,
  mark=at position #1 with {\arrow{>}}},postaction={decorate}}}
  \tikzset{->>-/.style={decoration={
  markings,
  mark=at position 0.48 with {\arrow{>}},
  mark=at position 0.52 with {\arrow{>}}},postaction={decorate}}}
    \tikzset{->>>-/.style={decoration={
  markings,
  mark=at position 0.46 with {\arrow{<}},
  mark=at position 0.5 with {\arrow{<}},
  mark=at position 0.54 with {\arrow{<}}},postaction={decorate}}}

\fill[red,  fill opacity=0.2] (0,2) -- (3,1) --  (7,2) -- (4, 3);
\fill[red,  fill opacity=0.2] (3/2,7/2) -- (3/2,-1/2) -- (11/2,1/2) --  (11/2,9/2);
\fill[red,  fill opacity=0.2] (2,9/2) -- (5,7/2) -- (5,-1/2) -- (2,1/2);

\draw  (3/2,3/2) -- (11/2, 5/2);
\draw (2,5/2) -- (5, 3/2);
\draw (7/2, 4) -- (7/2, 0);

\draw[thick, black, ->-=0.5] (0,0) -- (0,4);
\draw[thick, black, ->-=0.5] (3,-1) -- (3,3);
\draw[thick, black, dashed, ->-=0.45] (4,1) -- (4,5);
\draw[thick, black, ->-=0.5] (7,0) -- (7,4);
\draw[thick, black, ->>>-] (0,4) -- (4,5);
\draw[thick, black, ->>>-] (3,3) -- (7,4);
\draw[thick, black, ->>>-] (3,-1) -- (7,0);
\draw[thick, black, dashed, ->>>-] (0,0) -- (4,1);
\draw[thick, black, ->>-] (3,3) -- (0,4);
\draw[thick, black, ->>-] (7,4) -- (4,5);
\draw[thick, black, ->>-] (3,-1) -- (0,0);
\draw[thick, black, dashed, ->>-] (7,0) -- (4,1);

\node at (-1,2) {$S^1\langle i\rangle$};
\node at (0.9,-1) {$S^1\langle j\rangle$};
\node at (5.8,-1) {$S^1\langle k\rangle$};

\node at (1.3,1.3) {$D_{ij}$};

\node at (5.5,1.3) {$D_{ik}$};

\node at (3.3,4.3)  {$D_{jk}$};

\end{tikzpicture}
\caption{The boundary of $D_{ijk}$ in the construction of the CW complex $S$. Opposite faces of the cube are attached to the same $2$--cell via a degree one map.}
\label{fig:construction-of-S}
\end{figure}
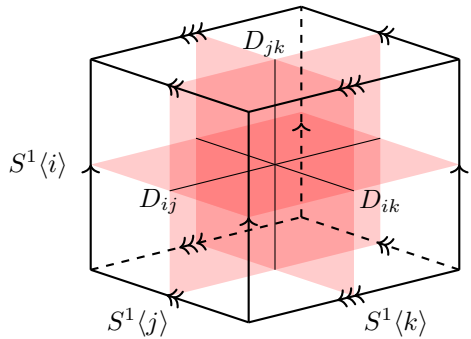

The next lemma relates the cell structure on $T^n$ to the map~$p_F$. Let $\{F_i\}$ be a collection of closed oriented surfaces in
the closed oriented $3$--manifold~$M$.
Suppose the surfaces~$F_i$ intersect transversally
in at most triple points (e.g.\ the $F_i$ are a double surface system in a double exterior).
Consequently, each triple intersection
point~$p \in F_i \cap F_j \cap F_k$
is contained in a chart~$U_p$ mapping the three surfaces
to the three coordinate hyperplanes. Denote the set of
triple points by~$P$.

\begin{figure}
\begin{tikzpicture}
  \tikzset{->-/.style={decoration={
  markings,
  mark=at position #1 with {\arrow[scale=2,>=stealth]{>}}},postaction={decorate}}}
    \tikzset{-<-/.style={decoration={
  markings,
  mark=at position #1 with {\arrow[scale=2,>=stealth]{<}}},postaction={decorate}}}
  \tikzset{->>-/.style={decoration={
  markings,
  mark=at position 0.55 with {\arrow[scale=2,>=stealth]{>}},
  mark=at position 0.7 with {\arrow[scale=2,>=stealth]{>}}},postaction={decorate}}}

\draw[ very thick, black] (-2.5, 0)
to [out=right, in=left] (2.5,0);

\draw [very  thick, black] (0,-2.5)
to [out=up, in=down] (0,2.5);

\draw [thick, blue] (0.5, -2.5)
to [out=up, in=down] (0.5, -0.5)
to [out=right, in=left] (2.5, -0.5);

\draw[thick, blue] (2.5, 0.5)
to [out=left, in=right] (0.5, 0.5)
to [out=up, in=down] (0.5, 2.5);

\draw [thick, blue] (-0.5,2.5)
to [out=down, in=up] (-0.5, 0.5)
to [out=left, in=right] (-2.5, 0.5);

\draw[thick, blue] (-2.5, -0.5)
to [out=right, in=left] (-0.5,-0.5)
to [out=down, in=up] (-0.5, -2.5);

\draw [thick, black, ->-=0.65] (-0.46, -2.5)
to [out= left, in=right] (0.46,-2.5);

\draw[thick, black, ->>-] (2.5, -0.51)
to [out=up, in=down] (2.5, 0.51);

\draw[thick, black, -<-=0.65] (0.51, 2.5)
to [out=left, in=right] (-0.51,2.5);

\draw[thick, black, ->>-] (-2.5, -0.46)
to [out=down, in=up] (-2.5, 0.46);

\draw [fill=black] (0,0) circle (3pt);

\draw[|->] [very thick, blue] (-1.15, -1.15)
to [out=-135, in=45] (-1.75, -1.75);

\node at (-1.9, -1.9) {$\ast$};

\draw[|->] [very thick, blue] (1.15, -1.15)
to [out=-45, in=135] (1.75, -1.75);

\node at (1.9, -1.9) {$\ast$};

\draw[|->] [very thick, blue] (-1.15, 1.15)
to [out=135, in=-45] (-1.75, 1.75);

\node at (-1.9, 1.9) {$\ast$};

\draw[|->] [very thick, blue] (1.15, 1.15)
to [out=45, in=-135] (1.75, 1.75);

\node at (1.9, 1.9) {$\ast$};

\node at (3.2,0) {$S^1 \langle i \rangle$};
\node at (-3.2,0) {$S^1 \langle i \rangle$};
\node at (0,3) {$S^1 \langle j \rangle$};
\node at (0,-3) {$S^1 \langle j \rangle$};

\node at (1.2,0.25) {$F_i$};
\node at (0.27,1.2) {$F_j$};
\end{tikzpicture}
\caption{A cross section of a neighbourhood of $F_i \cap F_j$.
The element~$\ast \in T^n$ indicates the unique point in $(T^n)^{(0)}$ i.e.\ the base point. Labelling around the exterior of the neighbourhood indicates the subset of $T^n$ to which that arc of the boundary is mapped.}
\label{fig:map-to-S}
\end{figure}
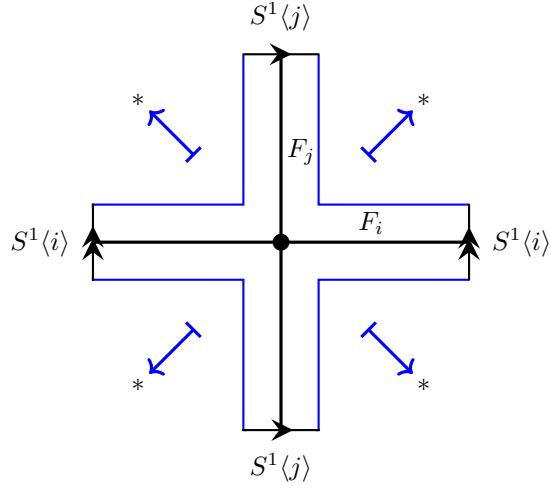

\begin{lemma}\label{lem:SmallNeighbourhood}
Pick such a chart~$U_p$ around each triple intersection point~$p$.
Then for small enough tubular neighbourhoods~$\nu F_i$,
the following statements hold:
\begin{enumerate}
\item\label{item:3Skel} $p_F \colon M \to T^n$ maps into the $3$--skeleton;
\item\label{item:2Skel} the complement~$M \setminus \bigcup_{p \in P} U_p$
is mapped to the $2$--skeleton of $T^n$; and
\item\label{item:Degree} the restriction~$p_F \colon (U_p, \partial U_p) \to (D_{ijk}, \partial D_{ijk})$
has degree the sign of the intersection point~$p \in F_i \cap F_j \cap F_k$.
\end{enumerate}
\end{lemma}

\begin{proof}
Since there are at most triple intersection points,
for small enough~$\nu F_i$ every point is contained in
at most three tubular neighbourhoods. Furthermore,
 by making the neighbourhoods even smaller,
 we can achieve that the points that are contained in three of the tubular
neighbourhoods are also contained in the interior of~$\bigcup_{p \in P} U_p$.
This shows~(\ref{item:3Skel}) and~(\ref{item:2Skel}).

Statement~(\ref{item:Degree}) can be verified in the local model
of three coordinate hyperplanes intersecting in the origin~$p = 0 \in \R^3$.
Note that the restriction of the map~$p_F$ to a cube around the origin
is illustrated in Figure~\ref{fig:construction-of-S} and it agrees with the attaching map
of the $3$--cell $D_{ijk}$.
\end{proof}

Let $L$ and $L'$ be two links.
Consider their double exterior~$M= -X_{L} \cup X_{L'}$ and
the set
\[  \Xi := \big \{ f \in [M, B\Z^n] \mid  f|_{X_{L} } = f_{L } \text{ and } f|_{X_{L'} } = f_{L'}\big\}. \]

\begin{remark}
Before computing the set~$\Xi$, we remark that, without making a choice,
one cannot simply glue the canonical homotopy
classes~$f_L$ and $f_{L'}$ together to form~$f = f_L \cup f_{L'}$.
The result of gluing, even as a homotopy class, depends on the choice
of representatives of $f_L$ and $f_{L'}$.
\end{remark}

\begin{lemma}\label{lem:AffineSpace}
The set $\Xi$ is a nonempty affine space over the abelian group~$\wt{H}^0(L \times S^1; \Z^n)$.
\end{lemma}

\begin{proof}
Using the correspondence~$[M, B\Z^n] = H^1(M;\Z^n)$, we place $[M,B\Z^n]$ in the Mayer-Vietoris exact sequence with $\Z^n$ coefficients
\[ 0 \to \wt{H}^0(L \times S^1) \to H^1(M) \xrightarrow{\operatorname{res}} H^1(X_L) \oplus H^1(X_{L'}) \to H^1(L\times S^1) .\]
Note that the set~$\Xi$ is the preimage of $f_L \oplus f_{L'}$
under the restriction map~$\operatorname{res}$.
Since the linking numbers of $L, L'$ agree, we have that $f_L - f_{L'}$ vanishes
in $H^1(L\times S^1)$. Consequently, the set~$\Xi$ is nonempty.  By exactness, $\Xi$ is then affine over~$\wt{H}^0(L \times S^1; \Z^n)$.
\end{proof}

\begin{remark}\label{rem:GeometricDescription}
The affine action of~$\wt{H}^0(L \times S^1; \Z^n)$ has a concrete description
in terms of intersection theory. It is derived from an unfaithful action of $H^0(L \times S^1; \Z)$,
which has the following description:
consider $\alpha \in \Xi$ as an element in the module~$H^1(M; \Z^n) = \Hom_{\Z}(H_1(M;\Z),\Z^n)$ and let $F \in [L \times S^1,\Z^n] \cong H^0(L \times S^1;\Z^n)$ be a map~$L \times S^1 \to \Z^n$. This associates
an element~$F(K_i \times S^1) \in \Z^n$ to each component~$K_i \times S^1$. Then we define
\[ (F  \cdot \alpha) (\gamma)  = \alpha(\gamma) + \sum_{i} \langle K_i\times S^1 , \gamma\rangle F(K_i \times S^1) \in \Z^n \]
for each $\gamma \in H_1(M; \Z)$, where $\langle K_i \times S^1 , \gamma\rangle \in \Z$ denotes the algebraic intersection number.
As $\sum_{i} \langle K_i\times S^1 , \gamma \rangle = 0$,
we have $F\cdot \alpha = \alpha$
for a (globally) constant~$F \colon L \times S^1 \to \Z^n$. As a consequence the action descends to the reduced cohomology $\wt H^0(L\times S^1; \Z)$.
\end{remark}

We can pinpoint concrete representatives of $f_L$ and $f_{L'}$
using surfaces systems, which allows us to
construct elements of~$\Xi$.

\begin{proposition}
Let $L$ and  $L'$ be two links and let $\Sigma$ and $\Sigma'$ be surface systems
with aligned clasp-words.
Let $M$ be the double exterior and let $F = - \Sigma \cup \Sigma'$
be the double surface system.
Then the map~$p_F  \colon M \to B\Z^n$
from Construction~\ref{const:CollapseMap} is an element of $\Xi$.
\end{proposition}

\begin{proof}
We have to check that $H_1(p_F)$ sends a meridian~$\mu_i$
to the $i$-th standard generator in $H_1(B\Z^n;\Z) \cong \Z^n$.
We verify that $H_1(p_F)$ sends a meridian~$\mu_i$
to the $i$-th standard generator in $H_1(B\Z^n;\Z) \cong \Z^n$
away from the double and triple points of~$F$.
This follows from the fact that
$\mu_i \cdot \Sigma_j = \delta_{ij}$.
\end{proof}

Given a surface system~$\Sigma$ for $L$, we count
the signed triple intersection points between $\Sigma_i$, $\Sigma_j$
and $\Sigma_k$, and denote the outcome by $t_{ijk}( \Sigma ) = [\Sigma_i] \cdot [ \Sigma_j]
\cdot [\Sigma_k]$.
Also recall that $\Omega_n(B)$ denotes the oriented bordism group of closed, oriented $n$--manifolds with a map to some space $B$.

\begin{theorem}\label{thm:FillableIntersection}
Let $L$ and $L'$ be two links with surface systems~$\Sigma$ and $\Sigma'$.
Suppose~$\Sigma$ and $\Sigma'$ have aligned clasp-words.
Let $M$ be the double exterior with double surface system~$F$.
Then the following two conditions are equivalent:
\begin{enumerate}[(i)]
\item $t_{ijk}(F) =  t_{ijk}(\Sigma') - t_{ijk}(\Sigma) = 0$ for all $1\leq i < j< k \leq n$;
\item\label{item:Bordism} the element~$(M,  p_F) \in \Omega_3(B\Z^n) $ vanishes.
\end{enumerate}
\end{theorem}

\begin{proof}
From the Atiyah-Hirzebruch spectral sequence with
second page~$E^2_{p,q} = H_p(B\Z^n;\Omega_q(\pt))$ and converging to $\Omega_{p+q}(B\Z^n)$, we obtain
\begin{align*}
\Omega_3(B\Z^n ) &\cong \Omega_3(\pt) \oplus H_3(B\Z^n; \Z)\\
(M, f) &\mapsto M \oplus f ([M]),
\end{align*}
where $[M]$ denotes the orientation class of $M$. The bordism group~$\Omega_3(\pt)$
is zero. As a result, Condition~(\ref{item:Bordism}) is equivalent
to $p_F([M]) = 0$.

Next, we compute $p_F ([M])$ in terms of triple intersection points.
First, consider the K\"unneth isomorphism
\[ H_3(T^n; \Z) \xrightarrow{\oplus \operatorname{pr}_{ijk}} \bigoplus_{i < j < k}
H_3(T^3_{ijk}; \Z),\]
where $\operatorname{pr}_{ijk}$ is the map on homology induced by the projection onto the
subtorus~$S^1\langle i\rangle \times S^1\langle j\rangle \times S^1\langle k\rangle$.

Now pick tubular neighbourhoods~$\nu F_i$ of the surfaces~$F_i$,
and tubular neighbourhoods~$U_p \ni p$ for each triple intersection
point~$p$ as in Lemma~\ref{lem:SmallNeighbourhood}. In particular recall that by~(\ref{item:3Skel}) of that lemma, $p_F$ factors through the 3-skeleton as $M \to (T^n)^{(3)} \subseteq T^n$.
Consider the commutative diagram of pairs
\[\begin{tikzcd}
  (M,\emptyset) \ar[r,"p_F"] \ar[d] & \Big( (T^n)^{(3)},\emptyset \Big) \ar[d] \\
  (M,M\sm \bigcup_{p \in P} \Int U_p) \ar[r] &
  \Big((T^n)^{(3)}, (T^n)^{(2)} \Big)  \\
  \bigcup_{p \in P}(U_p,\partial U_p) \ar[u] \ar[r, "\bigcup p_F \mid_{U_p}"] & \Big((T^n)^{(3)}, (T^n)^{(2)} \Big) \ar[u,equal]
\end{tikzcd},\]
where $P$ is the set of triple intersection points, the vertical maps are inclusions of pairs and the horizontal maps are induced by $p_F$.
Apply $H_3(-; \Z)$ to this diagram. By excision, the bottom left
vertical map induces an isomorphism in homology.
This gives rise to the left hand square of the commuting diagram below, in which all coefficients are the integers.
\[
\begin{tikzcd}
& H_3(T^n) \ar[dr, "\cong"] &\\
H_3(M) \ar[ur, "p_F"] \ar[r, "p_F"] \ar[d] & H_3\Big( (T^n)^{(3)}\Big) \ar[d] \ar[u] \ar[r] & \bigoplus_{i<j <k} H_3(T^3_{ijk}) \ar[d, "\cong"] \\
\bigoplus_{p \in P} H_3(U_p, \partial U_p) \ar[r] & H_3\Big( (T^n)^{(3)}, (T^n)^{(2)} \Big) \ar{r}[swap]{\operatorname{exc} }{\cong}  &
\bigoplus_{i<j <k} H_3(D_{ijk}, \partial D_{ijk})
\end{tikzcd}
\]
From the diagram, deduce that $p_F ([M])$ can be computed
from the map~$\bigcup p_F|_{U_p}$ as follows:
\begin{align*}
H_3(M; \Z) &\to H_3(T^n;\Z) \cong \bigoplus_{i<j<k} H_3(D_{ijk}, \partial D_{ijk}; \Z)\\
[M] &\mapsto \bigoplus_{i<j<k} \sum_{p \in P_{ijk}} p_F|_{U_p}  ([U_p]),
\end{align*}
where $P_{ijk} \subset P$ is the set of triple intersection points between
$F_i$, $F_j$, and $F_k$. By Lemma~\ref{lem:SmallNeighbourhood}~(\ref{item:Degree}),
$p_F|_{U_p} ([U_p]) = \sign p \cdot[D_{ijk}]$, where $\sign p$ is the
sign of the intersection point. Now this implies that
\[ p_F ([M]) = \bigoplus_{i<j<k} t_{ijk}(F) ([D_{ijk}]) \in \bigoplus_{i<j<k} H_3(D_{ijk}, \partial D_{ijk}; \Z) \cong H_3(T^n; \Z).\]
\end{proof}

\section{Sweeping}\label{section:sweeping}

The goal of this section is to prove the implication (\ref{item:main-thm-2})$\implies$(\ref{item:main-thm-1}) from Theorem~\ref{thm:Main-intro}.
First, we show how to replace an arbitrary (relative) bordism
over $B\Z^n$ between
two link exteriors with one that is constructed exclusively from
$2$--handles. This construction was used in~\cite[Proof~of~Theorem~4.2]{Cochran01},
and we include a detailed argument for the convenience of the reader.

For an integer $0\leq k\leq 4$, a $4$--dimensional \emph{elementary $k$--bordism} is a $4$--manifold
\[Y\cong (X\times [0,1]) \cup_{S^{k-1} \times D^{4-k}} D^{k} \times D^{4-k},\]
where $X$ is a 3--manifold and $S^{k-1} \times D^{4-k} \subset X \sm \partial X \times \{1\}$ is an attaching region for
a $k$--handle. By convention, $S^{-1}:=\emptyset$.

\begin{lemma}\label{lem:ElementaryReplacement}
Let $Y$ be an elementary $1$--bordism equipped with a map $F \colon Y\to B\Z^n$. Write~$\partial (Y, F) = - (X_0, f_0) \sqcup (X_{1}, f_{1})$ for the boundary over $B\Z^n$. Suppose that $H_1(f_0) \colon H_1(X_0; \Z) \to H_1(B\Z^n; \Z)$ is an epimorphism.
Then there exists an elementary $2$--bordism~$(Z,h)$ over $B\Z^n$ with the same boundary $\partial (Z,h) = - (X_0, f_0) \sqcup (X_{1}, f_{1})$.
\end{lemma}

\begin{proof}
We will find a curve in $X_1$ so that attaching a $2$--handle to $X_1$ along this curve cancels the $1$--handle attachment in the elementary $1$--bordism $Y$. But care must be taken that the map to $B \Z^n$ extends over this $2$--handle attachment.

Using that $X_1 \cong X_0 \# \left( S^1 \times S^2\right)$, consider the image of $S^1\times \pt$ in $H_1(B \Z^n;\Z)$ under $H_1(f_1)$. By assumption, $H_1(f_0) \colon H_1(X_0; \Z) \to H_1(B\Z^n; \Z)$ is surjective,
and so we take a curve $\gamma'\subset X_0$ such that $H_1(f_0)([\gamma'])=H_1(f_1)([S^1\times\pt]) \in H_1(B\Z^n;\Z) \cong \Z^n$. Use the curve $\gamma'$ to modify $S^1\times\pt\subset X_1$, and define a curve $\gamma\subset X_1$ such that $H_1(f_1)([\gamma])=0$.

Attach a $2$--handle along $\gamma$, with any framing. This cancels the~$1$--handle.
The associated elementary $2$--bordism~$\overline{Z}$ goes
from $X_1$ back to $X_0$. As $H_1(f_1)([\gamma]) = 0$,
we can extend the map~$f_1$ over $\overline Z$.
We write~$h \colon \overline Z \to B\Z^n$ for some choice of an extension and write $f'_i$
for its restriction to $X_0$.

We claim that $f'_{0}$ is homotopic to $f_0$. This can be seen by stacking
the bordisms~$Y$ and $Z'$ together along~$X_1$, and observing that
\[ Y \cup \overline Z \cong X_0 \times I,\]
which gives an homotopy from $f_0$ to $f_0'$. Modify the map on $\overline Z$
in a collar of $X_0$, to arrange that $f'_0 = f_0$.
Turn the bordism~$\overline Z$ upside-down to yield the required bordism~$Z$.
\end{proof}

Let $W$ be a bordism rel.\ boundary over $B\Z^n$ from $X_L$ to $X_{L'}$, so that
\[ \partial W = - X_L \cup \Big(\Big(\coprod_n S^1\times S^1\Big)\times[0,1]\Big)\cup X_{L'}.\]
We refer to the collection of thickened
tori~$\left(\coprod_n S^1\times S^1\right)\times[0,1]$ as the \emph{vertical boundary}.
Note that $\partial W$ decomposed in this way is homeomorphic to the usual boundary of the double link exterior $-X_L \cup X_{L'}$.
We will now use standard Morse theory arguments to present $W$ as a series of stacked elementary bordisms, after which we shall proceed to simplify that presentation to comprise a concatenation of elementary~$2$--bordisms.

By throwing away closed components, we can and will assume that $W$ is connected. Pick a Morse function~$g \colon W \to [0,1]$ such that
$g^{-1}(0)=X_L$, $g^{-1}(1)=X_{L'}$
and $g$ is the projection onto $[0,1]$ on the vertical
boundary. This implies that all critical points are in the interior of~$W$.

By cancelling the critical points of index $0$ and $4$, we may assume that $g$ has critical points of index $1$, $2$ and $3$ only~\cite[Theorem~8.1]{Milnor-h-cob}. Write $m$ for the total number of critical points of $g$, and rearrange them into increasing order~\cite[Theorem~4.8]{Milnor-h-cob}. Set $y_0 := 0$ and $y_m := 1$.  There now exist regular values~$y_1, \ldots, y_{m-1}\in[0,1]$ of~$g$ such that each interval $[y_i,y_{i+1}]$, for $i=0,\dots,m-1$, contains exactly one critical point, with index $k_i$, say, where
\[ k_i = \begin{cases}
1 & i=0, \ldots, a\\
2 & i = a + 1, \ldots, b\\
3 & i = b+1, \ldots, m-1,
\end{cases} \]
for some integers $a$ and $b$ with $-1\leq a \leq b \leq m-1$.

For $i=0,\dots,m-1$, define a collection of submanifolds~$W_i = g^{-1}([y_{i}, y_{i+1}])\subset W$, and write $X_i = g^{-1}(y_i)$. Then $\partial W_i = -X_i \sqcup X_{i+1}$, the index of the critical point in $W_i$ is $k_i$, and we have $X_0 = X_L$ and $X_m = X_{L'}$. We have presented $W$ as a series of stacked elementary $k_i$--bordisms. We will now use the map from $W$ to $B\Z^n$ to simplify the presentation using Lemma~\ref{lem:ElementaryReplacement}.

\begin{lemma}\label{lemma-W-only-2-handles}
Let~$W$ be a bordism rel.\ boundary from $X_L$ to $X_{L'}$ over $B\Z^n$.
There exists another bordism~$\wh{W}$ over $B\Z^n$ between these link exteriors obtained by stacking elementary~$2$--bordisms.
\end{lemma}

\begin{proof}
As described above, use a Morse function to decompose~$W$ into elementary
$k_i$--bordisms~$W_i$, so that
\[ W = W_0 \cup_{X_1} W_2 \cup_{X_2} \cdots \cup_{X_{m-1}} W_{m-1}. \]
Note that the preferred map~$f_0 = f_L \colon X_L \to B\Z^n$ induces a surjection
on first integral homology. From this we see that all $H_1(f_i;\Z)$ are surjections for $0 \leq i \leq a$.
Construct a new bordism~$W'$ by using
Lemma~\ref{lem:ElementaryReplacement} to replace $W_i$, for $0 \leq i \leq a$,
 by an elementary $2$--bordism~$Z_i$. We therefore have
\[ W' = Z_0 \cup_{X_1} Z_2 \cup_{X_2} \cdots \cup_{X_{a}}Z_{a}\cup_{X_{a+1}} W_{a+1} \cup\cdots \cup_{X_{m-1}} W_{m-1}.\]

Now we perform the same procedure from the other side.
For $b < i \leq m-1$, consider the elementary $3$--bordisms~$W_i$ as reversed elementary~$1$--bordisms~$\overline W_i$.
As above, substitute these with elementary~$2$--bordisms~$Z_i$ using
Lemma~\ref{lem:ElementaryReplacement}.
This results in the bordism~$\wh W$ from $X_L$ to $X_{L'}$, constructed by stacking $m$ elementary~$2$--bordisms
\[ \wh W \cong Z_0 \cup_{X_1} \cdots \cup_{X_{a}}Z_{a}\cup_{X_{a+1}} W_{a+1} \cup \cdots \cup_{X_{b}}W_{b} \cup_{X_{b+1}} \overline Z_{b+1}
\cup \cdots \cup \overline Z_{m-1}, \]
as desired.

\end{proof}

By Lemma~\ref{lemma-W-only-2-handles}, we may now assume that $W$ has only $2$--handles.  Note that none of the attaching circles for these $2$--handles links a component of $L$, for if an attaching circle were to link nontrivially with any component, then the resulting handle addition would kill an element of $H_1(X_L;\Z)$, and, since $H_1(f_L)\colon H_1(X_L;\Z)\to H_1(B\Z^n;\Z)$ is an isomorphism, $f_L$ would not extend over $W$.  Thus we see the following corollary.

\begin{corollary}\label{cor:surgery-along-curves}
Let $L, L'$ be two links with the same linking numbers.
Denote the double exterior by $M$ and suppose there is an $f \in \Xi$
such that $(M,f)=0 \in \Omega_3(B\Z^n)$.
Then $L$ can obtained from $L'$ by surgery on $S^3$ along curves~$\gamma_i$
that do not link~$L$, i.e.\ $\lk(\gamma_i, K_j) = 0$ for all $i$ and $j$.
\end{corollary}

Next we sweep a surface system through such a bordism $W$, in order to relate surface systems via surgery, as in~\cite[Section 3.2]{Davis16}.  The next theorem proves the implication (\ref{item:main-thm-2})$\implies$(\ref{item:main-thm-1}) of Theorem~\ref{thm:Main-intro}.

\begin{theorem}\label{thm:2->1}
Let $L, L'$ be two links with the same linking numbers.
Denote the double exterior by $M$, and suppose there is an $f \in \Xi \subseteq [M,B\Z^n]$
such that $(M,f)=0 \in \Omega_3(B\Z^n)$.
Then $L$ and $L'$ admit homeomorphic surface systems.
\end{theorem}

\begin{proof}
By Corollary \ref{cor:surgery-along-curves}, we have that $L'$ may be obtained
from $L$ by surgery along curves that have trivial linking
number with $L$. By~\cite[proof~of~Theorem~2~(3)~implies~(2)]{Davis16}, we can and will pick a surface system for $L$ that is disjoint from the collection of surgery curves.
After the surgery this becomes a surface system for $L'$ that is homeomorphic
to the former surface system for $L$, since we have only changed the ambient space.

\end{proof}
\begin{remark}\label{rem:CSweep}
Instead of a surface system, one can also arrange for
a C-complex to be disjoint from the surgery curves.
The C-complex can then be swept through the~$B\Z^n$--bordism
to produce a C-complex for $L'$. As a consequence,
under the hypothesis of Theorem~\ref{thm:2->1}, the links~$L$
and $L'$ also admit homeomorphic C-complexes.
\end{remark}

\section{Milnor numbers}\label{section:milnor-numbers}

For a link~$L$ with non-vanishing linking numbers, Milnor's triple linking numbers~\cite{Milnor:1957-1} are not well-defined integers. Mellor and Melvin gave a geometric interpretation of the triple Milnor numbers~\cite[Theorem p.~561]{Mellor03} that we will use to overcome the ambiguity.
In this section we refine the triple Milnor numbers in the case that the link components
have non-vanishing linking numbers.

Let us now recall the Mellor-Melvin~\cite{Mellor03} method to compute the triple Milnor numbers from a
surface system.  Let $L$ be a link and fix a surface system~$\Sigma$ for the link~$L$.
Recall that the \emph{triple intersection number}~$t_{ijk}(\Sigma) \in \Z$
is the number of intersection points counted with sign between the components~$\Sigma_i$,$\Sigma_j$,
and $\Sigma_k$. It is skew-symmetric (alternating) under permutations of the indices.

Denote the clasp-word of the component~$K$ by $w_K$.
Additionally, fix a point~$b_K \in K \sm I_K$ away from the intersections for
each component~$K$.
Reading off the cyclic words~$w_K$, starting from $b_K$ and in the positive direction,
we obtain a (linear) word~$\wt w_K$.

\begin{definition}\label{defn:SignedOccur}
Let $S$ be a finite set.
Let $w = s_1^{\varepsilon_1} \cdot \ldots \cdot s_m^{\varepsilon_m}$
be a word in the letters~$s_i^{\varepsilon_i} \in S \times \{\pm 1\}$
and let $r,s \in S$.
An \emph{$rs$--decomposition}~$(i,j)$ of $w$ is a pair
of indices with $i < j$ such that $s_i = r^{\pm 1}$ and $s_j = s^{\pm 1}$.
The \emph{sign} of a decomposition is
$\sign_w (i,j) = \varepsilon_i \cdot \varepsilon_j \in \{ \pm 1\}$.
Denote the set of $rs$--decompositions by
$D_{rs} (w) = \{ (i,j) \text{ is an $rs$--decomposition} \}$.
The \emph{signed occurrence}~$e_{rs}$ of the pair~$r,s$  is the integer
\[ e_{rs}(w) = \sum_{(i,j) \in D_{rs}} \sign_w (i,j). \]
\end{definition}

Let $e_r$ count the signed occurrences of the letter~$r$ in a word.
The following relations are helpful for computations~\cite[p.~559]{Mellor03}:
\begin{equation}\label{eqn:e-relations}
\begin{split}
e_r(u \cdot v) &= e_r(u) + e_r(v)\\
e_{rs}(u \cdot v) &=e_{rs}(u) + e_{rs}(v) +e_r(u) e_s(v) \\
e_{rs}(u) + e_{sr}(u) &= e_r(u) e_s(u),
\end{split}
\end{equation}
where $u,v$ are words, and $r,s \in S$.

\begin{remark}\label{rem:MagnusExp}
A word~$w$ in the letters~$S \times \{ \pm 1 \}$ can be considered as
an element~$g_w$ in the free group~$F$ over the set~$S$. The Magnus
expansion~\cite[Section~5.5]{MKS76} of $g_w$ is an element in the non-commutative power series ring~$\Z[[S]]$
obtained by mapping
\begin{align*}
s^{+1} &\mapsto 1 + X_s\\
s^{-1} &\mapsto 1 - X_s + X_s^2 - X_s^3 + \ldots.
\end{align*}
It is a short computation to see that the coefficient of $X_i X_j$ in the expansion of
$g_w$ is exactly $e_{ij}(w)$.
Also if $g_w \in F_3$ is in the third lower central series quotient of $F$,
then $e_{ij}(w) = 0$, which follows from the relations~\eqref{eqn:e-relations}.
\end{remark}

Given a surface system and a choice of base point~$b_K \in K$ for each component~$K$,
define the integer~$m_{ijk} \in \Z$ to be
\[ m_{ijk} = e_{ij}( \wt w_k) + e_{jk}( \wt w_i) + e_{ki}( \wt w_j), \]
with $1 \leq i,j,k \leq n$ distinct integers.

\begin{example}\label{example:example-link}
We provide a sample computation of the integers $m_{ijk}$.
Let $L = K_1 \cup \cdots \cup K_4$ be the link depicted in Figure~\ref{fig:example-link}.  The figure also shows base points and orientations.  We use a fairly obvious C-complex as the surface system for computing the clasp-words, and obtain:
\begin{align*}
  \wt{w}_1 &:=  2342^{-1}2\\ 
  \wt{w}_2 &:=  3411^{-1}313^{-1}\\ 
  \wt{w}_3 &:=  4122^{-1}2 \\
  \wt{w}_4 &:= 123.
\end{align*}

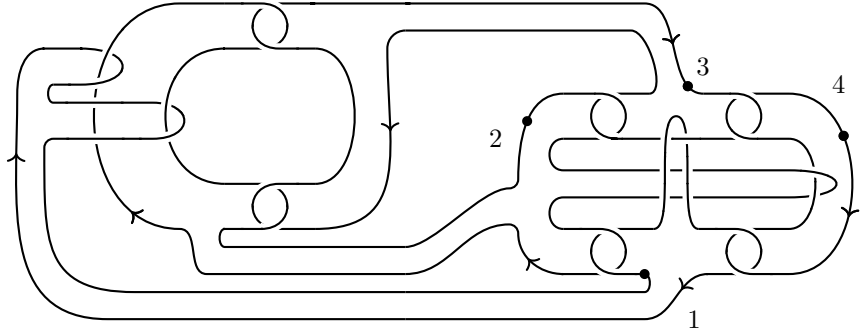
\begin{figure}
\begin{tikzpicture}[scale=1.2]
\tikzset{->-/.style={decoration={
  markings,
  mark=at position #1 with {\arrow{>}}},postaction={decorate}}}

\tikzset{-<-/.style={decoration={
  markings,
  mark=at position #1 with {\arrow{<}}},postaction={decorate}}}

\tikzset{bull/.style={decoration={
  markings,
  mark=at position #1 with {\draw[draw=black, fill=black] circle (1.5pt);}},postaction={decorate}}}

\begin{scope}[shift={(-5,0)}]

\begin{scope}[shift={(0,0)}, yscale=1,xscale=1]
\begin{knot}[
clip width=5,
]

\strand[thick, black] (0.5,0)
to [out=left, in=right] (0.1, 0)
to [out=left, in=left, looseness=2] (0.1,0.5)
to [out=right, in=left] (0.5, 0.5);

\strand[thick, black] (-0.5, 0)
to [out=left, in=right] (-0.1, 0)
to [out=right, in=right, looseness=2] (-0.1,0.5)
to [out=left, in=right] (-0.5, 0.5);

\flipcrossings{1}
\end{knot}
\end{scope}

\begin{scope}[shift={(0,2)}, yscale=1,xscale=1]
\begin{knot}[
clip width=5,
]

\strand[thick, black] (0.5,0)
to [out=left, in=right] (0.1, 0)
to [out=left, in=left, looseness=2] (0.1,0.5)
to [out=right, in=left] (0.5, 0.5);

\strand[thick, black] (-0.5, 0)
to [out=left, in=right] (-0.1, 0)
to [out=right, in=right, looseness=2] (-0.1,0.5)
to [out=left, in=right] (-0.5, 0.5);

\flipcrossings{2}
\end{knot}
\end{scope}


\draw [thick] (0.5, 0.5)
to [out=right, in=right] (0.5,2);

\draw [thick, -<-=0.6] (0.5, 0)
to [out=right, in=down, looseness=1.3] (1.3, 2)
to [out=up, in=left] (1.5, 2.2);

\draw [thick] (0.5, 2.5)
to [out=left, in=right] (1.5,2.5);

\begin{scope}[shift={(0,0)}, yscale=1,xscale=1]
\begin{knot}[
clip width=5,
]

\strand [thick, black] (-0.5,0.5)
to [out=left, in=left, looseness=1.5] (-0.5,2);

\strand[thick, ->-=0.5] (1.5, -0.5)
to [out=left, in=right] (-0.7, -0.5)
to [out=left, in=right] (-1,0)
to [out=left, in=left, looseness=1.3] (-1,2.5)
to [out=right, in=left] (-0.5,2.5);

\strand [thick, black] (-2.5, 2)
to [out=left, in=right] (-2.1, 2)
to [out=right, in=right, looseness=4] (-2.1, 1.6)
to [out=left, in=right] (-2.4, 1.6);

\strand [thick, black] (-2.4, 1.6)
to [out=left, in=left] (-2.4, 1.4)
to [out=right, in=left] (-1.3, 1.4)
to [out=right, in=right, looseness=3] (-1.3, 1)
to [out=left, in=right] (-2.4, 1)
to [out=left, in=up] (-2.5, 0.9)
to [out=down, in=left, looseness=1.4] (-1.5, -0.7)
to [out=right, in=left] (1.5,-0.7);

\flipcrossings{2,4, 5, 6}
\end{knot}

\draw [thick, black] (-0.5, 0)
to [out=left, in=left] (-0.5, -0.2)
to [out=right, in=left] (1.5, -0.2);

\draw[thick, -<-=0.2] (-2.5,2)
to [out=left, in=up] (-2.8,1.5)
to [out=down, in=left, looseness=1.15] (-1.8, -1)
to [out=right, in=left] (1.5, -1);

\end{scope}


\end{scope}


\begin{scope}[shift={(-0.5,0.25)}]
\begin{scope}[shift={(0.75,-0.75)}, yscale=1,xscale=1]
\begin{knot}[
clip width=5,
]

\strand[thick, black] (0.5,0)
to [out=left, in=right] (0.1, 0)
to [out=left, in=left, looseness=2] (0.1,0.5)
to [out=right, in=left] (0.5, 0.5);

\strand[thick, black] (-0.4, 0)
to [out=left, in=right] (-0.1, 0)
to [out=right, in=right, looseness=2] (-0.1,0.5)
to [out=left, in=right] (-0.5, 0.5);

\flipcrossings{2}
\end{knot}
\end{scope}

\begin{scope}[shift={(-0.75,-0.75)}, yscale=1,xscale=1]
\begin{knot}[
clip width=5,
]

\strand[thick, black] (0.4,0)
to [out=left, in=right] (0.1, 0)
to [out=left, in=left, looseness=2] (0.1,0.5)
to [out=right, in=left] (0.5, 0.5);

\strand[thick, black] (-0.5, 0)
to [out=left, in=right] (-0.1, 0)
to [out=right, in=right, looseness=2] (-0.1,0.5)
to [out=left, in=right] (-0.5, 0.5);

\flipcrossings{2}
\end{knot}
\end{scope}

\draw [thick,black, bull=0.0000001] (-0.35,-0.75)
to [out=right, in=right] (-0.35, -0.95)
to [out=left, in=right] (-3, -0.95);

\draw[thick, ->-=0.1] (0.35, -0.75)
to [out=left, in=right] (-0.35, -1.25)
to [out=left, in=right] (-3, -1.25);

\draw [thick, black, ->-=0.2] (-1.25,-0.75)
to [out=left, in=down] (-1.75, -0.3)
to [out=up, in=right] (-1.85, -0.2)
to [out=left, in=right] (-3, -0.75);

\draw[thick, bull=0.2] (-1.25,1.25)
to [out=left, in=up] (-1.75, 0.3)
to [out=down, in=right] (-1.85, 0.2)
to [out=left, in=right, looseness=0.6] (-3, -0.45);

\draw [thick, black, -<-=0.4, bull=0.7] (1.25,-0.75)
to [out=right, in=right, looseness=1.2] (1.25,1.25);

\begin{knot}[
clip width=5,
clip radius=5pt,
]

\strand [thick, black] (1.25, 0.75)
to [out=right, in=right] (1.25, -0.25);

\strand[thick, black] (-1.25, 0.75)
to [out=left, in=left, looseness=1.5] (-1.25, 0.4)
to [out=right, in=left] (1.25, 0.4)
to [out=right, in=right, looseness=6] (1.25, 0.1)
to [out=left, in=right] (-1.25, 0.1)
to [out=left, in=left, looseness=1.5] (-1.25, -0.25);

\strand [thick, black] (-0.25, -0.25)
to [out=right, in=left, looseness=0.5] (0,1)
to [out=right, in=left, looseness=0.5] (0.25, -0.25);

\strand [thick, black] (-1.25,0.75)
to [out=right, in=left] (-0.85, 0.75)
to [out=right, in=right, looseness=2] (-0.85, 1.25)
to [out=left, in=right] (-1.25, 1.25);

\strand[thick, black] (1.25,0.75)
to [out=left, in=right] (0.85, 0.75)
to [out=left, in=left, looseness=2] (0.85,1.25)
to [out=right, in=left] (1.25, 1.25);

\strand[thick, black] (-0.3,1.25)
to [out=left, in=right] (-0.65, 1.25)
to [out=left, in=left, looseness=2] (-0.65, 0.75)
to [out=left, in=right](0.65, 0.75)
to [out=right, in=right, looseness=2] (0.65,1.25)
to [out=left, in=right] (0.3,1.25);

\flipcrossings{1,3,4,5,6,8,10,12}
\end{knot}

\draw[thick, -<-=0.2, bull=0.05] (0.3, 1.25)
to [out=left, in=right, looseness=0.9] (-0.35, 2.25)
to [out=left, in=right] (-3, 2.25);

\draw[thick] (-0.3, 1.25)
to [out=right, in=right, looseness=0.7] (-0.5, 1.95)
to [out=left, in=right] (-3, 1.95);

\end{scope}

\node at (-0.3,-1) {$1$};
\node at (-2.5,1) {$2$};
\node at (-0.2,1.8) {$3$};
\node at (1.3,1.6) {$4$};

\end{tikzpicture}

\caption{A $4$--component link $L$ with base points and an orientation. Start with the simple $4$--component link on the right, with linking numbers all equal to $1$, and band $3$ of the components into a copy of the Borromean rings.  The left-most component of the Borromean rings is drawn so as to aid with  visualisation of a C-complex.}
\label{fig:example-link}
\end{figure}

\noindent From these clasp-words we can compute the integers $m_{ijk}$.
\begin{align*}
  m_{123} &= e_{12}( \wt w_3) + e_{23}( \wt w_1) + e_{31}( \wt w_2) = 1+1+2 = 4\\
  m_{124} &= e_{12}( \wt w_4) + e_{24}( \wt w_1) + e_{41}( \wt w_2) = 1+1+1=3\\
  m_{134} &= e_{13}( \wt w_4) + e_{34}( \wt w_1) + e_{41}( \wt w_3) = 1+1+1 =3\\
  m_{234} &= e_{23}( \wt w_4) + e_{34}( \wt w_2) + e_{42}( \wt w_3) = 1+1+1=3.
\end{align*}
This completes the sample computation.  We will use these values of the $m_{ijk}$ in Example~\ref{remark:total-quotient-is-a-massive-emphatic-awesome-refinement} below.
\end{example}

\subsection{Indeterminacy from choice of base points}\label{section:basepoints}

The integer~$m_{ijk}$ depends on the choice of linearisation of the clasp-word.
The exact dependency is described in the following lemmata. Note that
moving the base point~$b_k$ of the component~$K_k$ over an
intersection point~$j^{\pm 1} \in I_{K,j}$ changes $\wt w_k= j^{\pm 1} v$ to $v j^{\pm 1}$.

\begin{lemma}\label{lem:LkCount} Suppose $i,j,k$ are distinct.  The difference of the two decomposition sums is
\[ e_{ij}(v j^{\pm 1}) - e_{ij}(j^{\pm 1} v) = \pm \lk(K_i, K_k), \]
where all three signs are the same.
\end{lemma}

\begin{proof}
A consequence~\cite[p.~561]{Mellor03} of the second relation in~(\ref{eqn:e-relations}) above is
\begin{align*}
e_{ij}(v j^{\pm 1}) - e_{ij}( j^{\pm 1} v ) &= e_i(v) e_j(j^{\pm 1}) + e_i(j^{\pm 1})e_j(v) \\
& = \pm e_i(v)
\end{align*}
since $e_i(j^{\pm 1}) = \pm \delta_{ij}$.

The statement then follows from the fact that $e_i(v) = e_i(v j^{\pm 1})$ counts
the intersection arcs between $K_i$ and $K_k$ and so is $\lk(K_i , K_k)$.
\end{proof}

\begin{remark}
We note that the $m_{ijk}$ are not in general skew-symmetric with respect to switching two of the indices. That is, it need not be true that the integers $m_{ijk}$ and  $-m_{ikj}$ are equal.  Post hoc we will know that this skew-symmetry holds \emph{modulo the indeterminacy}.  But at this stage we do not have any skew-symmetry for the $m_{ijk}$.
\end{remark}

We will collect the integers~$m_{ijk}$ into a single vector~$m$.
Note that $m_{ijk}$ is invariant under cyclic
permutations of the indices.
Consider the alternating module~$W = \bigwedge^3  \Z\langle X^k \mid 1 \leq k \leq n \rangle$
on the free $\Z$--module generated by the $X^k$.
We abbreviate~$X^i \wedge X^j \wedge X^k$ to~$X^{[ijk]}$.
Consider the following elements
\[ m = \sum_{i<j<k} m_{ijk} X^{[ijk]} \in W, \]
Analogously, define the total triple intersection number by
\[ t = \sum_{i<j<k} t_{ijk} X^{[ijk]} \in W. \]

\begin{lemma}\label{lem:IndetermineVectors}
Let $L$ be a link with surface system~$\Sigma$.
Let $b_K$ and $b_K'$ be two collections of base points for every component of
the link.
Denote the associated elements in $W$ by $m$ and~$m'$.
Then
\[ m' - m \in \vspan  \big\{v_{s,r} \mid s \neq r \big \},\]
where
\[ v_{s,r} = \sum_{i=1}^n \lk(K_i, K_r) X^{[isr]}.\]
\end{lemma}

\begin{proof}
We may assume that for all but one component~$K = K_r$, the base points agree.
Furthermore, we can assume that $\wt w_r = s^{\pm 1} v$
and $\wt w_r' = v s^{\pm 1}$ for a suitable $s$, i.e.\ the chosen base points
are separated only by a single intersection point of $K$ with~$\Sigma_s$.
We claim that $m' - m$ is a multiple of $v_{s,r}$, whose verification
is the remainder of the proof.
Recall that $m_{ijk}$ is invariant under cyclic permutation of the indices.
In the upcoming calculation, $r$ and $s$ are fixed distinct integers $1 \leq r,s \leq n$.  Write
\[Z_{s,r} := \{i \in \{1,\dots,n\} \mid (isr) \text{ can be ordered by an even permutation}\}\]
Here we go:

\begin{align*}
m' - m &= \sum_{i<j<k} \Big( m_{ijk}' - m_{ijk} \Big) X^{[ijk]}\\
&= \sum_{i\in Z_{s,r}} \Big( m_{isr}' - m_{isr} \Big) X^{[isr]}
 +  \sum_{i \in Z_{r,s}} \Big( m_{sir}' - m_{sir} \Big) X^{[sir]}\\
&=\sum_{i\in Z_{s,r}} \Big( e_{is}( \wt w_r') - e_{is}( \wt w_r) \Big) X^{[isr]}
	+ \sum_{i\in Z_{r,s}} \Big( e_{si}( \wt w_r') - e_{si}( \wt w_r) \Big) X^{[sir]}\\
&=\sum_{i\in Z_{s,r}} \Big( e_{is}( \wt w_r') - e_{is}( \wt w_r) \Big) X^{[isr]}
	- \sum_{i\in Z_{r,s}} \Big( e_{is}( \wt w_r') - e_{is}( \wt w_r) \Big) X^{[sir]} \\
&=\sum_{i=1}^n \Big( e_{is}( \wt w_r') - e_{is}( \wt w_r) \Big) X^{[isr]}.
\end{align*}
The penultimate equality follows from the third relation of (\ref{eqn:e-relations}), since
\begin{align*}
e_{si}( \wt w_r') - e_{si}( \wt w_r) &= -e_{is}( \wt w_r') + e_i( \wt w_r')e_s( \wt w_r') + e_{is}( \wt w_r) - e_i(\wt w_r)e_s(\wt w_r) \\
 &=  -(e_{is}( \wt w_r') - e_{is}( \wt w_r))
\end{align*}
since $e_{\ell}( \wt w_r) =  e_{\ell}(\wt w_r')$ for all $\ell$.  To see the final equality, note that $X^{[isr]} = -X^{[sir]}$ and $X^{[isr]}=0$ if $i=r$ or $i=s$.
Now apply Lemma~\ref{lem:LkCount} to obtain:
\begin{align*}
m' - m = \sum_{i=1}^n \Big( \pm \lk(K_i, K_r) \Big) X^{[isr]} = \pm v_{s,r}
\end{align*}
as desired.
\end{proof}

\begin{definition}\label{defn:total-Milnor-invariant}
The \emph{total Milnor quotient}~$\Pow$ is the $\Z$--module obtained
as the quotient
\[ \Pow := W / \vspan\{ v_{s,r} \}, \]
where the elements~$v_{s,r}$ are defined in Lemma~\ref{lem:IndetermineVectors}.
For a link~$L$ with surface system~$\Sigma$, we call
the element \[\mu(L):=m(\Sigma)-t(\Sigma) \in \Pow\] constructed above the \emph{total Milnor invariant of $L$}.
\end{definition}

\begin{remark}
By Lemma~\ref{lem:IndetermineVectors}, the total Milnor invariant of $(L,\Sigma)$ is independent
of the choice of base points~$b_K$ for each component~$K$.
We show in Theorem~\ref{thm:MilnorWellDef} below that it is also independent of
the choice of surface system~$\Sigma$.  This justifies the absence of the surface system $\Sigma$ from the notation $\mu(L)$ and from the nomenclature.
\end{remark}

Denote the greatest common divisor of the linking numbers involving components~$i$,~$j$ and~$k$ by
\[ \Delta_{ijk} := \gcd \lk(K_i, K_j), \lk(K_j, K_k), \lk(K_k, K_i).\]
The triple Milnor invariants~$\ol{\mu}_L(ijk) \in \Z / \Delta_{ijk}$
are recovered as the coefficient of~$X^{[ijk]}$, as was shown in~\cite[p.~561]{Mellor03}.
By considering all the Milnor invariants simultaneously, the total Milnor invariant~$\mu(L)$
refines the collection of Milnor invariants~$\ol{\mu}_L(ijk)$, considered with their individual indeterminacy. This refinement is non-trivial, as
can be seen from the following example, which proves Proposition~\ref{prop:improved} from the introduction.

\begin{example}\label{remark:total-quotient-is-a-massive-emphatic-awesome-refinement}
Consider $4$--component links $L=K_1 \cup \cdots \cup K_4$ with $\lk(K_i, K_j)=1$ for all $1 \leq i,j \leq 4$, $i \neq j$.  Then $\Delta_L(ijk)=1$ for all multi-indices $ijk$, so all the Milnor invariants $\ol{\mu}_L(ijk)$ lie in the trivial group, with their classical indeterminacy.
By computing the $v_{s,r}$ we see that $\Pow$ is the cokernel of the linear map $\Z^6 \to \Z^4$ represented by the matrix
\[A = \begin{pmatrix}
        1 & -1 & 0 & 1 & 0 & 0 \\
        0 & 1 & -1 & 0 & 0 & 1 \\
        1 & 0 & -1 & 0 & 1 & 0 \\
        0 & 0 & 0 & 1 & -1 & 1
      \end{pmatrix}.\]
Simplifying this matrix with row operations, we compute that $\Pow \cong \Z^4/A\Z^6 \cong \Z$, and indeed writing
\begin{align*}
  x_1 &:= m_{123} -t_{123} \\
  x_2 &:= m_{124} -t_{124} \\
  x_3 &:= m_{134} -t_{134} \\
  x_4 &:= m_{234} -t_{234},
\end{align*}
the map
\[\begin{array}{rcl}
\Z^4/A\Z^6 & \to & \Z \\
(x_1,x_2,x_3,x_4) &\mapsto & x_1-x_2 +x_3 -x_4
\end{array}\]
is an isomorphism.

Consider the link $L$ from Example~\ref{example:example-link}.  This link was constructed by taking a simple $4$--component link with all linking numbers equal to $1$, and banding three of the components into a Borromean rings.  We computed the $m_{ijk}$ for this link in that example.  We used a C-complex for the computation, so all $t_{ijk}=0$.  Therefore we obtain \[x_1 = m_{123}=4,\, x_2=m_{124} =3,\, x_3=m_{134}=3 \text{ and } x_4 = m_{234}=3.\]  It follows that \[x_1-x_2+x_3-x_4 = 4-3+3-3 = 1,\] so this link has nontrivial total Milnor
invariant in $\Pow$.

Consider the link $L'$ obtained by replacing the Borromean rings on the left of Figure~\ref{fig:example-link} with an unlink before banding i.e.\ not banding at all.  Then for the link~$L'$ we have
\[\wt{w}_1 = 234,\, \wt{w}_2 = 341,\, \wt{w}_3 = 412 \text{ and }\wt{w}_4 = 123,\] from which it is straightforward to compute that $x_1=x_2=x_3=x_4 =3$.  Thus $x_1-x_2+x_3-x_4 =0$ and so the links~$L$ and~$L'$ determine distinct elements in~$\Pow$. Taking Theorem~\ref{thm:Main-intro} as given for a moment, we see that $L$ and $L'$ do not admit homeomorphic surface systems.

We can also construct a link $L_m$ by taking the Borromean rings, replacing the component labelled $1$ with its $(m,1)$ cable, and then performing the banding as in the construction of $L$.  We assert that this results in a link for which $x_1-x_2+x_3-x_4 = m$, so that all of $\Pow$ can be realised.  To see this, note that the only changes in the clasp-words from Example~\ref{example:example-link} are that~$\wt{w}_2$ becomes~$3411^{-m}31^m 3^{-1}$ and~$\wt{w}_1$ becomes $2342^{-m}2^m$. The change in $\wt{w_1}$ has no effect. We still have $e_{41}(\wt{w}_2) =1$ and $e_{34}(\wt{w}_2) =1$, but now $e_{31}(\wt{w}_2) = 1+m$, so $m_{123}=3+m$ and we still have $m_{124}=m_{134}=m_{234}=3$.  Then $x_1-x_2+x_3-x_4 =m$ as claimed.
\end{example}

\begin{remark}\label{remark:rank-of-M}
The number of relations in $\Pow$ is $n(n-1)$, while the rank of $\Z^n \wedge \Z^n \wedge \Z^n$ is $n(n-1)(n-2)/6$. Thus for $n\geq 9$, we have a presentation of $\Pow$ having more generators than relations, so $\Pow$ has nonzero rank.
In the case that every triple of indices contains a pair whose associated components have non-vanishing linking number, the rank of the classical Milnor quotient is zero, since it is a product of finite cyclic groups.  Moreover, by banding into copies of the Borromean rings, using the surface system for the Borromean rings with empty clasp-words and one triple point depicted in Figure~\ref{fig:Borromean-standard-surfaces}, we can replace the realisation construction in the example above and realise any element of $\Pow$.  We preferred to use the construction above in order to provide a nontrivial example of clasp-word computation.
It follows that $\Pow$ is always a nontrivial refinement whenever $n \geq 9$ and every triple of indices contains a pair whose associated components have non-vanishing linking number.

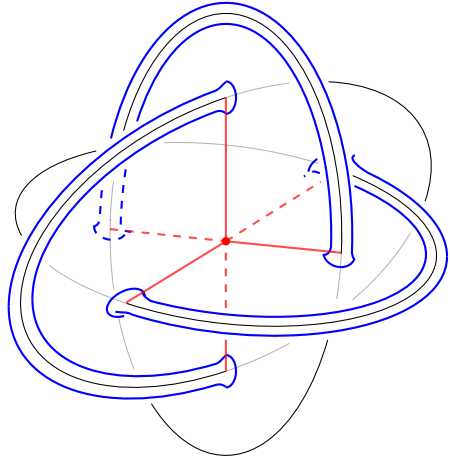
\begin{figure}
\begin{tikzpicture}[scale=0.7]


\tikzset{
    partial ellipse/.style args={#1:#2:#3}{
        insert path={+ (#1:#3) arc (#1:#2:#3)}
    }
}
 \draw[thin] (0,0) [x={(1,0)},y={(0,1)}, partial ellipse=-5:152:2.2 and 4.2];
  \draw[thin, opacity=0.3] (0,0) [x={(1,0)},y={(0,1)}, partial ellipse=166:176.5:2.2 and 4.2];
  \draw[thin, opacity=0.3] (0,0) [x={(1,0)},y={(0,1)}, partial ellipse=176:194.5:2.2 and 4.2];
  \draw[thin, opacity=0.3] (0,0) [x={(1,0)},y={(0,1)}, partial ellipse=202:217.5:2.2 and 4.2];
  \draw[thin] (0,0) [x={(1,0)},y={(0,1)}, partial ellipse=230:331.5:2.2 and 4.2];
    \draw[thin, opacity=0.3] (0,0) [x={(1,0)},y={(0,1)}, partial ellipse=343:360.5:2.2 and 4.2];
  \draw[thick, blue] (0,0) [x={(1,0)},y={(0,1)}, partial ellipse=155.5:-4:2.4 and 4.4]
   to [out=down, in=130] (2.45, -0.49);
   \draw[thick, blue, dashed] (0,0) [x={(1,0)},y={(0,1)}, partial ellipse=169:176:2.4 and 4.4]
   to [out=down, in=right] (-2.53, 0.15);
   \draw[thick, blue] (0,0) [x={(1,0)},y={(0,1)}, partial ellipse=147.5:-3:2 and 4]
    to [out=down, in=right] (1.86, -0.39);
      \draw[thick, blue, dashed] (0,0) [x={(1,0)},y={(0,1)}, partial ellipse=162:177.5:2 and 4]
      to [out=down, in=left] (-1.8, 0.05);

  \draw[thick, blue] (0,2.6) [x={(0.08,0)},y={(0,0.07)}, partial ellipse=-90:90:2.4 and 4.4];
  \draw[thick, blue] (0,-2.6) [x={(0.08,0)},y={(0,0.07)}, partial ellipse=-90:90:2.4 and 4.4];

 \draw[thin] (0,0) [x={(1,-0.1)},y={(0,1)}, partial ellipse=-118:53:4 and 1.7];
 \draw[thin, opacity=0.3] (0,0) [x={(1,-0.1)},y={(0,1)}, partial ellipse=64:107:4 and 1.7];
 \draw[thin] (0,0) [x={(1,-0.1)},y={(0,1)}, partial ellipse=128:194:4 and 1.7];
 \draw[thin, opacity=0.3] (0,0) [x={(1,-0.1)},y={(0,1)}, partial ellipse=213:243:4 and 1.7];
  \draw[thick, blue] (0,0) [x={(1,-0.1)},y={(0,1)}, partial ellipse=0:-115:4.2 and 1.88]
  to [out=150, in=-10] (-2.1,-1.68);
  \draw[thick, blue] (0,0) [x={(1,-0.1)},y={(0,1)}, partial ellipse=0:50:4.2 and 1.88]
  to [out=150, in=200] (2.45,1.75);
    \draw[thick, blue] (0,0) [x={(1,-0.1)},y={(0,1)}, partial ellipse=50:-112:3.8 and 1.52]
      to [out=150, in=200] (-1.55,-1.25);
        \draw[thick, blue, dashed] (0,0) [x={(1,-0.1)},y={(0,1)}, partial ellipse=62:63:3.8 and 1.52]
        to [out=150, in=20] (1.48,1.25);

 \draw[thick, blue, dashed] (-2.2,0.1) [x={(1,-0.1)},y={(0,1)}, partial ellipse=170:350:0.3 and 0.2];
  \draw[thick, blue] (2.15,-0.37) [x={(1,-0.1)},y={(0,1)}, partial ellipse=187:340:0.3 and 0.25];

 \draw[thin] (0,0) [x={(1.5,0.5)},y={(0,1)}, partial ellipse=90:270:2.6 and 2.6];
  \draw[thin, opacity=0.3] (0,0) [x={(1.5,0.5)},y={(0,1)}, partial ellipse=270:288:2.6 and 2.6];
    \draw[thin, opacity=0.3] (0,0) [x={(1.5,0.5)},y={(0,1)}, partial ellipse=308:334:2.6 and 2.6];
      \draw[thin] (0,0) [x={(1.5,0.5)},y={(0,1)}, partial ellipse=-14:60:2.6 and 2.6];
        \draw[thin, opacity=0.3] (0,0) [x={(1.5,0.5)},y={(0,1)}, partial ellipse=72:90:2.6 and 2.6];
  \draw[thick, blue] (0,0) [x={(1.5,0.5)},y={(0,1)}, partial ellipse=180:93:2.45 and 2.4]
  to [out=20, in=140] (0,2.3);
    \draw[thick, blue] (0,0) [x={(1.5,0.5)},y={(0,1)}, partial ellipse=180:267:2.45 and 2.4]
    to [out=20, in=220] (0,-2.3);
    \draw[thick, blue] (0,0) [x={(1.5,0.5)},y={(0,1)}, partial ellipse=180:93:2.75 and 2.8]
  to [out=20, in=220] (0,2.9);
    \draw[thick, blue] (0,0) [x={(1.5,0.5)},y={(0,1)}, partial ellipse=180:267:2.75 and 2.8]
    to [out=20, in=140] (0,-2.9);

 \draw[thick, blue] (-1.9,-1.3) [x={(1.5,0.5)},y={(0,1)}, partial ellipse=0:265:0.24 and 0.24];
  \draw[thick,blue] (1.9,1.3) [x={(1.5,0.5)},y={(0,1)}, partial ellipse=120:180:0.22 and 0.22];

\draw[thick, red, opacity=0.7] (0,-2.6) -- (0,-2);
\draw[thick, red, opacity=0.7, dashed] (0,-1.4) -- (0,-0.13);
\draw[thick, red, opacity=0.7] (0,-0.13) -- (0,2.6);

\draw[thick, red, opacity=0.7] (-1.9,-1.3) -- (0,-0.13);
\draw[thick, red, opacity=0.7, dashed] (0,-0.13) -- (1.8,1);

\draw[thick, red, opacity=0.7, dashed] (-2.2,0.1) -- (0,-0.13);
\draw[thick, red, opacity=0.7] (0,-0.13) -- (2.2,-0.35);

\draw[red, fill=red] (0,-0.13) [x={(1,0)},y={(0,1)}, partial ellipse=0:360:0.07 and 0.07];

\end{tikzpicture}
\caption{The Borromean rings together with a surface system
consisting of three genus one surfaces in the exterior of the link, that have no clasps and exactly one triple point.}
\label{fig:Borromean-standard-surfaces}
\end{figure}

\end{remark}

\subsection{Indeterminacy from choice of surface system}\label{section:indet-surface-systems}

Before answering the question on the dependency of the total Milnor invariant
on the choice of surface system, we consider a construction to modify a given
surface system. This gives further motivation for the quotient~$\Pow$.

\begin{const}[Torus sum]\label{construction:boundary sum}
Let $\Sigma$ be a surface system for a link~$L$. Let $T_K \subset X_L$
be a push-off of the boundary torus~$K \times S^1 \subset \partial X_L$, with the same orientation as $K \times S^1$.
Let $J \neq K$ be another component of $L$. As a first step of the construction,
we make $T_K$ disjoint from $\Sigma_J$. Note that $T_K \cap \Sigma_J$ consists
of push-offs of meridional circles of $K$. These intersections can be resolved by cut-and-pasting
annuli, as illustrated in Figure~\ref{fig:ResolveIntersection}.

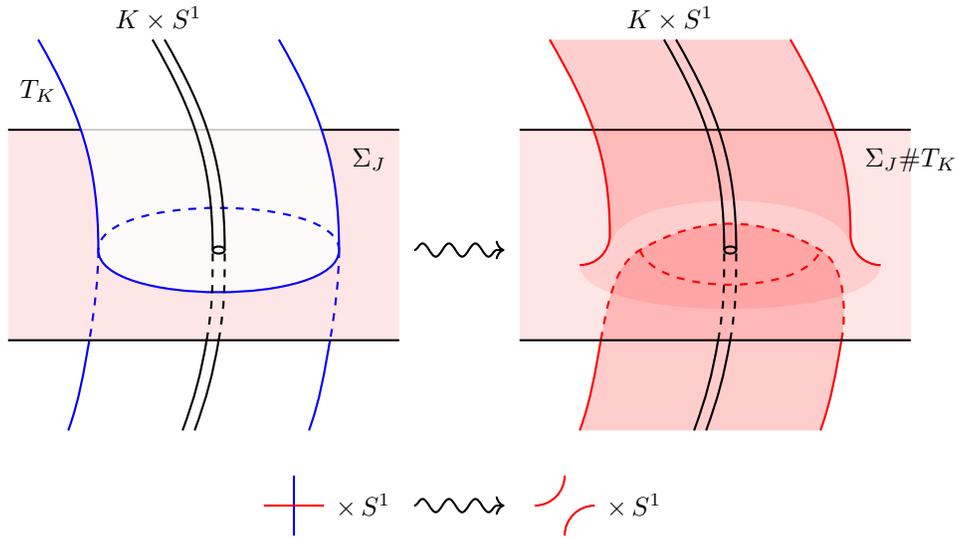
\begin{figure}
\begin{tikzpicture}[scale=0.4]


\begin{scope}[shift={(0,0)}]

\draw[thick, blue, dashed] (-4,0)
to [out=down, in=80, looseness=0.8] (-4.3,-3);
\draw[thick, blue, looseness=0.9] (-4.3,-3)
to [out=-100, in=70] (-5,-6);

\draw[thick, blue, dashed] (4,0)
to [out=down, in=80, looseness=0.8] (3.7,-3);
\draw[thick, blue, looseness=0.9] (3.7,-3)
to [out=-100, in=70] (3,-6);

\draw[thick, draw=black, draw opacity=0, fill=red, fill opacity=0.1] (6,4)
to [out=left, in=right, looseness=0.8] (-7,4)
to [out=down, in=up] (-7,-3)
to [out=right, in=left] (-4.3,-3)
to [out=right, in=left, looseness=0.3] (3.7,-3)
to [out=right, in=left, looseness=0.7] (6, -3);

\draw[thick, black] (6,4)
to [out=left, in=right, looseness=0.8] (-7,4);
\draw[thick, black] (-7,-3)
to [out=right, in=left, looseness=0.8] (6,-3);

\draw[thick, blue, fill=white, fill opacity=0.8] (-6,7)
to [out=-60, in=up] (-4,0)
to [out=down, in=down, looseness=0.6] (4,0)
to [out=up, in=-60] (2,7);

\draw[thick, blue, dashed] (4,0)
to [out=up, in=up, looseness=0.6] (-4,0);

\draw[thick, black] (-2.2,7)
to [out=-60, in=up] (-0.2,0)
to [out=down, in =down] (0.2,0)
to [out=up, in=-60] (-1.8,7);

\draw[thick, black] (-0.2, 0)
to [out=up, in = up] (0.2,0);

\draw[thick, black] (-1.2,-6)
to [out=70, in=-100] (-0.4,-3);

\draw[thick, black, dashed] (-0.4,-3)
to [out=80, in =down] (-0.2,0);

\draw[thick, black] (-0.8,-6)
to [out=70, in=-100] (0,-3);

\draw[thick, black, dashed] (0,-3)
to [out=80, in =down] (0.2,0);

\node at (-2,7.7) {$K\times S^1$};
\node at (-6,5.3) {$T_K$};
\node at (5,3) {$\Sigma_J$};

\end{scope}

\begin{scope}[shift={(17,0)}]

\draw[thick, draw=red, draw opacity=0, fill=red, fill opacity=0.2, dashed] (3,-6)
to [out=70, in=-100] (3.7,-3)
to [out=80, in=-30, looseness=0.8] (3,0)
to [out=150, in=30, looseness=1] (-3,0)
to [out=-150, in=80, looseness=0.8] (-4.3,-3)
to [out=-100, in=70] (-5,-6);

\draw[thick, red, dashed] (3.7,-3)
to [out=80, in=-30, looseness=0.8] (3,0)
to [out=150, in=30, looseness=1] (-3,0)
to [out=-150, in=80, looseness=0.8] (-4.3,-3);

\draw[thick, red, looseness=0.9] (-4.3,-3)
to [out=-100, in=70] (-5,-6);

\draw[thick, red, looseness=0.9] (3.7,-3)
to [out=-100, in=70] (3,-6);

\draw[thick, draw=black, draw opacity=0, fill=red, fill opacity=0.1] (6,4)
to [out=left, in=right, looseness=0.8] (-7,4)
to [out=down, in=up] (-7,-3)
to [out=right, in=left] (-4.3,-3)
to [out=right, in=left, looseness=0.3] (3.7,-3)
to [out=right, in=left, looseness=0.7] (6, -3);

\draw[thick, black] (6,4)
to [out=left, in=right, looseness=0.8] (-7,4);
\draw[thick, black] (-7,-3)
to [out=right, in=left, looseness=0.8] (6,-3);

\draw[thick, draw=black, draw opacity=0, fill=red, fill opacity=0.1] (-6,7)
to [out=-60, in=up](-4,0.5)
to [out=down, in=right] (-5,-0.5)
to [out=down, in=down, looseness= 0.5] (5,-0.5)
to [out=left, in=down] (4,0.5)
to [out=up, in=-60] (2,7);

\draw[thick, draw=black, draw opacity=0, fill=red, fill opacity=0.1] (-6,7)
to [out=-60, in=up](-4,0)
to [out=up, in=up, looseness= 0.7](4,0)
to [out=up, in=-60] (2,7);

\draw[thick, draw=red, draw opacity=1, fill=red, fill opacity=0] (-6,7)
to [out=-60, in=up] (-4,0.5)
to [out=down, in=right] (-5,-0.5);
\draw[thick, draw=red, draw opacity=1, fill=red, fill opacity=0] (5,-0.5)
to [out=left, in=down] (4,0.5)
to [out=up, in=-60] (2,7);

\draw[thick, black] (-2.2,7)
to [out=-60, in=up] (-0.2,0)
to [out=down, in =down] (0.2,0)
to [out=up, in=-60] (-1.8,7);

\draw[thick, black] (-0.2, 0)
to [out=up, in = up] (0.2,0);

\draw[thick, black] (-1.2,-6)
to [out=70, in=-100] (-0.4,-3);

\draw[thick, black, dashed] (-0.4,-3)
to [out=80, in =down] (-0.2,0);

\draw[thick, black] (-0.8,-6)
to [out=70, in=-100] (0,-3);

\draw[thick, black, dashed] (0,-3)
to [out=80, in =down] (0.2,0);

\draw[thick, red, dashed] (-3,0)
to [out=-70, in=-110, looseness=0.7] (3,0);

\node at (-2,7.7) {$K\times S^1$};
\node at (6,3) {$\Sigma_J\# T_K$};

\end{scope}

\begin{scope}[shift={(2.5,-8.5)}]

\draw[thick, blue] (0,-1)
to [out=up, in=down] (0,1);

\draw[thick, red] (-1, 0)
to [out=right, in =left] (1,0);

\node at (2.3,0) {$\times \,S^1$};

\draw [->] [thick, decorate, decoration={snake}] (4,0) -- (7,0);

\draw [->] [thick, decorate, decoration={snake}] (4,8.5) -- (7,8.5);

\end{scope}

\begin{scope}[shift={(11.5,-8.5)}]

\draw[thick, red] (0,-1)
to [out=up, in=left] (1,0);

\draw[thick, red] (-1, 0)
to [out=right, in =down] (0,1);

\node at (2.3,0) {$\times \,S^1$};

\end{scope}

\end{tikzpicture}
\caption{Resolving an intersection of $T_K$ and $\Sigma_J$ to create a connected surface.}
\label{fig:ResolveIntersection}
\end{figure}

We write~$\Sigma_J' = \Sigma_J \# T_K$ for the result of this operation.
Note that if $K \cap \Sigma_J = \emptyset$, then this is just the disjoint
union~$\Sigma_J \sqcup T_K$.
We call the new surface system where~$\Sigma_J$ is substituted by $\Sigma_J'$
a \emph{torus sum}~$\Sigma \#_J T_K$.
Note the construction also works with $-T_K$, which carries the opposite
of the boundary orientation. Changing the orientation of $T_K$ changes the direction of the smoothings in Figure~\ref{fig:ResolveIntersection}.
\end{const}

Note that a torus sum does not change the clasp-words, because the added surface $T_K$ is disjoint from all boundary components.
The triple intersection numbers change in the following
determined way.

\begin{lemma}\label{lem:BoundarySum}
Let $L$ be a link with surface system~$\Sigma$. Let $\Sigma \#_{s} T_r$ be
a torus sum. Then
\begin{align*}
m(\Sigma \#_s T_r) &= m(\Sigma)\\
t(\Sigma \#_s T_r) - t(\Sigma ) &=  - v_{s,r},
\end{align*}
where the vectors~$v_{s,r}$ are defined in Lemma~\ref{lem:IndetermineVectors}.
\end{lemma}

\begin{proof}
The additional triple intersection points come from intersections with $T_r$.
Consequently, the difference~$t_{rsi} (\Sigma \#_s T_r) - t_{rsi}(\Sigma )$ is given by
$\Sigma_r \cdot T_r \cdot \Sigma_i$. Each such point is contained
in exactly one intersection arc from $\Sigma_r$ to $\Sigma_i$.
From this we obtain
\[ t_{rsi} (\Sigma \#_r T_s) - t_{rsi}(\Sigma) = \lk(K_r, K_i).\]
By the skew-symmetry of triple intersection numbers, we deduce that also
\[ t_{isr} (\Sigma \#_r T_s) - t_{isr}(\Sigma ) = -\lk(K_r, K_i) = - \lk(K_i,K_r). \]
In formal sum form, this is
$ t(\Sigma \#_s T_r) - t(\Sigma ) = - v_{s,r} $.
\end{proof}

\begin{lemma}[Ordered form]\label{lem:MakeBoundaryAlike}
Let $L$ be link and a surface system~$\Sigma$.
By modifying the surface systems
in an arbitrarily small neighbourhood of
each component, and without changing $m(\Sigma)-t(\Sigma)$, we can arrange
each clasp-word to be
\[ \wt w_k = 1^{\lk(K_1, K_k)}  \cdot \ldots \cdot
n^{\lk(K_n, K_k)}.\]
\end{lemma}

\begin{proof}
Note that near $K_k$, we can pick a tubular neighbourhood~$K_k \times D^2$ of $K_k$
and assume that the other surfaces intersect the neighbourhood
in discs~$\{x\} \times D^2$ with $x \in I_{K_k}$.
In such a neighbourhood we can use the finger move~\cite[Figure~6]{Mellor03}
also depicted in our Figure~\ref{fig:finger-move}, to change the order of two adjacent intersection points.
This creates another triple intersection point that exactly
equals the change in the~$m_{ijk}$, and so $m - t$ is unaltered; cf.\ \cite[Lemma 2]{Mellor03}.

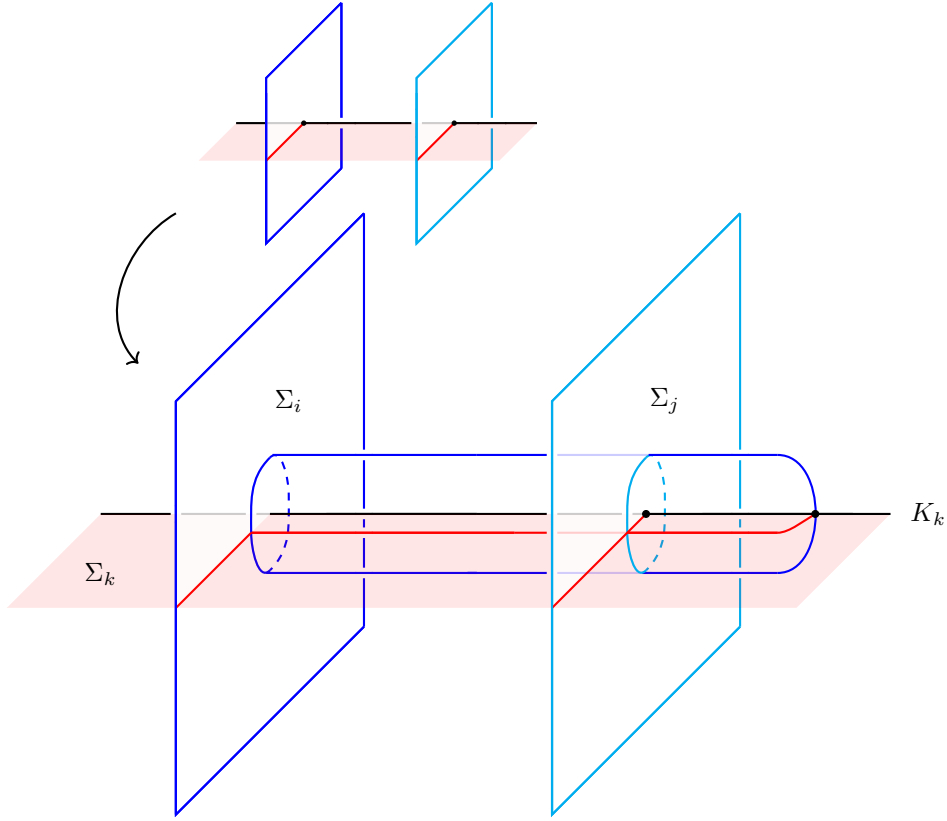
\begin{figure}
\begin{tikzpicture}


\begin{scope}[scale=0.4, shift={(-3, 13)}]

\draw[thick, draw=black, draw opacity=0, fill=red, fill opacity=0.1] (-5, 0)
to [out=right, in=left] (-2.75,0)
to [out=-135, in=45] (-4, -1.25)
to [out=left, in=right] (-6.25,-1.25)
to [out=45, in=-135] (-5,0);

\begin{knot}[
clip width=5,
clip radius=2pt,
]
\strand [thick, black] (-5, 0)
to [out=right, in=left] (-2.75,0);

\strand [thick, blue] (-1.5, 4)
to [out=-135, in=45] (-4,1.5)
to [out=down, in=up] (-4,-4)
to [out=45, in=-135] (-1.5, -1.5);

\flipcrossings{1}
\end{knot}

\draw [thick, draw=blue, draw opacity=1, fill=white, fill opacity=0.8] (-1.5, 4)
to [out=-135, in=45] (-4,1.5)
to [out=down, in=up] (-4,-4)
to [out=45, in=-135] (-1.5, -1.5)
to [out=up, in=down] (-1.5, 4);

\draw [thick, red] (-2.75, 0)
to [out=-135, in=45] (-4, -1.25);

\begin{knot}[
clip width=5,
clip radius=5pt,
]

\strand[thick, blue] (-1.5, -1.5)
to [out=up, in=down] (-1.5, 4);

\strand [thick, black] (-2.75, 0)
to [out=right, in=left] (2.25, 0);

\strand [thick, cyan] (3.5, 4)
to [out=-135, in=45] (1,1.5)
to [out=down, in=up] (1,-4)
to [out=45, in=-135] (3.5, -1.5)
to [out=up, in=down] (3.5, 4);

\flipcrossings{1,2}
\end{knot}

\draw[thick, draw=black, draw opacity=0, fill=red, fill opacity=0.1] (-2.75, 0)
to [out=right, in=left] (2.25,0)
to [out=-135, in=45] (1, -1.25)
to [out=left, in=right] (-4,-1.25)
to [out=45, in=-135] (-2.75,0);

\draw [thick, cyan, draw opacity=1, fill=white, fill opacity=0.8] (3.5, 4)
to [out=-135, in=45] (1,1.5)
to [out=down, in=up] (1,-4)
to [out=45, in=-135] (3.5, -1.5);

\draw [thick, red] (2.25, 0)
to [out=-135, in=45] (1, -1.25);

\begin{knot}[
clip width=5,
clip radius=5pt,
]
\strand[thick, cyan] (3.5, -1.5)
to [out=up, in=down] (3.5, 4);

\strand [thick, black] (2.25, 0)
to [out=right, in=left] (5, 0);

\flipcrossings{1}
\end{knot}

\draw[thick, draw=black, draw opacity=0, fill=red, fill opacity=0.1] (2.25, 0)
to [out=right, in=left] (5,0)
to [out=-135, in=45] (3.75, -1.25)
to [out=left, in=right] (1,-1.25)
to [out=45, in=-135] (2.25,0);

\draw [thick, blue] (-4, 1)
to [out=down, in=up] (-4, -4);

\draw [thick, cyan] (1, 1)
to [out=down, in=up] (1, -4);

\draw [black, fill=black] (2.25,0) circle (2pt);
\draw [black, fill=black] (-2.75,0) circle (2pt);

\end{scope}


\begin{knot}[
clip width=5,
clip radius=3pt,
]

\strand [thick, black] (-5, 0)
to [out=right, in=left] (-2, 0);

\strand [thick, blue] (-2.75, -0.75)
to [out=-135, in=down] (-3, -0.25)
to [out=up, in=-135] (-2.75, 0.75);

\strand [thick, blue] (-4, 1.5)
to [out=down, in=up] (-4, -4);

\flipcrossings{1,2}
\end{knot}

\draw[thick, draw=black, draw opacity=0, fill=red, fill opacity=0.1] (-5, 0)
to [out=right, in=left] (-2.75,0)
to [out=-135, in=45] (-4, -1.25)
to [out=left, in=right] (-6.25,-1.25)
to [out=45, in=-135] (-5,0);

\draw [thick, blue, fill=white, fill opacity=0.8] (-1.5, 4)
to [out=-135, in=45] (-4,1.5)
to [out=down, in=up] (-4,-4)
to [out=45, in=-135] (-1.5, -1.5);

\draw[thick, blue, dashed] (-2.75, 0.75)
to [out=45, in=up] (-2.5,0.25)
to [out=down, in=45] (-2.75,-0.75);

\begin{knot}[
clip width=5,
clip radius=3pt,
]

\strand [thick, black] (-2, 0)
to [out=right, in=left] (0, 0);

\strand [thick, blue] (-1.5, -1.5)
to [out=up, in= down] (-1.5, 4);

\strand [thick, blue, fill=white, fill opacity=0] (-1.5, 4)
to [out=-135, in=45] (-4,1.5)
to [out=down, in=up] (-4,-4)
to [out=45, in=-135] (-1.5, -1.5);

\strand [thick, blue] (-2.715, 0.785)
to [out=right, in=left] (0, 0.785);
\strand [thick, blue] (0,-0.785)
to [out=left, in=right] (-2.8,-0.785);

\strand [thick, red] (-3, -0.25)
to [out=right, in=left] (0.5, -0.25);

\flipcrossings{2,3, 4}
\end{knot}

\draw [thick, blue] (-2.75, -0.75)
to [out=-135, in=down] (-3, -0.25)
to [out=up, in=-135] (-2.75, 0.75);

\draw [thick, red] (-4, -1.25)
to [out=45, in=-135] (-3, -0.25);

\draw[thick, black] (-2.75,0)
to [out=right, in=left] (0,0);

\draw [thick, red] (-3, -0.25)
to [out=right, in=left] (0.5, -0.25);

\draw [thick, blue] (-4, -4)
to [out=up, in= down] (-4, 1.5);


\begin{scope}[shift={(5,0)}]
\begin{knot}[
clip width=5,
clip radius=3pt,
]

\strand [thick, black] (-5, 0)
to [out=right, in=left] (-2, 0);

\strand [thick, cyan] (-2.75, -0.75)
to [out=-135, in=down] (-3, -0.25)
to [out=up, in=-135] (-2.75, 0.75);

\strand [thick, cyan] (-4, 1.5)
to [out=down, in=up] (-4, -4);

\strand [thick, red] (-4.5, -0.25)
to [out=right, in=left] (-3, -0.25);

\strand [thick, blue] (-5, 0.785)
to [out=right, in=left] (-2.715, 0.785);
\strand [thick, blue] (-5,-0.785)
to [out=left, in=right] (-2.8,-0.785);

\flipcrossings{1,2}
\end{knot}

\draw[thick, draw=black, draw opacity=0, fill=red, fill opacity=0.1] (-7.75, 0)
to [out=right, in=left] (-2.75,0)
to [out=-135, in=45] (-4, -1.25)
to [out=left, in=right] (-9,-1.25)
to [out=45, in=-135] (-7.75,0);

\draw [thick, cyan, fill=white, fill opacity=0.8] (-1.5, 4)
to [out=-135, in=45] (-4,1.5)
to [out=down, in=up] (-4,-4)
to [out=45, in=-135] (-1.5, -1.5);

\draw[thick, cyan, dashed] (-2.75, 0.75)
to [out=45, in=up] (-2.5,0.25)
to [out=down, in=45] (-2.75,-0.75);

\begin{knot}[
clip width=5,
clip radius=3pt,
]

\strand [thick, black] (-2, 0)
to [out=right, in=left] (0, 0);

\strand [thick, cyan] (-1.5, -1.5)
to [out=up, in= down] (-1.5, 4);

\strand [thick, cyan, fill=white, fill opacity=0] (-1.5, 4)
to [out=-135, in=45] (-4,1.5)
to [out=down, in=up] (-4,-4)
to [out=45, in=-135] (-1.5, -1.5);

\strand [thick, blue] (-2.715, 0.785)
to [out=right, in=left] (-1, 0.785);
\strand [thick, blue] (-1,-0.785)
to [out=left, in=right] (-2.8,-0.785);

\strand [thick, red] (-3, -0.25)
to [out=right, in=left] (-1, -0.25);

\flipcrossings{2,3, 4}
\end{knot}

\draw[thick, draw=black, draw opacity=0, fill=red, fill opacity=0.1] (-2.75, 0)
to [out=right, in=left] (0.5,0)
to [out=-135, in=45] (-0.75, -1.25)
to [out=left, in=right] (-4,-1.25)
to [out=45, in=-135] (-2.75,0);

\draw [thick, cyan] (-2.75, -0.75)
to [out=-135, in=down] (-3, -0.25)
to [out=up, in=-135] (-2.75, 0.75);

\draw [thick, red] (-4, -1.25)
to [out=45, in=-135] (-2.75, 0);

\draw[thick, black] (-2.75,0)
to [out=right, in=left] (0.5,0);

\draw [thick, red] (-3, -0.25)
to [out=right, in=left] (-1, -0.25);

\draw [thick, cyan] (-4, -4)
to [out=up, in= down] (-4, 1.5);

\draw[thick, red] (-1,-0.25)
to [out=right, in=-150, looseness=0.7] (-0.5, 0);

\end{scope}

\draw [thick, blue] (4,0.785)
to [out=right, in=right, looseness=1.1] (4, -0.785);

\draw [black, fill=black] (4.5,0) circle (1.3pt);
\draw [black, fill=black] (2.25,0) circle (1.3pt);

\node at (-2.5,1.5) {$\Sigma_i$};
\node at (2.5,1.5) {$\Sigma_j$};
\node at (-5, -0.8) {$\Sigma_k$};
\node at (6, 0) {$K_k$};

\draw [->] [thick, black, ] (-4, 4)
to [out=-150, in=135] (-4.5, 2);

\end{tikzpicture}

\caption{The finger move switches the position of two intersection points in the clasp-word
and introduces a new triple point.}
\label{fig:finger-move}
\end{figure}

It remains to show that adjacent letters of opposite sign can be cancelled.
Suppose $x,y \in I_{K_k}$ are the corresponding intersection points of $\Sigma_j$ with $K_k$.
We can remove the intersection points by
tubing: replace the two discs~$\{x\} \times D^2$ and
$\{y\} \times D^2$ in $\Sigma_j$ with a tube around $K_k$, as shown in Figure~\ref{fig:tubing}.

\begin{figure}
\begin{tikzpicture}

\begin{scope}[scale=0.4, shift={(-3, 13)}]

\draw[thick, draw=black, draw opacity=0, fill=red, fill opacity=0.1] (-5, 0)
to [out=right, in=left] (-2.75,0)
to [out=-135, in=45] (-4, -1.25)
to [out=left, in=right] (-6.25,-1.25)
to [out=45, in=-135] (-5,0);

\begin{knot}[
clip width=5,
clip radius=2pt,
]
\strand [thick, black] (-5, 0)
to [out=right, in=left] (-2.75,0);

\strand [thick, blue] (-1.5, 4)
to [out=-135, in=45] (-4,1.5)
to [out=down, in=up] (-4,-4)
to [out=45, in=-135] (-1.5, -1.5);

\flipcrossings{1}
\end{knot}

\draw [thick, draw=blue, draw opacity=1, fill=white, fill opacity=0.8] (-1.5, 4)
to [out=-135, in=45] (-4,1.5)
to [out=down, in=up] (-4,-4)
to [out=45, in=-135] (-1.5, -1.5)
to [out=up, in=down] (-1.5, 4);

\draw [thick, red] (-2.75, 0)
to [out=-135, in=45] (-4, -1.25);

\begin{knot}[
clip width=5,
clip radius=5pt,
]

\strand[thick, blue] (-1.5, -1.5)
to [out=up, in=down] (-1.5, 4);

\strand [thick, black] (-2.75, 0)
to [out=right, in=left] (2.25, 0);

\strand [thick, blue] (3.5, 4)
to [out=-135, in=45] (1,1.5)
to [out=down, in=up] (1,-4)
to [out=45, in=-135] (3.5, -1.5)
to [out=up, in=down] (3.5, 4);

\flipcrossings{1,2}
\end{knot}

\draw[thick, draw=black, draw opacity=0, fill=red, fill opacity=0.1] (-2.75, 0)
to [out=right, in=left] (2.25,0)
to [out=-135, in=45] (1, -1.25)
to [out=left, in=right] (-4,-1.25)
to [out=45, in=-135] (-2.75,0);

\draw [thick, blue, draw opacity=1, fill=white, fill opacity=0.8] (3.5, 4)
to [out=-135, in=45] (1,1.5)
to [out=down, in=up] (1,-4)
to [out=45, in=-135] (3.5, -1.5);

\draw [thick, red] (2.25, 0)
to [out=-135, in=45] (1, -1.25);

\begin{knot}[
clip width=5,
clip radius=5pt,
]
\strand[thick, blue] (3.5, -1.5)
to [out=up, in=down] (3.5, 4);

\strand [thick, black] (2.25, 0)
to [out=right, in=left] (5, 0);

\flipcrossings{1}
\end{knot}

\draw[thick, draw=black, draw opacity=0, fill=red, fill opacity=0.1] (2.25, 0)
to [out=right, in=left] (5,0)
to [out=-135, in=45] (3.75, -1.25)
to [out=left, in=right] (1,-1.25)
to [out=45, in=-135] (2.25,0);

\draw [thick, blue] (-4, 1)
to [out=down, in=up] (-4, -4);

\draw [thick, blue] (1, 1)
to [out=down, in=up] (1, -4);

\draw [black, fill=black] (2.25,0) circle (2pt);
\draw [black, fill=black] (-2.75,0) circle (2pt);

\end{scope}

\begin{knot}[
clip width=5,
clip radius=3pt,
]

\strand [thick, black] (-5, 0)
to [out=right, in=left] (-2, 0);

\strand [thick, blue] (-2.75, -0.75)
to [out=-135, in=down] (-3, -0.25)
to [out=up, in=-135] (-2.75, 0.75);

\strand [thick, blue] (-4, 1.5)
to [out=down, in=up] (-4, -4);

\flipcrossings{1,2}
\end{knot}

\draw[thick, draw=black, draw opacity=0, fill=red, fill opacity=0.1] (-5, 0)
to [out=right, in=left] (-2.75,0)
to [out=-135, in=45] (-4, -1.25)
to [out=left, in=right] (-6.25,-1.25)
to [out=45, in=-135] (-5,0);

\draw [thick, blue, fill=white, fill opacity=0.8] (-1.5, 4)
to [out=-135, in=45] (-4,1.5)
to [out=down, in=up] (-4,-4)
to [out=45, in=-135] (-1.5, -1.5);

\draw[thick, blue, dashed] (-2.75, 0.75)
to [out=45, in=up] (-2.5,0.25)
to [out=down, in=45] (-2.75,-0.75);

\begin{knot}[
clip width=5,
clip radius=3pt,
]

\strand [thick, black] (-2, 0)
to [out=right, in=left] (0, 0);

\strand [thick, blue] (-1.5, -1.5)
to [out=up, in= down] (-1.5, 4);

\strand [thick, blue, fill=white, fill opacity=0] (-1.5, 4)
to [out=-135, in=45] (-4,1.5)
to [out=down, in=up] (-4,-4)
to [out=45, in=-135] (-1.5, -1.5);

\strand [thick, blue] (-2.715, 0.785)
to [out=right, in=left] (0, 0.785);
\strand [thick, blue] (0,-0.785)
to [out=left, in=right] (-2.8,-0.785);

\strand [thick, red] (-3, -0.25)
to [out=right, in=left] (0.5, -0.25);

\flipcrossings{2,3, 4}
\end{knot}

\draw [thick, blue] (-2.75, -0.75)
to [out=-135, in=down] (-3, -0.25)
to [out=up, in=-135] (-2.75, 0.75);

\draw [thick, red] (-4, -1.25)
to [out=45, in=-135] (-3, -0.25);

\draw[thick, black] (-2.75,0)
to [out=right, in=left] (0,0);

\draw [thick, red] (-3, -0.25)
to [out=right, in=left] (0.5, -0.25);

\draw [thick, blue] (-4, -4)
to [out=up, in= down] (-4, 1.5);


\begin{scope}[shift={(5,0)}]
\begin{knot}[
clip width=5,
clip radius=3pt,
]

\strand [thick, black] (-5, 0)
to [out=right, in=left] (-2, 0);

\strand [thick, blue] (-2.75, -0.75)
to [out=-135, in=down] (-3, -0.25)
to [out=up, in=-135] (-2.75, 0.75);

\strand [thick, blue] (-4, 1.5)
to [out=down, in=up] (-4, -4);

\strand [thick, red] (-4.5, -0.25)
to [out=right, in=left] (-3, -0.25);

\strand [thick, blue] (-5, 0.785)
to [out=right, in=left] (-2.715, 0.785);

\strand [thick, blue] (-5,-0.785)
to [out=left, in=right] (-2.95,-0.785);

\flipcrossings{1,2}
\end{knot}

\draw[thick, draw=black, draw opacity=0, fill=red, fill opacity=0.1] (-7.75, 0)
to [out=right, in=left] (-2.75,0)
to [out=-135, in=45] (-4, -1.25)
to [out=left, in=right] (-9,-1.25)
to [out=45, in=-135] (-7.75,0);

\draw [thick, blue, fill=white, fill opacity=0.8] (-1.5, 4)
to [out=-135, in=45] (-4,1.5)
to [out=down, in=up] (-4,-4)
to [out=45, in=-135] (-1.5, -1.5);

\draw[thick, blue] (-2.75, 0.75)
to [out=45, in=up] (-2.5,0.25)
to [out=down, in=45] (-2.75,-0.75);

\draw [white, fill=white] (-2.5,0) circle (3pt);

\begin{knot}[
clip width=5,
clip radius=3pt,
]

\strand [thick, black] (-2.75, 0)
to [out=right, in=left] (0, 0);

\strand [thick, blue] (-1.5, -1.5)
to [out=up, in= down] (-1.5, 4);

\strand [thick, blue, fill=white, fill opacity=0] (-1.5, 4)
to [out=-135, in=45] (-4,1.5)
to [out=down, in=up] (-4,-4)
to [out=45, in=-135] (-1.5, -1.5);


\flipcrossings{2,3, 4}
\end{knot}

\draw[thick, draw=black, draw opacity=0, fill=red, fill opacity=0.1] (-2.75, 0)
to [out=right, in=left] (0,0)
to [out=-135, in=45] (-1.25, -1.25)
to [out=left, in=right] (-4,-1.25)
to [out=45, in=-135] (-2.75,0);

\draw [thick, blue] (-2.75, -0.75)
to [out=-135, in=down] (-3, -0.25)
to [out=up, in=-135] (-2.75, 0.75);

\draw [thick, red] (-4, -1.25)
to [out=45, in=-135] (-3, -0.25);

\draw [thick, blue] (-4, -4)
to [out=up, in= down] (-4, 1.5);

\end{scope}

\node at (-2.5,1.5) {$\Sigma_i$};
\node at (2.5,1.5) {$\Sigma_i$};
\node at (-5, -0.8) {$\Sigma_j$};
\node at (5.5, 0) {$K_j$};

\draw [->] [thick, black, ] (-4, 4)
to [out=-150, in=135] (-4.5, 2);

\end{tikzpicture}\caption{Tubing together adjacent intersections of $K_j$ with the same surface~$\Sigma_i$, but
with opposite signs. The corresponding letters in the clasp-words are cancelled.}
\label{fig:tubing}
\end{figure}
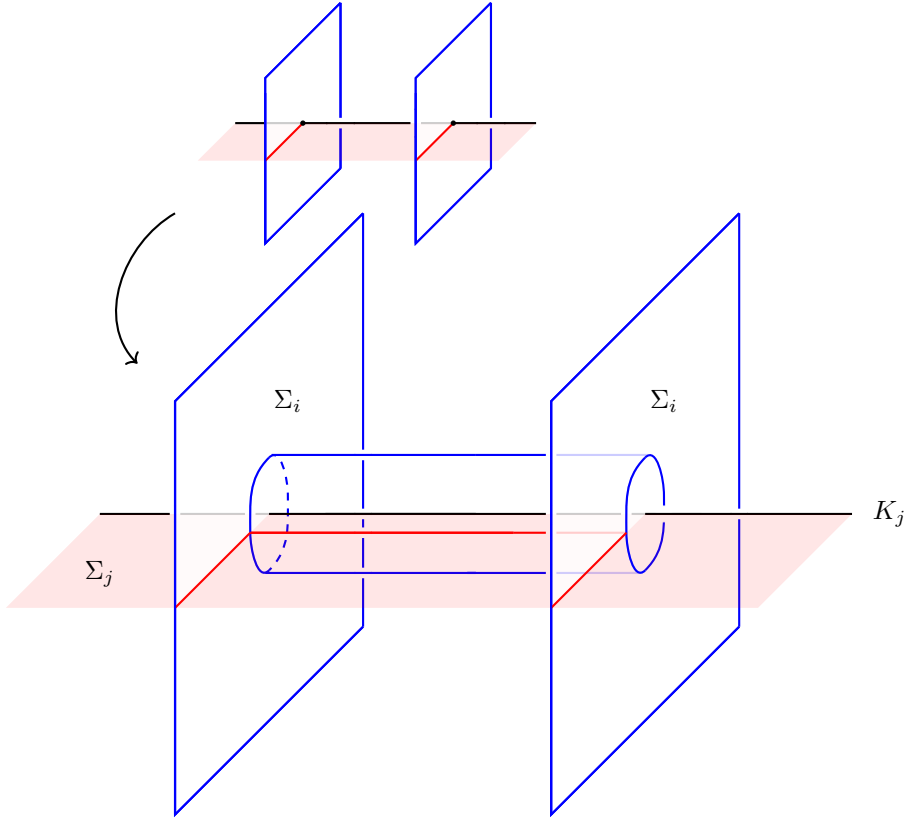

 The operation of tubing does not change $m_{ijk}$, and we
see in the local model that no additional triple intersection points
are created. Note that tubing converts two clasps into a ribbon intersection, but the contribution of the endpoints of this ribbon to the clasp-word of $L_j$ is the same as the contributions of the endpoints of the original clasps.  If one of the intersection points was already the end point of a ribbon intersection, the outcome is again a single ribbon.  If the two intersection points are the end points of a single ribbon intersection, then the outcome of tubing is a circle intersection.  In either case, the contributions to other clasp-words is unaltered.
\end{proof}

The following theorem subsumes the corresponding theorem
for $\mu_{ijk}$~\cite[p.~561]{Mellor03}. We give a new proof using
bordisms instead of disentangled surfaces.

\begin{theorem}\label{thm:MilnorWellDef}
Let $L$ be a link with two surface systems~$\Sigma$ and~$\Sigma'$.
Then the total Milnor invariants
\[ m(\Sigma) - t(\Sigma) = m(\Sigma') - t(\Sigma') \]
coincide as elements of~$\mathcal{M}$.
\end{theorem}

\begin{proof}
By Lemma~\ref{lem:MakeBoundaryAlike} above, we may
assume that the clasp-words are in ordered form.
Therefore the clasp-words agree and the base points
give a preferred alignment between the clasp-words.
Consequently, $m(\Sigma) = m(\Sigma')$.

We consider the double surface system~$F = - \Sigma \cup \Sigma'$
in the double exterior~$M = -X_L \cup X_L$.
Although $(M,  p_F) \in \Omega_3(\Z^n)$ might not be zero on the nose,
there is a $g\in \Xi \subseteq [M,B\Z^n]$ such that $(M, g)$ vanishes, because we take two copies of the exterior of
the same link~$L$.
By Lemma~\ref{lem:AffineSpace}, there exists an element $\eta \in \wt{H}^0(L \times S^1; \Z^n)$
such that $g = \eta \cdot  p_F$.

By Remark~\ref{rem:GeometricDescription}, we can translate this
into a torus sum.  That is, identifying
\[[M,B\Z^n] \cong H^1(M;\Z^n) \cong \Z^n \otimes_{\Z} H^1(M;\Z) \cong \Z^n \otimes_{\Z} H_2(M;\Z),\]
we can consider $p_F$ as an element of this latter group.

Then we can take a sequence $\{T_{k}\}_{k=1}^m$,
for $1 \leq k \leq m$, where $T_{k}$ is the boundary of a closed regular neighbourhood of some component~$K_{j_k}$, with either orientation permitted, such that
\[ g = \sum_k e_{i_k} \otimes [T_{k}]  + p_F, \]
where $e_{i_k}$ denotes the $i_k$-th standard basis element of $\Z^n$. Since the sum of all boundary
tori of $X_L$ is zero in $H_2(X_L;\Z)$, we may choose $i_k \neq j_k$ for all $k$. That is, replace $e_{i_k} \otimes [T_{k}]$ with
\[e_{i_k} \otimes [T_{k}] - \sum_{\ell=1}^n e_{i_k} \otimes [T_\ell] = \sum_{\ell\neq k} e_{i_k} \otimes [-T_\ell]\] if necessary.
We consider the associated torus
sum~$\Sigma_\mathcal{T} = \Sigma \#_{i_1} T_{1} \# \cdots \#_{i_m} T_{m}$.

By construction~$f_{\Sigma_\mathcal{T}} = g$, so $(M,g) =0 \in \Omega_3(B\Z^n)$.  By Theorem~\ref{thm:FillableIntersection}, we have
$t( \Sigma_\mathcal{T} ) - t( \Sigma' ) =0$.
  By Lemma~\ref{lem:BoundarySum}, a torus sum does
not change the total Milnor invariant.
We therefore deduce that
\[ 0 = t( \Sigma_\mathcal{T} ) - t( \Sigma' ) = t(\Sigma) - t(\Sigma') \in \mathcal{M}. \]
\end{proof}

We can now prove our second implication of Theorem~\ref{thm:Main-intro}.

\begin{proof}[Proof of (\ref{item:main-thm-1})$\implies$(\ref{item:main-thm-3})]
Suppose that $L$ and $L'$ admit homeomorphic surface systems~$\Sigma$ and~$\Sigma'$.  The linking number $\lk(K_i, K_j)$ can be computed by counting intersections between $K_i$ and $\Sigma_j$, which are preserved by the homeomorphism $\Sigma \cong \Sigma'$.  Thus, $\lk(K_i,K_j) = \lk(K_i',K_j')$.
Furthermore, given a choice of base point for $K_i$, a homeomorphism between $\Sigma$ and $\Sigma'$ produces a choice of base point for $K_i'$ with the property that the clasp-words $\wt \omega_i$ and $\wt \omega'_i$ are identical.  Since $m(\Sigma)$ depends only on the words $\wt \omega_i$,
we see that  $m(\Sigma) = m(\Sigma')$.  Moreover a homeomorphism of surface systems preserves the triple points and their signs, so $t(\Sigma)=t(\Sigma')$.
Thus, $m(\Sigma)-t(\Sigma)$
agrees with $m(\Sigma')-t(\Sigma')$. In light of Theorem~\ref{thm:MilnorWellDef}, the total Milnor invariants do not depend on the choices of surface systems, so the total Milnor invariants $\mu(L)$ and $\mu(L')$ may be computed using $\Sigma$ and $\Sigma'$ respectively, and therefore coincide.
\end{proof}

The next theorem completes the proof of the final implication of Theorem~\ref{thm:Main-intro}, namely (\ref{item:main-thm-3})$\implies$(\ref{item:main-thm-2}), thereby completing the proof of the main theorem.

\begin{theorem}\label{thm:3->2}
Let $L$ and $L'$ be two links with the same linking numbers
and agreeing total Milnor invariants. Then
there exists an element~$f \in \Xi \subseteq [M,B\Z^n]$ such that
the double exterior~$(M, f) \in \Omega_3(B\Z^n)$ bounds.
\end{theorem}

\begin{proof}
Let $\Sigma$ and $\Sigma'$ be two surface systems for $L$ and $L'$ respectively.
We are free to pick~$\Sigma$ and $\Sigma'$ to have the same clasp-words, as in Lemma~\ref{lem:MakeBoundaryAlike}. We have
\begin{align*}
m(\Sigma) &= m(\Sigma')\\
m(\Sigma) - t(\Sigma) &= m(\Sigma') - t(\Sigma').
\end{align*}
Consequently, $t(\Sigma) = t( \Sigma') \in \Pow$.
By Lemma~\ref{lem:BoundarySum},
we can take a suitable torus sum~$\Sigma_\mathcal{T}$ of $\Sigma$
such that $t_{ijk}(\Sigma_\mathcal{T}) = t_{ijk}( \Sigma') \in \Z$ agree for all $i,j,k$.
Recall from Theorem~\ref{thm:FillableIntersection} that the
associated double surface system~$F = -\Sigma_{\mathcal{T}} \cup \Sigma'$ gives rise, via Construction~\ref{const:CollapseMap}, to a map~$p_F \colon M\to B\Z^n \in \Xi$ such that the double exterior~$(M,p_F)$ is fillable.
\end{proof}

\section{Lower central series quotients}\label{section:lower-central-series-quotients}

Let $L = K_1 \cup \cdots \cup K_n$ be an $n$--component oriented, ordered link, with $n \geq 3$, and with a base point $\pt$ in the exterior $X_L$. In this section we write~$\pi(L) := \pi_1(X_L, \operatorname{pt})$ for the link group.  When the link is obvious from the context, we will sometimes just write~$\pi$ for the link group.
Recall once again that the lower central subgroups of a group $G$ are defined iteratively by $G_1 := G$ and $G_{k} := [G,G_{k-1}]$ for $k \geq 2$.

The abelianisation~$\pi(L)/\pi(L)_2$ is isomorphic to $\Z^n$, and the image of the zero-framed longitudes of $L$ determine the linking numbers.  Thus the linking numbers of two links $L$ and $L'$ are the same if and only if the lower central series quotients~$\pi(L)/\pi(L)_2$ and $\pi(L')/\pi(L')_2$ are isomorphic via an isomorphism that sends meridians to meridians and longitudes to longitudes.  It follows from the latter statement, via a well-known argument using Theorem~\ref{thm:milnor-thm-4} below, that the lower central series quotients~$\pi(L)/\pi(L)_3$ and $\pi(L')/\pi(L')_3$ are isomorphic, via an isomorphism that preserves a choice of oriented ordered meridians.  Milnor's triple linking numbers~$\ol{\mu}_L(ijk)$ can be computed from the image of the longitudes in $\pi(L)/\pi(L)_3$.  The next natural step should be:
 \begin{enumerate}[(i)]
   \item that the triple linking numbers of two links $L$ and $L'$ with the same linking numbers coincide if and only if the quotients $\pi(L)/\pi(L)_3$ and $\pi(L')/\pi(L')_3$ are isomorphic via an isomorphism that sends meridians to meridians and longitudes to longitudes;
   \item  that this implies that the quotients $\pi(L)/\pi(L)_4$ and $\pi(L')/\pi(L')_4$ are isomorphic via an isomorphism that sends meridians to meridians.
 \end{enumerate}
 In this section we prove that these indeed hold if one uses the refined triple linking numbers.  More precisely, we will prove the implications (\ref{item:main-thm-1})$\implies$(\ref{item:main-thm-4}) $\implies$(\ref{item:main-thm-3}) of Theorem~\ref{thm:Main-intro}, and Theorem~\ref{theorem:lower-central-series-intro}.

\medbreak

We begin by recalling Milnor's presentation for lower central series quotients of a link group.
Pick a \emph{basing} of the link~$L$, that is a choice of base point~$b_i \in T_i = \partial \nu K_i$
and a choice of path~$\beta_i$ in $X_L = S^3 \sm \bigcup_i \nu K_i$ from
$\operatorname{pt}$ to $b_i$.
This defines meridians~$\mu_i \in \pi$ and
zero-framed longitudes~$\lambda_i \in \pi$, based at~$\operatorname{pt}$.
We write~$F = F\langle \mu_1, \ldots, \mu_n\rangle$ for the free group on the generators~$\mu_1, \ldots, \mu_n$, which is equipped with a
map~$F \to \pi$. Since $F/ [F,F] \to \pi/ [\pi, \pi]$ is surjective, one can verify
algebraically that the composition~$F \to \pi \to \pi/\pi_k$ is an epimorphism;
see~e.g.~\cite[Rewriting Proposition~4.1]{Cochran90}.
As a consequence, write
$\lambda_i = \ell_i (\mu_1, \ldots, \mu_n)$ as a product of the chosen meridians~$\mu_i$ in the group~$\pi/ \pi_k$.
Independently of the choice of the words~$\ell_i$ made,
Milnor showed that the $k$-th lower central series quotient admits the following presentation~\cite[proof of Theorem 4]{Milnor:1957-1}.

\begin{theorem}[Milnor]\label{thm:milnor-thm-4}
Let $L$ be a link with a basing. Denote the associated meridians by~$\mu_i$,
and the zero-framed longitudes by~$\lambda_i$.
Let $\ell_i \in F$ be any word that is sent to the class of $\lambda_i$ in $\pi/ \pi_k$.
Then the lower central series
quotient~$\pi / \pi_k$ admits the following presentation
\[ \pi / \pi_k \xleftarrow{\cong} \Big \langle \mu_1, \ldots, \mu_n  \Bigm\vert [\mu_i, \ell_i],
	F_k\Big \rangle, \]
where the group~$F_k$ is the $k$-th lower central series subgroup
of the free group~$F$ on~$ \mu_1, \ldots, \mu_n$.
\end{theorem}

The statement above is slightly stronger than \cite[Theorem 4]{Milnor:1957-1}, and can be extracted from Milnor's original proof,
but is not readily obtained from the statement of Milnor's original theorem.
For the convenience of the reader, we sketch a different proof that is well-known to experts.

\begin{proof}
Isotope the paths~$\beta_i$ from the basing to be disjoint and embedded.
Pick a~$2$--disc~$D$ that contains the paths~$\beta_i$ and intersects
each link component in a single intersection point~$p_i$ with positive orientation.
The complement~$D \sm \bigcup \{ p_i\}$ has fundamental group
the free group
\[ \pi_1 \Big( D \sm \bigcup_i \{p_i \}, \operatorname{pt} \Big) \xleftarrow{\cong} F\langle \mu_1, \ldots, \mu_n\rangle \]
on the given meridians~$\mu_i$.
Now remove an open tubular neighbourhood $\nu\partial D$ of the boundary~$\partial D$ from $S^3$, with $\nu\partial D$ chosen small enough that it remains disjoint from~$L$.
The result is a framed solid torus~$V$ containing the link.
Cut along~$D' := D \cap V$, that is delete $D'$ and compactify the two ends, each with a copy of~$D'$. We obtain a solid cylinder~$D^2 \times I$
containing a collection of $n$ strands~$\{ \gamma_i\}$, such that the two endpoints of $\gamma_i$ are at $p_i\times\{0\}$ and $p_i \times \{1\}$.
This is a string link associated to $L$; see e.g.~\cite{LeDimet:1988-1}, \cite[Section~2]{HL90}.
Both the top disc and the bottom disc come with
an identification~$D^2 \times \{i\} \cong D'$.
The solid cylinder $D^2 \times I$ comes with a map to $S^3$ given by identifying $D^2 \times \{0\}$ with $D^2 \times \{1\}$ to recover $V$, and then including $V \subset S^3$.

The exterior~$R = D^2 \sm \bigcup \nu \gamma_i$
of these $n$ strands in the cylinder
is a relative bordism from $D' \sm \bigcup_i \nu p_i$ to itself.
A Mayer-Vietoris sequence argument shows this relative bordism is
a homology bordism.
Note that $R$ is equipped with two base points~$\operatorname{pt}^\pm$
from the two inclusions~$D'\subset D^2 \times \{i\}$.
Now deduce from Stallings' Theorem~\cite{Stallings65} that the inclusion
induced map
\[ F / F_k \xrightarrow{\cong}
	\pi_1(R, \operatorname{pt}^+ ) / \pi_1(R, \operatorname{pt}^+)_k \]
is an isomorphism.

Pick a path~$\tau$ on the boundary~$\partial \big( D^2 \times I \big)$,
connecting~$\operatorname{pt}^+$ with $\operatorname{pt}^-$, that maps to a meridian of $\partial D$ in $S^3 \sm \nu \partial D$ under the map $D^2 \times I \to S^3$.
Note that the longitudes~$\lambda_i$ lift to paths in $R$ from $\operatorname{pt}^-$ to $\operatorname{pt}^+$.
We turn these paths into loops based at $\operatorname{pt}^+$
by defining~$\tau_i := \tau*\lambda_i$.
Next we glue to recover the link exterior.
The link exterior in the solid torus~$V$ has fundamental group
\[ \pi_1(V \sm \nu L, \operatorname{pt})
	= \Big\langle \pi_1(R,\operatorname{pt}^+), t \Bigm\vert
		t \mu_i t^{-1} = \tau_i \mu_i \tau_i^{-1} \Big\rangle. \]
Once we fill $\nu \partial D$ back in, we get
\[ \pi_1(X_L, \operatorname{pt})
	= \Big\langle \pi_1(R,\operatorname{pt}^+) \Bigm\vert
		\mu_i = \tau_i \mu_i \tau_i^{-1}
	\Big\rangle. \]
Now calculate the lower central series quotients.
\begin{align*}
& \pi_1(X_L, \operatorname{pt}) / \pi_1(X_L, \operatorname{pt})_k\\
	\xleftarrow{\cong} & \Big\langle \pi_1(R,\operatorname{pt}^+) \Bigm\vert
		\mu_i = \tau_i \mu_i \tau_i^{-1}, i=1,\dots,n,
	 	\pi_1(R,\operatorname{pt}^+)_k\Big\rangle\\
	=& \Big\langle \pi_1(R,\operatorname{pt}^+) \Bigm\vert
		\mu_i = \tau_i \mu_i \tau_i^{-1}, \tau_i = \ell_i(\mu_1, \ldots, \mu_n), i=1,\dots,n,
	 \pi_1(R,\operatorname{pt}^+)_k\Big\rangle\\
	=& \Big\langle \pi_1(R,\operatorname{pt}^+) \Bigm\vert
		[\mu_i,\ell_i(\mu_1, \ldots, \mu_n)], i=1,\dots,n,
	 \pi_1(R,\operatorname{pt}^+)_k\Big\rangle\\
	 \xleftarrow{\cong}& \Big\langle \mu_1, \ldots, \mu_n \, \Bigm\vert [\mu_i, \ell_i], i=1,\dots,n, F_k\Big\rangle
\end{align*}
The composition of these maps sends both $\mu_i$ to the~$i$-th meridian, and $\ell_i$
to the $i$-th longitude.
\end{proof}

When do two presentations of the above form give rise
to the same group? We see that the group only depends
on the words~$\ell_i \in F$.
In fact, something stronger is true:
already the cosets~$\ell_i \in F / F_3$ in the lower central series quotient
determine the group.  This follows from Lemma~\ref{lem:CommutatorCalculus}~(\ref{comm-calc-item-3}) below, with $k=3$, $a=\mu_i$ and $b = \ell_i$.
These elements $\ell_i \in F / F_3$ can be manipulated using commutator calculus, which
offers the following relations, recorded here for later use.

\begin{lemma}\label{lem:CommutatorCalculus}
For arbitrary elements~$a,b \in F$,
the following relations hold:
\begin{enumerate}
\item\label{comm-calc-item-2} $a^g:=g^{-1}a g = a [a^{-1}, g] = a \in F/F_{k+1}$ for all~$g \in F_k$.
\item\label{comm-calc-item-3} $[a, gb] = [a,g] [a,b]^g = [a,b] \in F/F_{k+1}$ for all~$g \in F_k$.
\end{enumerate}
\end{lemma}

\begin{proof}
By definition, we have~$[a, g]\in F_{k+1} $ for all $g \in F_k$.
Equality~(\ref{comm-calc-item-2}) is obtained by expansion of the commutators,
and Equality~(\ref{comm-calc-item-3}) follows by expanding commutators once again, and then applying~(\ref{comm-calc-item-2}).
\end{proof}

Our next main goal is to compute the
words~$\ell_i \in F/F_3$ solely from
the combinatorics of a C-complex, so let us bring a C-complex
into the picture -- in this section we will not consider arbitrary surface systems, only C-complexes.
Let $L$ be an oriented ordered link together with a C-complex~$\{ \Sigma_i \}$. The orientation on each~$K_i$ and the orientation of $S^3$ induces  orientations of the double point arcs of the~C-complex.
Denote the surfaces~$\Sigma_i \cap X_L$ by $C_i$.
Each~$C_i$ has exactly one boundary component~$\lambda_{C_i}$
in~$T_i = \partial \nu K_i \subset X_L$, which is a
zero-framed push-off of~$K_i$.
The boundary component~$\lambda_{C_i}$ is called the \emph{longitudinal}
boundary, and the other boundary components are called \emph{meridional}.

Now we temporarily fix a link component~$K$ of $L$.
Pick a base point~$b_K$ on $\lambda_{C_i}$
that is disjoint from the set~$\Sigma_j$ for all $j \neq i$.
Connect~$b_K$ to the base point~$\operatorname{pt} \in X_L$ of the exterior
via a path~$\beta_K$ that is disjoint from each surface~$C_i$, and approaches~$b_K$ from the negative side.
This defines classes for the meridian~$\mu_i \in \pi$ and
a longitude~$\lambda_i = \big( \beta_i\big)_\# \lambda_{C_i} \in \pi$, where the $\big( \beta_i\big)_\#$ is the change of base point map on based loops that conjugates with the path $\beta_i$ to change the base point from $\beta_i(1)$ to~$\beta_i(0)$.
The classes $\mu_i$ and $\lambda_i$ are respectively the meridians and the longitudes associated to the basing~$\{ \beta_i \}$.

\begin{definition}\label{defn:subordination}
Let $L$ be a link with a C-complex~$\{ \Sigma_i\}$.
A basing~$\{ \beta_i \}$ of the link~$L$ as described above is said to be
\emph{subordinate} to the C-complex.
\end{definition}

We proceed by introducing further notation, that will help us with the calculation
of~$\ell_i \in F$, a word in the $\mu_i$ such that~$\ell_i(\mu_1, \ldots, \mu_n) = \lambda_i$ modulo length three commutators.
Order the intersection points
\[ I_K = \{ a_{K,1}, \ldots, a_{K, m_K} \} = \big\{ x \in \lambda_{C_K} \bigm\vert  x \in \Sigma_j
\text{ for some } K_j \neq K \big\},\]
starting from base point~$b_K$ and traversing~$\lambda_{C_K}$ in the positive direction.
Let~$r \in \{1,\dots,m_K\}$.  Denote the path from~$b_K$ to $a_{K,r}$ following $\lambda_{C_i}$ in the positive direction
by~$\alpha_{K, r}$.

Consider an intersection arc in $C_K$ corresponding to a clasp. This arc connects
the longitudinal boundary with a meridional boundary component.
Follow the intersection arc that emanates from a point~$a_{K, r} \in I_K$
and terminates at a point~$a_{J_r, s} \in I_{J(r)}$
of another link component~$J_r = K_j$, to define a path~$\iota_{K,r}$
from~$a_{K,r}$ to~$a_{J_r,s}$, where this equation defines $s$.
Let $\sigma_K \colon \{1, \ldots, m_K\} \to \{1, \ldots, n\}$ be the map
that associates~$r \mapsto s$ for each $r \in \{1,\dots,m_K\}$.

We will introduce another path~$\gamma_{K, r} \in \pi_1(C_K,b_K)$.
First note that traversing the meridional boundary
starting at $a_{J_r,s}$ defines a loop~$\mu_{K,r} \in \pi_1(C_K, a_{J_r,s})$,
Observe that $\mu_{K,r}$ is freely homotopic to a meridian
of the knot component~$J_r$.
We base~$\mu_{K,r}$ at $b_K$ by defining
\[ \gamma_{K,r} = \big( \alpha_{K,r} * \iota_{K,r} \big)_\#  \mu_{K,r} \in \pi_1(C_K,b_K). \]
We direct the reader to Figure~\ref{fig:Lollipop} for illustrations of the defined paths.

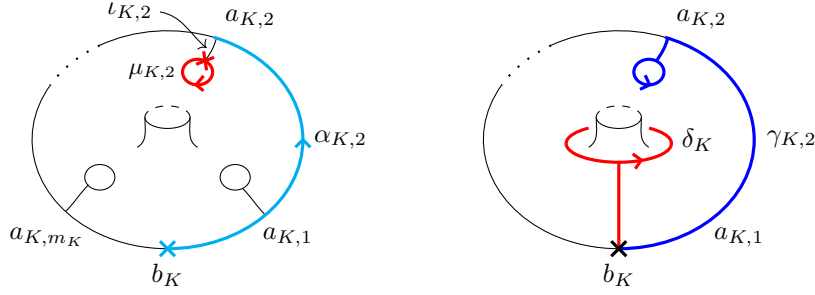
\begin{figure}
\begin{tikzpicture}

\tikzset{
    partial ellipse/.style args={#1:#2:#3}{
        insert path={+ (#1:#3) arc (#1:#2:#3)}
    }
}


\draw[thin, black] (0.5,1.05)
to [out=55, in=255] (0.65, 1.35);

\draw[thin, black] (-1.35,-0.95)
to [out=45, in=225] (-1.05, -0.6);

\draw[thin, black] (1,-0.65)
to [out=-45, in=135] (1.3, -1);

\draw[thin, black] (0,0) [partial ellipse=70:120:1.8 and 1.45];
\draw[thick, black, loosely dotted] (0,0) [partial ellipse=120:150:1.8 and 1.45];
\draw[thin, black] (0,0) [partial ellipse=150:270:1.8 and 1.45];

\draw[thin, black, dashed] (0,0.3) [partial ellipse=0:180:0.3 and 0.15];
\draw[thin, black] (0,0.3) [partial ellipse=180:360:0.3 and 0.15];
\draw[thin, black] (0.3,0.3)
to [out=down, in=150] (0.4, -0.1);
\draw[thin, black] (-0.3,0.3)
to [out=down, in=30] (-0.4, -0.1);

\draw[very thick, red] (0.4,0.9) [partial ellipse=0:360:0.2 and 0.17];

\draw[thin, black] (0.9,-0.5) [partial ellipse=0:360:0.2 and 0.17];

\draw[thin, black] (-0.9,-0.5) [partial ellipse=0:360:0.2 and 0.17];

\draw[very thick, cyan] (0,0) [partial ellipse=-90:70:1.8 and 1.45];

\begin{scope}[shift={(0, -1.45)}]
\draw[very thick, cyan] (-0.1, 0.1) -- (0.1, -0.1);
\draw[very thick, cyan] (-0.1, -0.1) -- (0.1, 0.1);
\end{scope}

\begin{scope}[shift={(1.8,0)}]
\draw[very thick, cyan] (-0.1, -0.1) -- (0, 0);
\draw[very thick, cyan] (0.1, -0.1) -- (0, 0);
\end{scope}

\begin{scope}[shift={(0.38,0.73)}]
\draw[very thick, red] (0.1, 0.12) -- (0, 0);
\draw[very thick, red] (0.1, -0.08) -- (0, 0);
\end{scope}

\begin{scope}[shift={(0.5,1.05)}, rotate=60]
\draw[very thick, red] (-0.1, 0.08) -- (0.1, -0.08);
\draw[very thick, red] (-0.1, -0.08) -- (0.1, 0.08);
\end{scope}

\node at (0, -1.8) {$b_K$};
\node at (1.6, -1.3) {$a_{K,1}$};
\node at (1.1, 1.6) {$a_{K,2}$};
\node at (-1.6, -1.3) {$a_{K,m_K}$};
\node at (-0.2, 0.9) {\small{$\mu_{K,2}$}};
\node at (2.3, 0) {$\alpha_{K,2}$};

\node at (-0.5,1.7) {$\iota_{K,2}$};
\draw[->] (-0.1,1.7)
to [out=-10, in=140] (0.52, 1.22);

\begin{scope}[shift={(6, 0)}]
\draw[very thick, blue] (0,0) [partial ellipse=-90:70:1.8 and 1.45];
\draw[thin, black] (0,0) [partial ellipse=70:120:1.8 and 1.45];
\draw[thick, black, loosely dotted] (0,0) [partial ellipse=120:150:1.8 and 1.45];
\draw[thin, black] (0,0) [partial ellipse=150:270:1.8 and 1.45];

\draw[thin, black, dashed] (0,0.3) [partial ellipse=0:180:0.3 and 0.15];
\draw[thin, black] (0,0.3) [partial ellipse=180:360:0.3 and 0.15];
\draw[thin, black] (0.3,0.3)
to [out=down, in=150] (0.4, -0.1);
\draw[thin, black] (-0.3,0.3)
to [out=down, in=30] (-0.4, -0.1);

\draw[very thick, red] (0,-0.05) [partial ellipse=125:415:0.7 and 0.25];

\draw[very thick, blue] (0.4,0.9) [partial ellipse=0:360:0.2 and 0.17];
\draw[very thick, blue] (0.5,1.05)
to [out=55, in=255] (0.65, 1.35);

\draw[very thick, red] (0, -1.47)
to [out=up, in=down] (0,-0.30);

\begin{scope}[shift={(0.41,0.73)}, rotate=-20]
\draw[very thick, blue] (-0.1, -0.1) -- (0, 0);
\draw[very thick, blue] (-0.1, 0.1) -- (0, 0);
\end{scope}

\begin{scope}[shift={(0.3,-0.27)}, rotate=10]
\draw[very thick, red] (-0.1, -0.1) -- (0, 0);
\draw[very thick, red] (-0.1, 0.1) -- (0, 0);
\end{scope}

\begin{scope}[shift={(0, -1.45)}]
\draw[very thick] (-0.1, 0.1) -- (0.1, -0.1);
\draw[very thick] (-0.1, -0.1) -- (0.1, 0.1);
\end{scope}

\node at (0, -1.8) {$b_K$};
\node at (1.05, 0) {$\delta_K$};
\node at (1.6, -1.3) {$a_{K,1}$};
\node at (1.1, 1.6) {$a_{K,2}$};
\node at (2.3, 0) {$\gamma_{K,2}$};

\end{scope}

\end{tikzpicture}
\caption{Paths in $C_K\subset \Sigma_K$.}
\label{fig:Lollipop}
\end{figure}

Pick a collar of the longitudinal boundary of the surface~$\Sigma_K$
that contains all intersection arcs and all loops $\mu_{K,r}$.
Note that the inside boundary of that collar is a separating
curve that cuts~$\Sigma_K$ into two components:
one containing all the genus, and
an annulus containing the intersection arcs and loops $\mu_{K,r}$.
Connect the inside boundary to $b_K$ by a path in the complement of the intersection arcs.
This defines a loop~$\delta_K \in \pi_1(C_K, b_K)$.  This is also illustrated on the right of Figure~\ref{fig:Lollipop}.
From now on we consider all loops as living in the fundamental group of the link exterior via the appropriate inclusion induced maps,
changing the notation neither for the loop nor its base point.

\begin{lemma}\label{lem:GenusTripleCommutator}
The loop~$\delta_K \in \pi_1(X_L,b_K)$ is a length $3$ commutator.
\end{lemma}

\begin{proof}
The loop~$\delta_K$ bounds the surface~$S$ in $C_K$ given by
the complement of a collar of the longitudinal boundary.  Since that collar contains
all of the intersection arcs, we see that $S \cap C_j = \emptyset$ for all~$j$.
This implies that all loops in $S$ are length $2$ commutators, since they are zero in
$H_1(X_L; \Z) \cong \Z^n$. Consequently, the loop~$\delta_K$,
as the boundary of $S$, is a length $3$ commutator.
\end{proof}

Consider an intersection arc $\iota_{K,r}$ in $C_K$ connecting~$a_{K,r}$
with $a_{J_r,s}$.
Write $g_{K,r}$ for the
loop~$\beta_{K} * \alpha_{K,r} * \iota_{K,r} * \alpha^{-1}_{J_r,s} * \beta_{J_r}^{-1} \in \pi = \pi_1(X_L,\pt)$.

\begin{lemma}\label{lem:MeridianRewrite}
For each $r \in \{1,\dots,m_K\}$, the loop~$\big( \beta_K \big)_\# \gamma_{K,r} \in \pi$ is a conjugate
of the meridian~$\mu_{J_r} \in \pi$, namely~$\big( \beta_K \big)_\# \gamma_{K,r} = {\big(\mu_{\sigma_K(r)} \big)}^{g_{K,r}} = {\big(\mu_{J_r} \big)}^{g_{K,r}} \in \pi$.
\end{lemma}

\begin{proof}
A meridian of $J_r$ at~$a_{J_r,s}$, which is based to $\operatorname{pt}$
via the whisker~$\beta_{J_r} * \alpha_{J_r,s}$, is homotopic to $\mu_{J_r}$.  This is via a homotopy sliding the meridian $\big( \beta_K \big)_\# \gamma_{K,r}$ along~$J_r$ to the chosen meridian of $J_r$. Recording the new basing path created during the slide tells us what we need to conjugate by.  It might help to inspect Figure~\ref{fig:MeridianExpression}.  More precisely, we have:
\begin{align*}
  {\big(\mu_{J_r} \big)}^{g_{K,r}}
  =& \beta_K * \alpha_{K,r} * \iota_{K,r} * \alpha_{J_r,s}^{-1} * \beta_{J_r}^{-1} * \mu_{J_r} * \beta_{J_r} * \alpha_{J_r,s} * \iota_{K,r}^{-1} *  \alpha_{K,r}^{-1} * \beta_K^{-1} \\
  =& \beta_K * \alpha_{K,r} * \iota_{K,r} * \mu_{K,r} * \iota_{K,r}^{-1} *  \alpha_{K,r}^{-1} * \beta_K^{-1} \\
   =& \beta_K * \gamma_{K,r} * \beta_K^{-1}.
  \end{align*}

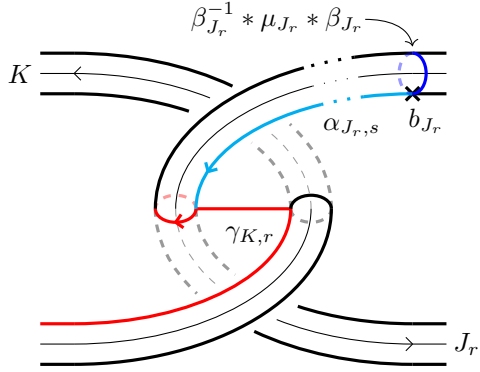
\begin{figure}
\begin{tikzpicture}[scale=0.9]


\tikzset{
    partial ellipse/.style args={#1:#2:#3}{
        insert path={+ (#1:#3) arc (#1:#2:#3)}
    }
}


\draw[very thick, black, dashed, opacity=0.4] (-2.5,0) [partial ellipse=0:35:3.2 and 1.7];
\draw[thin, black, dashed, opacity=0.4] (-2.5,0) [partial ellipse=0:34:3.5 and 2];
\draw[very thick, black, dashed, opacity=0.4] (-2.5,0) [partial ellipse=0:33:3.8 and 2.3];

\draw[very thick, black] (-2.5,0) [partial ellipse=58:90:3.2 and 1.7];
\draw[thin, black] (-2.5,0) [partial ellipse=55:90:3.5 and 2];
\draw[very thick, black] (-2.5,0) [partial ellipse=52:90:3.8 and 2.3];

\draw[very thick, red, dashed, opacity=0.4] (-1,0) [partial ellipse=0:180:0.3 and 0.2];


\draw[very thick, black, dashed, opacity=0.4] (2.5,0) [partial ellipse=180:215:3.2 and 1.7];
\draw[thin, black, dashed, opacity=0.4] (2.5,0) [partial ellipse=180:214:3.5 and 2];
\draw[very thick, black, dashed, opacity=0.4] (2.5,0) [partial ellipse=180:213:3.8 and 2.3];

\draw[very thick, black] (2.5,0) [partial ellipse=238:270:3.2 and 1.7];
\draw[thin, black] (2.5,0) [partial ellipse=235:270:3.5 and 2];
\draw[very thick, black] (2.5,0) [partial ellipse=232:270:3.8 and 2.3];

\draw[very thick, black, dashed, opacity=0.4] (1,0) [partial ellipse=180:360:0.3 and 0.2];


\begin{scope}[shift={(2.5, 1.7)}]
\draw[very thick] (0.1, 0.1) -- (-0.1,-0.1);
\draw[very thick](0.1, -0.1) -- (-0.1,0.1);
\end{scope}


\draw[thin, black] (-2.5,0) [partial ellipse=-90:0:3.5 and 2];
\draw[very thick, black] (-2.5,0) [partial ellipse=-90:0:3.8 and 2.3];
\draw[very thick, red] (-2.5,0)
to [partial ellipse=-90:0:3.2 and 1.7] (-0.7,0);
\draw[very thick, red] (-1,0) [partial ellipse=180:360:0.3 and 0.2];


\draw[thin, black] (2.5,0) [partial ellipse=90:103:3.5 and 2];
\draw[very thick, black] (2.5,0) [partial ellipse=90:103:3.8 and 2.3];
\draw[very thick, cyan] (2.5,0) [partial ellipse=90:103:3.2 and 1.7];

\draw[thin, black, loosely dotted] (2.5,0) [partial ellipse=103:115:3.5 and 2];
\draw[very thick, black, loosely dotted] (2.5,0) [partial ellipse=103:115:3.8 and 2.3];
\draw[very thick, cyan, loosely dotted] (2.5,0) [partial ellipse=103:115:3.2 and 1.7];

\draw[thin, black] (2.5,0) [partial ellipse=115:180:3.5 and 2];
\draw[very thick, black] (2.5,0) [partial ellipse=115:180:3.8 and 2.3];
\draw[very thick, cyan] (2.5,0) [partial ellipse=115:180:3.2 and 1.7];

\draw[very thick, black] (1,0) [partial ellipse=0:180:0.3 and 0.2];

\draw[thin] (2.5, 2) -- (3,2);
\draw[thin] (2.5, -2) -- (3,-2);

\draw[very thick] (2.5, 2.3) -- (3,2.3);
\draw[very thick] (2.5, -2.3) -- (3,-2.3);

\draw[very thick] (2.5, 1.7) -- (3,1.7);
\draw[very thick] (2.5, -1.7) -- (3,-1.7);

\draw[very thick, blue] (2.5,2) [partial ellipse=-90:90:0.2 and 0.3];
\draw[very thick, blue, dashed, opacity=0.4] (2.5,2) [partial ellipse=90:270:0.2 and 0.3];

\draw[thin] (-2.5, 2) -- (-3,2);
\draw[thin] (-2.5, -2) -- (-3,-2);

\draw[very thick] (-2.5, 2.3) -- (-3,2.3);
\draw[very thick] (-2.5, -2.3) -- (-3,-2.3);

\draw[very thick] (-2.5, 1.7) -- (-3,1.7);
\draw[very thick, red] (-2.5, -1.7) -- (-3,-1.7);


\begin{scope}[shift={(-2.5, 2)}]
\draw[thin] (0.1, 0.1) -- (0,0);
\draw[thin] (0.1, -0.1) -- (0,0);
\end{scope}

\begin{scope}[shift={(2.5, -2)}]
\draw[thin] (-0.1, 0.1) -- (0,0);
\draw[thin] (-0.1, -0.1) -- (0,0);
\end{scope}

\begin{scope}[shift={(-1, -0.2)}, rotate=10]
\draw[very thick, red] (0.1, 0.1) -- (0,0);
\draw[very thick, red](0.1, -0.1) -- (0,0);
\end{scope}

\begin{scope}[shift={(-0.52, 0.55)}, rotate=60]
\draw[very thick, cyan] (0.1, 0.1) -- (0,0);
\draw[very thick, cyan](0.1, -0.1) -- (0,0);
\end{scope}


\node at (0.1,-0.4) {$\gamma_{K,r}$};
\node at (-3.3,2) {$K$};
\node at (3.3,-2) {$J_r$};
\node at (2.7,1.3) {$b_{J_r}$};
\node at (1.6,1.2) {$\alpha_{J_r,s}$};
\node at (0.5,2.8) {$ \beta_{J_r}^{-1} * \mu_{J_r} * \beta_{J_r}$};
\draw[->] (1.8,2.8)
to [out=right, in=100] (2.5, 2.4);

\end{tikzpicture}

\caption{$\gamma_{K,r}$ as a meridian~$\mu_{J_r}$.}
\label{fig:MeridianExpression}
\end{figure}

\end{proof}

\begin{lemma}\label{lem:LongitudeMeridian}
The longitude~$\lambda_K \in \pi$ agrees with
\[ \lambda_K =  \big( \beta_K \big)_\# \gamma_{K, m_K} * \cdots *
\big( \beta_K \big)_\# \gamma_{K,1} * \big(\beta_K\big)_\# \delta_K.\]
\end{lemma}

\begin{proof}
The longitude~$\lambda_{C_K} \in \pi_1(C_K, b_K)$
is homotopic to
\[ \lambda_{C_K} = \gamma_{K, m_K} * \cdots * \gamma_{K, 1} * \delta_K \in \pi_1(C_K, b_K) \]
as depicted in Figure~\ref{fig:DiskHomotopy}. Whisker both sides with~$(\beta_K)_\#$ to obtain the statement.

\begin{figure}
\begin{tikzpicture}


\tikzset{
    partial ellipse/.style args={#1:#2:#3}{
        insert path={+ (#1:#3) arc (#1:#2:#3)}
    }
}


\draw[thin, black] (0.5,1.05)
to [out=55, in=255] (0.65, 1.35);

\draw[thin, black] (-1.35,-0.95)
to [out=45, in=225] (-1.05, -0.6);

\draw[thin, black] (1,-0.65)
to [out=-45, in=135] (1.3, -1);

\draw[very thick, black] (0,0) [partial ellipse=70:120:1.8 and 1.45];
\draw[very thick, black, loosely dotted] (0,0) [partial ellipse=120:150:1.8 and 1.45];
\draw[very thick, black] (0,0) [partial ellipse=150:270:1.8 and 1.45];

\draw[thin, black, dashed] (0,0.3) [partial ellipse=0:360:0.3 and 0.15];
\draw[thin, black] (0.3,0.3)
to [out=down, in=150] (0.4, -0.1);
\draw[thin, black] (-0.3,0.3)
to [out=down, in=30] (-0.4, -0.1);

\draw[thin, black] (0.4,0.9) [partial ellipse=0:360:0.2 and 0.17];

\draw[thin, black] (0.9,-0.5) [partial ellipse=0:360:0.2 and 0.17];

\draw[thin, black] (-0.9,-0.5) [partial ellipse=0:360:0.2 and 0.17];

\draw[very thick, black] (0,0) [partial ellipse=-90:70:1.8 and 1.45];

\node at (0, -1.8) {$b_K$};
\node at (-0.3, 1.7) {$\lambda_{C_K}$};

\begin{scope}[shift={(4, 0)}]
\draw[thin, black] (0.5,1.05)
to [out=55, in=255] (0.65, 1.35);

\draw[thin, black] (-1.35,-0.95)
to [out=45, in=225] (-1.05, -0.6);

\draw[thin, black] (1,-0.65)
to [out=-45, in=135] (1.3, -1);

\draw[very thick, black] (0,0) [partial ellipse=70:120:1.8 and 1.45];
\draw[very thick, black, loosely dotted] (0,0) [partial ellipse=120:150:1.8 and 1.45];
\draw[very thick, black] (0,0) [partial ellipse=150:240:1.8 and 1.45];

\draw[thin, black, dashed] (0,0.3) [partial ellipse=0:360:0.3 and 0.15];
\draw[thin, black] (0.3,0.3)
to [out=down, in=150] (0.4, -0.1);
\draw[thin, black] (-0.3,0.3)
to [out=down, in=30] (-0.4, -0.1);

\draw[thin, black] (0.4,0.9) [partial ellipse=0:360:0.2 and 0.17];

\draw[thin, black] (0.9,-0.5) [partial ellipse=0:360:0.2 and 0.17];

\draw[thin, black] (-0.9,-0.5) [partial ellipse=0:360:0.2 and 0.17];

\draw[very thick, black] (0,0) [partial ellipse=-100:70:1.8 and 1.45];

\draw[thin, black] (0,0) [partial ellipse=-120:-100:1.8 and 1.45];

\draw[very thick, black] (-1,-1.2)
to [out=-40, in=-60] (-0.65,-0.3)
to [out=120, in=-120] (-0.85, 0.1)
to [out=30, in=up] (-0.4, -0.8)
to [out=down, in=170] (-0.3, -1.43);

\node at (0, -1.8) {$b_K$};
\node at (-0.3, 1.7) {$\lambda_{C_K}$};

\end{scope}

\begin{scope}[shift={(8, 0)}]
\draw[thin, black] (0.5,1.05)
to [out=55, in=255] (0.65, 1.35);

\draw[thin, black] (-1.35,-0.95)
to [out=45, in=225] (-1.05, -0.6);

\draw[thin, black] (1,-0.65)
to [out=-45, in=135] (1.3, -1);

\draw[very thick, black] (0,0) [partial ellipse=70:120:1.8 and 1.45];
\draw[very thick, black] (0,0) [partial ellipse=75:210:1.65 and 1.30];
\draw[very thick, black] (0,0) [partial ellipse=-35:60:1.65 and 1.30];
\draw[very thick, black, loosely dotted] (0,0) [partial ellipse=120:150:1.8 and 1.45];
\draw[very thick, black] (0,0) [partial ellipse=150:227:1.8 and 1.45];

\draw[thin, black, dashed] (0,0.3) [partial ellipse=0:360:0.3 and 0.15];
\draw[thin, black] (0.3,0.3)
to [out=down, in=150] (0.4, -0.1);
\draw[thin, black] (-0.3,0.3)
to [out=down, in=30] (-0.4, -0.1);

\draw[thin, black] (0.4,0.9) [partial ellipse=0:360:0.2 and 0.17];
\draw[very thick, black] (0.4,0.9) [partial ellipse=90:390:0.35 and 0.32];

\draw[thin, black] (0.9,-0.5) [partial ellipse=0:360:0.2 and 0.17];
\draw[very thick, black] (0.9,-0.5) [partial ellipse=-20:270:0.35 and 0.32];

\draw[thin, black] (-0.9,-0.5) [partial ellipse=0:360:0.2 and 0.17];
\draw[very thick, black] (-0.9,-0.5) [partial ellipse=-100:190:0.35 and 0.32];

\draw[very thick, black] (0,0) [partial ellipse=-90:70:1.8 and 1.45];

\draw[thin, black] (0,0) [partial ellipse=227:270:1.8 and 1.45];

\draw[very thick, black] (-1.25,-1.04)
to [out=-45, in=left, looseness=0.4] (-0.95, -0.82);

\draw[very thick, black] (0.9, -0.82)
to [out=right, in=30] (1, -1)
to [out=-150, in=right] (0.2,-1.2)
to [out=left, in=left, looseness=0.4] (0.2,-0.3)
to [out=right, in=down] (0.6, 0)
to [out=up, in=right] (0.38, 0.2);

\draw[very thick, black] (0, -1.45)
to [out=left, in=right, looseness=0.3] (-0.2, -0.3)
to [out=left, in=down] (-0.6, 0)
to [out=up, in=left] (-0.38, 0.2);

\draw[very thick, black] (-1.25, -0.55)
to [out=-80, in=-50] (-1.43, -0.64);

\draw[very thick, black] (0.42, 1.26)
to [out=-20, in=-20] (0.39, 1.22);

\draw[very thick, black] (0.70, 1.05)
to [out=120, in=160] (0.83, 1.12);

\draw[very thick, black] (1.23, -0.60)
to [out=-100, in=-135] (1.36, -0.74);

\node at (0, -1.8) {$b_K$};
\node at (-0.3, 1.7) {$\lambda_{C_K}$};

\end{scope}

\begin{scope}[shift={(4, -1.45)}]
\draw[very thick, black] (-0.1, 0.1) -- (0.1, -0.1);
\draw[very thick, black] (-0.1, -0.1) -- (0.1, 0.1);
\end{scope}
\begin{scope}[shift={(8, -1.45)}]
\draw[very thick, black] (-0.1, 0.1) -- (0.1, -0.1);
\draw[very thick, black] (-0.1, -0.1) -- (0.1, 0.1);
\end{scope}
\begin{scope}[shift={(0, -1.45)}]
\draw[very thick, black] (-0.1, 0.1) -- (0.1, -0.1);
\draw[very thick, black] (-0.1, -0.1) -- (0.1, 0.1);
\end{scope}

\begin{scope}[shift={(1.8,0)}]
\draw[very thick, black] (-0.1, -0.1) -- (0, 0);
\draw[very thick, black] (0.1, -0.1) -- (0, 0);
\end{scope}
\begin{scope}[shift={(-1.8,0)}]
\draw[very thick, black] (0.1, 0.1) -- (0, 0);
\draw[very thick, black] (-0.1, 0.1) -- (0, 0);
\end{scope}

\begin{scope}[shift={(5.8,0)}]
\draw[very thick, black] (-0.1, -0.1) -- (0, 0);
\draw[very thick, black] (0.1, -0.1) -- (0, 0);
\end{scope}
\begin{scope}[shift={(2.2,0)}]
\draw[very thick, black] (0.1, 0.1) -- (0, 0);
\draw[very thick, black] (-0.1, 0.1) -- (0, 0);
\end{scope}

\end{tikzpicture}
\caption{Homotopy from the path~$\lambda_{C_K}$ to
$\gamma_{K, m_i} * \cdots * \gamma_{K, 1} * \delta_K$.}
\label{fig:DiskHomotopy}
\end{figure}
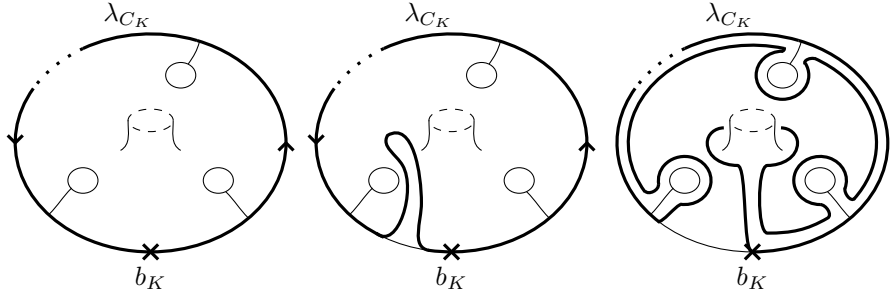

\end{proof}

Write $\varepsilon_{K,r} \in \{\pm 1\}$ for the sign of the intersection
at $a_{K,r}$ between $K$ and $\Sigma_{J_r}$. Note that the map $\sigma_K$, and each of the $\varepsilon_{K,r}$, can be read off from the abstract~C-complex. We now proceed to show that~$g_{K, r} \in \pi/\pi_2$
is also determined by the abstract~C-complex.

\begin{lemma}\label{lem:Conjugates}
The loop~$g_{K,r} = \beta_{K} * \alpha_{K,r} * \iota_{K,r} * \alpha^{-1}_{J,s} * \beta_J^{-1}$ has abelianisation
\[ [ g_{K,r} ] = \sum_{q = 1}^{r-1}\varepsilon_{K, q} [\mu_{\sigma_K(q)}]
 - \sum_{q = 1}^{s-1}\varepsilon_{J, q} [\mu_{\sigma_J(q)}]
 + \begin{cases}
 0 & \text{$r$-th clasp is positive} \\
 \mu_K - \mu_J & \text{$r$-th clasp is negative}
 \end{cases}\]
 in~$\pi/\pi_2$.
\end{lemma}

\begin{proof}
Note that $\pi/\pi_2$ is the free abelian group~$\Z\langle \mu_1, \ldots, \mu_n\rangle$
generated by the meridians, and that
the coefficient of $\mu_q$ is the intersection number~$[g_{K,r}] \cdot \Sigma_q$.
To compute this number, we make~$g_{K,r}$ transverse to each~$\Sigma_q$:
push off~$\alpha_{K,r}$ and~$\alpha^{-1}_{J,s}$ slightly to the negative side
of~$C_K$ and $C_J$, and let $\iota_{K,r}$ deform accordingly.
As~$\beta_K$ was chosen to be disjoint from the C-complex
and approaching $C_K$ from the negative side,
all intersections lie on the path~$ \alpha_{K,r} * \iota_{K,r} * \alpha^{-1}_{J,s}$.
There are two kind of contributions: the intersection points on the $\alpha$ paths,
which correspond to intersections of the components~$K_r$ and $K_s$
with surfaces of the C-complex. These give rise to the first two summands.
The second contribution are intersections points on the (deformed)~$\iota_{K,r}$.
These depend on the sign of the clasp, and can be computed from the explicit
local models, as shown in Figure~\ref{fig:PositiveNegativeClasp}.

\begin{figure}
\begin{tikzpicture}


\draw[thick, black] (1.41,1.29)
to [out=45, in=-135] (2.44,2.32)
to [out=60, in=30, looseness=1.3] (2.04, 2.74)
to [out=-135,in=45] (0.99,1.71);
\draw[thick, black, fill=white] (1.2, 1.5) circle [radius=0.3];

\draw[very thick, blue] (1.5, 1.5) -- (2.5,2.5) -- (6,2.5);


\draw[thick, white, fill=white] (3.9,3.5)
to [out=down, in=up] (3.9,0.5)
to [out=down, in=down, looseness=0.7] (4.7, 0.5)
to [out=up, in=down] (4.7, 3.5);

\draw[thick, black, fill=white] (4,3.5)
to [out=down, in=up] (4,0.5)
to [out=down, in=down, looseness=0.7] (4.6, 0.5)
to [out=up, in=down] (4.6, 3.5);

\draw[thick, black, fill=white] (4,3.5)
to [out=up, in=up, looseness=0.7] (4.6,3.5)
to [out=down, in=down, looseness=0.7] (4,3.5);


\draw[very thick, cyan] (0.5, 0.5) -- (4,0.5) -- (4,3.5) -- (0.5, 3.5);

\draw [very thick, cyan] (1.5,2.2) arc (135:-45:0.29);

\draw [very thick, blue] (4,2)
to [out=down, in=down, looseness=0.7] (4.6, 2);

\draw[very thick, red] (4.2, 3.4)
to [out=down, in=right, looseness=0.2] (4,2)
to [out=left, in=right] (2,2)
to [out=left, in=45] (1.42,1.68);

\draw[very thick, blue] (5,1.5) -- (1.5,1.5);


\begin{scope}[shift={(6.5,0)}]
\draw[thick, black] (1.41,1.29)
to [out=45, in=-135] (2.44,2.32)
to [out=60, in=30, looseness=1.3] (2.04, 2.74)
to [out=-135,in=45] (0.99,1.71);
\draw[thick, black, fill=white] (1.2, 1.5) circle [radius=0.3];

\draw[very thick, blue] (1.5, 1.5) -- (2.5,2.5) -- (6,2.5);


\draw[thick, white, fill=white] (3.9,3.5)
to [out=down, in=up] (3.9,0.5)
to [out=down, in=down, looseness=0.7] (4.7, 0.5)
to [out=up, in=down] (4.7, 3.5);

\draw[thick, black, fill=white] (4,3.5)
to [out=down, in=up] (4,0.5)
to [out=down, in=down, looseness=0.7] (4.6, 0.5)
to [out=up, in=down] (4.6, 3.5);

\draw[thick, black, fill=white] (4,3.5)
to [out=up, in=up, looseness=0.7] (4.6,3.5)
to [out=down, in=down, looseness=0.7] (4,3.5);


\draw[very thick, cyan] (0.5, 0.5) -- (4,0.5) -- (4,3.5) -- (0.5, 3.5);

\draw [very thick, cyan] (1.5,2.2) arc (135:-45:0.29);

\draw [very thick, blue] (4,2)
to [out=down, in=down, looseness=0.7] (4.6, 2);

\draw[very thick, red] (4.2, 3.4)
to [out=down, in=right, looseness=0.2] (4,2)
to [out=left, in=right] (2,2)
to [out=down, in=225] (2.44,2.32);

\draw[very thick, blue] (5,1.5) -- (1.5,1.5);

\node at (1, 3.1) {\large $-$};
\node at (5.4, 2.2) {\large $+$};

\end{scope}

\node at (1, 3.1) {\large $-$};
\node at (5.4, 2.2) {\large $-$};

\end{tikzpicture}
\caption{Contribution of the intersection points of~$\iota_{K,r}$ in a positive clasp (left)
and a negative clasp (right)
}\label{fig:PositiveNegativeClasp}
\end{figure}
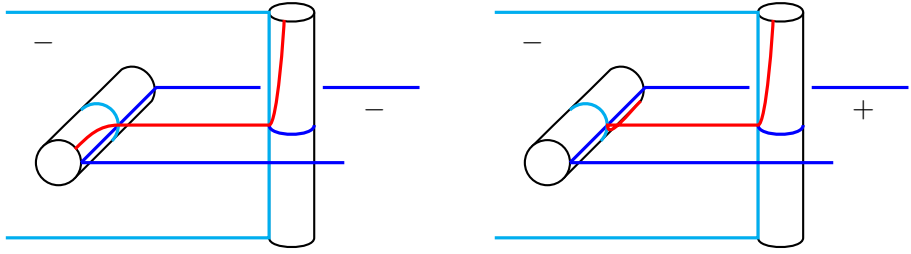

\end{proof}

\noindent We have now collected enough information on the longitude words to prove the first main technical theorem of this section.

\begin{theorem}\label{thm:SameWords}
Let $L$ and $L'$ be two oriented, ordered links with homeomorphic C-complexes and
fundamental groups~$\pi(L)$ and $\pi(L')$ respectively.
Then for every $k=1,\dots,n$, there exists words~$\ell_k, \ell_k' \in F$ such that both
\begin{enumerate}
\item $\ell_k(\mu_1, \ldots, \mu_n) = \lambda_k \text{ mod } \pi(L)_4$ and
$\ell'_k(\mu'_1, \ldots, \mu'_n) = \lambda'_k \text{ mod } \pi(L')_4$,
\item $\ell_k = \ell'_k \text { mod } F_3$
\end{enumerate}
hold.
\end{theorem}

\begin{proof}
As described above, the strategy is to show that the longitude words, modulo the required subgroups, can be read off from the C-complex in a way that only depends on the homeomorphism type of the C-complex together with the signs of the clasps. But the signs of the clasps are determined by the orientations of intersection arcs, as explained in Remark~\ref{remark:clasp-orientations}.

Pick basings~$\{\beta_i\}$ for $L$ that are subordinate to the C-complex.
For each link component $K$, define~$\sigma_K$, and $\varepsilon_{K, i}$ as above.
Substitute using Lemma~\ref{lem:MeridianRewrite} in the expression of
Lemma~\ref{lem:LongitudeMeridian} to write the longitude as
\[  {\big(\mu_{\sigma_K(m_K)}^{\varepsilon_{K,m_K}} \big)}^{g_{K,m_K}} * \cdots *
	{\big(\mu_{\sigma_K(1)}^{\varepsilon_{K,1}} \big)}^{g_{K,1}} * \delta_K \in \pi(L).\]
Pick a triple commutator~$d_K \in F_3$,
and words~$h_{K, r} \in F$ such that
\begin{align*}
d_K(\mu_1, \ldots, \mu_n) &= \delta_K \text{ mod } \pi_4 &
h_{K,r}(\mu_1, \ldots, \mu_n) &= g_{K,r} \text{ mod } \pi_4
\end{align*}
for all~$1\leq r \leq m_K$. Now define
\[ \ell_K = {\big(\mu_{\sigma_K(m_K)}^{\varepsilon_{K,m_K}} \big)}^{h_{K,m_K}} * \cdots *
	{\big(\mu_{\sigma_K(1)}^{\varepsilon_{K,1}} \big)}^{h_{K,1}} * d_K \in F, \]
which is a word such that~$\ell_K(\mu_1, \ldots, \mu_n) = \lambda_K \text{ mod } \pi_4$. Do this for each component~$K$ of~$L$.
	
Repeat all of the above with the link~$L'$ to obtain words~$\ell'_{K} \in F$; note that the orderings determine a bijection between the set of components of $L$ and the set of components of $L'$, and we use this identification implicitly from now on.
We claim that~$\ell_K = \ell'_K \in F/F_3$.
Note that both~$d_K$ and $d'_K$
are triple commutators, so they can be safely ignored. Since the C-complexes are homeomorphic, the only difference that can occur is in the conjugating words $h_{K,r}$  and $h'_{K,r}$.
By Lemma~\ref{lem:CommutatorCalculus}~(\ref{comm-calc-item-2}), with~$k=2$, we just have to show that
$h_{K,r} = h'_{K,r} \text{ mod } F_2$.

Observe that under the isomorphism~$F/F_2 \xrightarrow{\cong} \pi/\pi_2$,
the abelianisation~$[h_{K,r}]$ of~$h_{K,r}$ is sent to~$[g_{K,r}]$, which
we computed in Lemma~\ref{lem:Conjugates} to be:
\[ [ g_{K,r} ] = \sum_{q = 1}^{r-1}\varepsilon_{K, q} [\mu_{\sigma_K(q)}]
 - \sum_{q = 1}^{s-1}\varepsilon_{J, q} [\mu_{\sigma_J(q)}]
 + \begin{cases}
 0 & \text{$r$-th clasp is positive} \\
 \mu_K - \mu_J & \text{$r$-th clasp is negative}.
 \end{cases}\]
Consequently, the elements~$[ g_{K,r} ] = [g'_{K,r}]$ agree, and so
do~$[h_{K,r}] = [h'_{K,r}]$.
This shows~$\ell_K = \ell'_K \in F/F_3$.
\end{proof}

This enables us to prove Theorem~\ref{thm:Main-intro}~(\ref{item:main-thm-1})$\implies$(\ref{item:main-thm-4}).  First, as discussed in Section~\ref{section:C-complexes}, two links admit homeomorphic surface systems if and only if they admit homeomorphic C-complexes.  Then apply the next corollary.

\begin{corollary}\label{cor:homeo-C-cx-implies-pi-mod-pi-3}
Let $L$ and $L'$ be two oriented, ordered links with homeomorphic C-complexes and
fundamental groups~$\pi(L)$ and $\pi(L')$ respectively.
There exist choices of basings for $L$ and $L'$ and an isomorphism between the lower central series quotients
$\pi(L) / \pi(L)_{3}$ and $\pi(L') / \pi(L')_{3}$ that preserves the oriented, ordered meridians and the oriented, ordered longitudes determined by the respective basings.
\end{corollary}

\begin{proof}
  For every $i=1,\dots,n$, let $\ell_i$ and $\ell_i' \in F$ be as in Theorem~\ref{thm:SameWords}.
By Theorem~\ref{thm:milnor-thm-4}, the quotient~$\pi(L)/\pi(L)_3$
admits a presentation
\[ \pi(L) / \pi(L)_3 \xleftarrow{\cong} \Big \langle \mu_1, \ldots, \mu_n  \Bigm\vert [\mu_i, \ell_i],
	F_3\Big \rangle. \]
An analogous presentation holds for $L'$, simply replacing $\ell_i$ with $\ell_i'$ for each $i$ and changing each $\mu_i$ as $\mu_i'$.
The homomorphism defined by sending $\mu_i \mapsto \mu_i'$ is an isomorphism since the relation $[\mu_i, \ell_i]$ is sent to $[\mu_i',\ell_i]$, which equals $[\mu_i,\ell_i']$ modulo~$F_3$ by Lemma~\ref{lem:CommutatorCalculus}~(\ref{comm-calc-item-3}) with $k=2$.
The fact that $\ell_i$ and $\ell_i'$ coincide in $\pi/\pi_4$ implies that they coincide in $\pi/\pi_3$.
\end{proof}

Similarly, the next corollary implies Theorem~\ref{theorem:lower-central-series-intro}, because the condition that two links admit homeomorphic C-complexes is equivalent to any of the conditions in Theorem~\ref{thm:Main-intro}, as discussed in Section~\ref{section:C-complexes}.

\begin{corollary}
Let $L$ and $L'$ be two oriented, ordered links with homeomorphic C-complexes and
fundamental groups~$\pi(L)$ and $\pi(L')$ respectively.
There exists an isomorphism between the lower central series quotients
$\pi(L) / \pi(L)_{4}$ and $\pi(L') / \pi(L')_{4}$ that preserves oriented, ordered meridians.
\end{corollary}

\begin{proof}
For every $i=1,\dots,n$, let $\ell_i$ and $\ell_i' \in F$ be as in Theorem~\ref{thm:SameWords}.
Apply Theorem~\ref{thm:milnor-thm-4} to deduce that~$\pi(L)/\pi(L)_4$
admits the presentation
\[ \pi(L) / \pi(L)_4 \xleftarrow{\cong} \Big \langle \mu_1, \ldots, \mu_n  \Bigm\vert [\mu_i, \ell_i],
	F_4\Big \rangle. \]
The analogous presentation hold for $L'$, simply replacing $\mu_i$ with $\mu_i'$ and $\ell_i$ with $\ell_i'$ for each $i$.
By Lemma~\ref{lem:CommutatorCalculus}~(\ref{comm-calc-item-3}) with $k=3$,
the associated presentations are equal,
 and so we obtain
\begin{align*}
\pi(L) / \pi(L)_4 &\xleftarrow{\cong} \Big \langle \mu_1, \ldots, \mu_n  \Bigm\vert [\mu_i, \ell_i],
	F_4\Big \rangle \\
	&=\Big \langle \mu_1, \ldots, \mu_n  \Bigm\vert [\mu_i, \ell'_i(\mu_1,\ldots, \mu_n)],
	F_4\Big \rangle\\
	&= \Big \langle \mu'_1, \ldots, \mu'_n  \Bigm\vert [\mu'_i, \ell'_i],
	F_4\Big \rangle\\
	&\xrightarrow{\cong} \pi(L') / \pi(L')_4.
\end{align*}
\end{proof}

In the rest of the section, in order to show Theorem~\ref{thm:Main-intro}~(\ref{item:main-thm-4})$\implies$(\ref{item:main-thm-3}), we will show that the two weaker conditions below
are already enough to show that two links admit homeomorphic C-complexes:
\begin{align*}
\ell_k(\mu_1, \ldots, \mu_n) &= \lambda_k \text{ mod } \pi(L)_3 &
\ell'_k(\mu'_1, \ldots, \mu'_n) &= \lambda'_k \text{ mod } \pi(L')_3,\\
\ell_k &= \ell'_k \text { mod } F_3
\end{align*}
The second main technical result of this section is the next proposition.

\begin{proposition}\label{prop:mijk = magnus}
Let $L$ be a link with basing $\beta$ that is subordinate to a C-complex $\Sigma$.  For any word $\ell_k\in F$ in the meridians that is sent to the longitude~$\lambda_k = \ell_k(\mu_1, \ldots, \mu_n) \in \pi_1(L)/\pi_1(L)_3$, the $X_iX_j$--coefficient in the Magnus expansion of~$\ell_k$ is
\[ e_{ij}(\ell_k) = m_{ijk}(\Sigma) - \lk(L_k, L_j)\lk(L_i,L_j), \]
where $m_{ijk}$ is the quantity defined in  Section~\ref{section:milnor-numbers}.
\end{proposition}

\begin{proof}
In order to match the notation used previously in this section,
let $K$ be the~$k$-th component of the link~$L$.
Recall from Lemma~\ref{lem:LongitudeMeridian} that
$\lambda_K =  \big( \beta_K \big)_\# \gamma_{K, m_p} * \cdots *
\big( \beta_K \big)_\# \gamma_{K,1} * \big(\beta_K\big)_\# \delta_K$,
where $p = m_K$ is the number of clasps in $\Sigma_k$.
Appeal to Lemma~\ref{lem:GenusTripleCommutator} and Lemma~\ref{lem:MeridianRewrite}
to obtain
\[ \lambda_K =  {\big(\mu_{J_p} \big)}^{g_{K,p}} * \cdots *{\big(\mu_{J_1} \big)}^{g_{K,1}} \text{ mod }\pi_3.\]
Pick~$h_{K,i} \in F$ such that $h_{K,i}(\mu_1, \ldots, \mu_n) = g_{K,i}$.
Now define
\[ \ell_K := {\big(\mu_{J_p} \big)}^{h_{K,p}} * \cdots *{\big(\mu_{J_1} \big)}^{h_{K,1}} \in F. \]
By Remark~\ref{rem:MagnusExp}, we can compute
the Magnus expansion of $\ell_K$ from $e_{ij}(\ell_K)$ and it only depends
on the coset $\ell_K \in F/F_3$.

Multiple applications of the occurrence calculus from Equation~\eqref{eqn:e-relations}
in Section~\ref{section:milnor-numbers}, in particular the relation $e_{rs}(u \cdot v) =e_{rs}(u) + e_{rs}(v) + e_r(u) e_s(v)$, reveals that
\begin{equation}\label{eqn:e-ij-lK-1}
\begin{split}
e_{ij}(\ell_K) &= e_{ij}\left({\big(\mu_{J_p} \big)}^{h_{K,p}} * \cdots *{\big(\mu_{J_1} \big)}^{h_{K,1}}\right)\\
&= \sum_{r=1}^p e_{ij}\left({\big(\mu_{J_r} \big)}^{h_{K,r}}\right) +
\sum_{r=1}^p e_{i}\left({\big(\mu_{J_r} \big)}^{h_{K,r}}\right)e_{j}
	\left(\ell_{K, r-1}\right)\\
&= \sum_{r=1}^p e_{ij}\left({\big(\mu_{J_r} \big)}^{h_{K,r}}\right) +
\sum_{r=1}^p e_{i}\left(\mu_{J_r}\right) e_{j}	\left(\ell_{K, r-1}\right),
\end{split}
\end{equation}
where $\ell_{K,r} := {\big(\mu_{J_r} \big)}^{h_{K,r}} * \cdots *{\big(\mu_{J_1} \big)}^{h_{K,1}}$.
By inspection, $e_j(\ell_{K,r-1})$ is precisely the number of clasps between $\Sigma_k$ and $\Sigma_j$ that appear before the $r$-th clasp, counted with sign, while
\[ e_{i}\left(\mu_{J_r}\right) = \begin{cases}
 0 & J_r\neq K_i \\
 1 & J_r = K_i\text{ and the }r\text{-th clasp is positive}
 \\
-1 & J_r = K_i\text{ and the }r\text{-th clasp is negative.}
 \end{cases}\]
 Thus $\sum_{r=1}^p e_{i}\left(\mu_{J_r}\right) e_{j}\left(\ell_{K, r-1}\right)$ counts how many times a $(\Sigma_K,\Sigma_j)$ clasp appears before a $(\Sigma_K, \Sigma_i)$ clasp (counting with signs).  That is,
 \begin{equation}\label{eqn:e-ji-wk-2}
   \sum_{r=1}^p e_{i}\left(\mu_{J_r}\right)e_{j}\left(\ell_{r-1}\right) = e_{ji}(w_k),
   \end{equation}
 where $w_k$ is the clasp-word of the component~$\Sigma_k$
of the C-complex $\Sigma$.

It remains to analyse $\sum_{r=1}^p e_{ij}\big({\big(\mu_{J_r} \big)}^{h_{K,r}}\big)$.
Note that $e_j\big(h_{K,r}^{-1}\big) = -e_j\big(h_{K,r}\big)$.
Expand the trivial word~$h_{K,r} h_{K,r}^{-1}$ using
Equation~\eqref{eqn:e-relations}
to see that~$e_{ij}\big(h_{K,r}^{-1}\big) = -e_{ij}(h_{K,r})+ e_i(h_{K,r})e_j(h_{K,r})$. Together with $e_{ij}(\mu_{J_r})=0$, it follows that
 \[
 e_{ij}\left({\big(\mu_{J_r} \big)}^{h_{K,r}}\right)
 =
 e_{ij}\big(h_{K,r}* \mu_{J_r}* h_{K,r}^{-1}\big)
 =
 e_{i}(h_{K,r}) e_j(\mu_{J_r}) - e_i(\mu_{J_r})e_{j}(h_{K,r}),
 \]
 so that
 \begin{equation}\label{eqn:e-ij-3}
 \sum_{r=1}^p e_{ij}\left({\big(\mu_{J_r} \big)}^{h_{K,r}}\right)  =
 \sum_{r=1}^p e_{i}(h_{K,r}) e_j(\mu_{J_r})
 	- \sum_{r=1}^p e_i(\mu_{J_r})e_{j}(h_{K,r}).\end{equation}

Notice that since $e_j(\mu_{J_r}) = 0$ unless $J_r = K_j$, and since $e_i(\mu_{J_r}) = 0$ unless~$J_r = K_i$,  we may throw out most of the terms in the sums above.  When $J_r = K_j$, we read $e_{i}(h_{K,r})$ from the abelianisation of $h_{K,r}$ found in Lemma~\ref{lem:Conjugates}:
 \[
 e_{i}(h_{K,r}) = \sum_{q = 1}^{r-1}\varepsilon_{K, q}
 	e_i\big( \mu_{\sigma_K(q)}\big)
 	-  \sum_{q = 1}^{s(r)-1} \varepsilon_{J, q} e_i\big( \mu_{\sigma_J(q)}\big).
 \]
The notation $s(r)$ has not appeared for some time: the definition can be found just after Definition~\ref{defn:subordination}.
Now,  $\sum_{q = 1}^{r-1}\varepsilon_{K, q} e_i\big( \mu_{\sigma_K(q)}\big)$
returns the number of $(\Sigma_k,\Sigma_i)$ clasps on $\Sigma_k$ (counted with sign) prior to the $r$-th clasp of $\Sigma_k$.  On the other hand,~$\sum_{q = 1}^{s-1}\varepsilon_{J, q} e_i\big(\mu_{\sigma_J(q)}\big)$ gives the number of $(\Sigma_j , \Sigma_i)$ clasps on $\Sigma_j$ (counted with sign)  prior to the same clasp (now we order the clasps by following $\partial \Sigma_j$).
Let~$e_i^{< r} (w_k)$ be the number of signed occurrences of the letter~$i$
before the $r$-th letter in the word~$w_k$. In this notation:
\[ \sum_{q = 1}^{r-1}\varepsilon_{K, q} e_i\big( \mu_{\sigma_K(q)}\big)
	= e_i^{< r} (w_k)
\]
Now compute
\begin{equation}\label{eqn:e-ij-4}
\begin{split}
\sum_{r=1}^p e_{i}(h_{K,r}) e_j(\mu_{J_r})
&= \sum_{r=1}^p \left( \sum_{q = 1}^{r-1}\varepsilon_{K, q}
 	e_i\big( \mu_{\sigma_K(q)}\big)
 	-  \sum_{q = 1}^{s(r)-1} \varepsilon_{J, q} e_i\big( \mu_{\sigma_J(q)}\big)\right)e_j(\mu_{J_r}) \\
&= \sum_{r=1}^p \big( e_i^{< r} (w_k) -  e_i^{< s(r)} (w_j)\big)e_j(\mu_{J_r}) \\
&= \sum_{r=1}^p e_i^{< r} (w_k)e_j(\mu_{J_r})  - \sum_{r=1}^p e_i^{< s(r)} (w_j)e_j(\mu_{J_r}) \\
&= e_{ij}(w_k) - e_{ik}(w_j) \\
\end{split}
\end{equation}
Rename the indices to get
\begin{equation}\label{eqn:e-ij-5}
\sum_{r=1}^p e_{j}(h_{K,r})  e_i(\mu_{J_r}) = e_{ji}(w_k) - e_{jk}(w_i),
\end{equation}
 from which we deduce, by substituting (\ref{eqn:e-ij-4}) and (\ref{eqn:e-ij-5}) into (\ref{eqn:e-ij-3}):
 \begin{equation}\label{eqn:e-ij-6}
 \sum_{r=1}^p e_{ij}\left({\big(\mu_{J_r} \big)}^{h_{K,r}}\right)  = e_{ij}(w_k) - e_{ik}(w_j) - e_{ji}(w_k) + e_{jk}(w_i).
 \end{equation}
 Take the right hand side of (\ref{eqn:e-ij-lK-1}), apply (\ref{eqn:e-ij-6}) to the first term, and (\ref{eqn:e-ji-wk-2}) to the second term, to yield:
\begin{align*}
 e_{ij}(\ell_k) &= e_{ji}(w_k) + e_{ij}(w_k) - e_{ik}(w_j) - e_{ji}(w_k) + e_{jk}(w_i)\\
 &= e_{ij}(w_k) - e_{ik}(w_j) + e_{jk}(w_i)\\
 &= e_{ij}(w_k) + e_{ki}(w_j) - e_k(w_j)e_i(w_j) + e_{jk}(w_i).
\end{align*}
This shows the claim, since $e_k(w_j)e_i(w_j) = \lk(L_k, L_j)\lk(L_i, L_j)$ and $m_{ijk} = e_{ij}(w_k) + e_{ki}(w_j) + e_{jk}(w_i)$ by definition.
\end{proof}

The proof of Theorem~\ref{thm:LCS converse} below will require choosing a C-complex subordinate to an arbitrary basing, and we demonstrate that this is always possible in the following lemma.

\begin{lemma}\label{lem:Subordinate}
If $\beta$ is any basing of the link $L$, then there exists a C-complex $\Sigma$ for $L$ that is subordinate to $\beta$.
\end{lemma}

\begin{proof}
Let $\operatorname{pt}$ be a choice of base point in $X_L$ and let $\beta$ be any basing for $L$.  For each  link component $K_i$, recall that $b_i$ denotes the end point of $\beta_i$ that lies on $K_i$.  Let $\Sigma$ be a C-complex for $L$ disjoint from $\operatorname{pt}$.  Make a local change to each $\Sigma_i$ close to $b_i \in K_i$ to arrange that $\beta_i$ approaches the point $b_i$ from the negative side of $\Sigma_i$, and otherwise has no points in common with $\Sigma_i$, at least close to $b_i$.

The local move in Figure \ref{fig:subordinate1} allows us to eliminate a point of intersection between~$\Sigma_i$ and the corresponding basing arc $\beta_i$ at a cost of adding intersections with every other base arc $\beta_j$ ($j\neq i$).  This move will also force us out of the category of~C-complexes as it may introduce many new intersections between $\Sigma_i$ and the other components of $\Sigma$.

\begin{figure}
\begin{tikzpicture}[scale=1]


\tikzset{
    partial ellipse/.style args={#1:#2:#3}{
        insert path={+ (#1:#3) arc (#1:#2:#3)}
    }
}

\draw[thick] (0,2) -- (0,-2);

\begin{scope}[shift={(0, 0)}]
\draw[thick, black] (-0.1, 0.1) -- (0.1, -0.1);
\draw[thick, black] (-0.1, -0.1) -- (0.1, 0.1);
\end{scope}

\draw[fill=red, fill opacity = 0.1, draw opacity=0] (0,2)
-- (0,-2)
-- (-0.6, -2.2)
-- (-0.6, 1.8);

\begin{scope}[shift={(2, 0.2)}]
\draw[thick, dotted, red, fill=red, fill opacity = 0.1] (0,-0.3)
-- (0,-1)
-- (-1.2, -1.4)
-- (-1.2, 0.6)
-- (0,1)
-- (0,-0.1);
\end{scope}

\begin{scope}[shift={(3,0)}]
\draw[very thick, black, fill = black] (0,0) [partial ellipse=0:360:0.04 and 0.04];
\draw (0,0)
to [out=70, in=down] (0,1)
to [out=up, in =-70] (-0.3,1.7);
\draw[very thick, black, fill = black] (0,0) [partial ellipse=0:360:0.04 and 0.04];
\draw (0,0)
to [out=10, in=200] (1,0.3)
to [out=20, in =160] (1.5,0);
\draw[very thick, black, fill = black] (0,0) [partial ellipse=0:360:0.04 and 0.04];
\draw (0,0)
to [out=-50, in=left] (1,-0.5)
to [out=right, in =left] (1.5,-1);

\draw (0,0) -- (-1.6, 0);
\draw[opacity=0.2] (-1.6,0) -- (-2.15, 0);
\draw (-2.3,0) -- (-3, 0);
\end{scope}

\begin{scope}[shift={(1.4, 0)}]
\draw[thick, black] (-0.1, 0.1) -- (0.1, -0.1);
\draw[thick, black] (-0.1, -0.1) -- (0.1, 0.1);
\end{scope}

\node at (0.25, -0.3) {$b_i$};
\node at (-0.3, 2.2) {$\Sigma_i$};
\node at (1.4,1.3) {$\Sigma_i$};
\node at (2.95,-0.4) {$\pt$};
\node at (2.5, 0.3) {$\beta_i$};

\begin{scope}[shift={(7,0)}]

\draw[thick] (0,2) -- (0,-2);

\begin{scope}[shift={(0, 0)}]
\draw[thick, black] (-0.1, 0.1) -- (0.1, -0.1);
\draw[thick, black] (-0.1, -0.1) -- (0.1, 0.1);
\end{scope}

\draw[fill=red, fill opacity = 0.1, draw opacity=0] (0,2)
-- (0,-2)
-- (-0.6, -2.2)
-- (-0.6, 1.8);

\begin{scope}[shift={(2, 0.2)}]
\draw[thick, dotted, red, fill=red, fill opacity = 0.1] (0,-0.5)
-- (0,-1)
-- (-1.2, -1.4)
-- (-1.2, 0.6)
-- (0,1)
-- (0,0.1);
\end{scope}

\begin{scope}[shift={(3,0)}]
\draw[very thick, black, fill = black] (0,0) [partial ellipse=0:360:0.04 and 0.04];
\draw (0,0)
to [out=70, in=down] (0,1)
to [out=up, in =-70] (-0.3,1.7);
\draw[very thick, black, fill = black] (0,0) [partial ellipse=0:360:0.04 and 0.04];
\draw (0,0)
to [out=10, in=200] (1,0.3)
to [out=20, in =160] (1.5,0);
\draw[very thick, black, fill = black] (0,0) [partial ellipse=0:360:0.04 and 0.04];
\draw (0,0)
to [out=-50, in=left] (1,-0.5)
to [out=right, in =left] (1.5,-1);

\draw [opacity=0.2] (0,0) -- (-1.6, 0);
\draw[opacity=0.2] (-1.6,0) -- (-2.15, 0);
\draw (-2.3,0) -- (-3, 0);
\end{scope}

\draw[red, thick, fill=red, fill opacity=0.1] (1.4, 0.4)
to [out=down, in=left] (1.6, 0.2)
to [out=right, in=left] (3, 0.2)
to [out=right, in=right, looseness=2] (3, -0.2)
to [out=left, in=right] (1.6, -0.2)
to [out=left, in=up] (1.4, -0.4);
\draw[draw opacity=0, fill opacity=0.1, fill = red] (1.4,0) [partial ellipse=90:270:0.1 and 0.4];

\end{scope}

\end{tikzpicture}
\caption{Left:  A point of intersection between $\Sigma_i$ a component of a C-complex and $\beta_i$ the corresponding basing arc.  Right:  A finger move replaces this point of intersection with a point of intersection between $\Sigma_i$ and $\beta_j$ for every $j\neq i$.}\label{fig:subordinate1}
\end{figure}
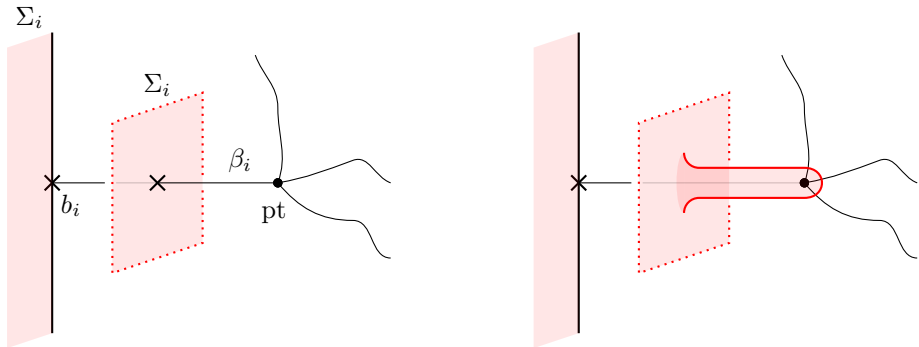

Thus, we need consider only intersection points in $\Sigma_i\cap \beta_j$ with $j\neq i$.  The finger move in Figure \ref{fig:subordinate2} replaces this point of intersection with a ribbon intersection between $\Sigma_i$ and $\Sigma_j$.

We have thus produced a surface system for $L$ that is subordinate to the basing~$\beta$.  In~\cite[Lemma 1]{Cimasoni04}, a general procedure is given for transforming a surface system into a C-complex. It amounts to a series of finger moves, each of which involves pushing $\Sigma_i$ along an arc in $\Sigma_j$ for some $j \neq i$.  By ensuring that the arc avoids the point $b_j$, we arrange that these finger moves do not introduce new intersections between $\Sigma$ and $\beta$.

\begin{figure}
\begin{tikzpicture}[scale=0.95]


\tikzset{
    partial ellipse/.style args={#1:#2:#3}{
        insert path={+ (#1:#3) arc (#1:#2:#3)}
    }
}

\draw[thick] (0,2) -- (0,-2);

\begin{scope}[shift={(0, 0)}]
\draw[thick, black] (-0.1, 0.1) -- (0.1, -0.1);
\draw[thick, black] (-0.1, -0.1) -- (0.1, 0.1);
\end{scope}

\draw[fill=red, fill opacity = 0.1, draw opacity=0] (0,2)
-- (0,-2)
-- (-1.2, -2.4)
-- (-1.2, 1.6);

\begin{scope}[shift={(2, 0.2)}]
\draw[thick, dotted, blue, fill=blue, fill opacity = 0.1] (-1.2,-0.1)
-- (-1.2, 0.6)
-- (0,1)
-- (0,-1)
-- (-1.2, -1.4)
-- (-1.2, -0.1);
\end{scope}

\begin{scope}[shift={(3,0)}]
\draw[very thick, black, fill = black] (0,0) [partial ellipse=0:360:0.04 and 0.04];
\draw (0,0)
to [out=70, in=down] (0,1)
to [out=up, in =-70] (-0.3,1.7);
\draw[very thick, black, fill = black] (0,0) [partial ellipse=0:360:0.04 and 0.04];
\draw (0,0)
to [out=10, in=200] (1,0.3)
to [out=20, in =160] (1.5,0);
\draw[very thick, black, fill = black] (0,0) [partial ellipse=0:360:0.04 and 0.04];
\draw (0,0)
to [out=-50, in=left] (1,-0.5)
to [out=right, in =left] (1.5,-1);

\draw (0,0) -- (-1.6, 0);
\draw[opacity=0.2] (-1.6,0) -- (-2.1, 0);
\draw (-2.3,0) -- (-3, 0);
\end{scope}

\begin{scope}[shift={(1.4, 0)}]
\draw[thick, black] (-0.1, 0.1) -- (0.1, -0.1);
\draw[thick, black] (-0.1, -0.1) -- (0.1, 0.1);
\end{scope}

\node at (0.35, -0.3) {$b_i$};
\node at (-0.6, 2.1) {$\Sigma_i$};
\node at (1.4,1.4) {$\Sigma_j$};
\node at (2.95,-0.4) {$\pt$};
\node at (2.5, 0.3) {$\beta_i$};

\begin{scope}[shift={(7,0)}]

\draw[thick] (0,2) -- (0, -0.15);
\draw[thick] (0, -0.35) -- (0,-2);

\begin{scope}[shift={(0, 0)}]
\draw[thick, black, opacity=0.3] (-0.1, 0.1) -- (0.1, -0.1);
\draw[thick, black, opacity=0.3] (-0.1, -0.1) -- (0.1, 0.1);
\end{scope}

\draw[fill=red, fill opacity = 0.1, draw opacity=0] (0,2)
-- (0,-2)
-- (-1.2, -2.4)
-- (-1.2, 1.6);

\begin{scope}[shift={(2, 0.2)}]
\draw[thick, dotted, blue, fill=blue, fill opacity = 0.1] (-1.2,0.0)
-- (-1.2, 0.6)
-- (0,1)
-- (0,-1)
-- (-1.2, -1.4)
-- (-1.2, -0.0);
\end{scope}

\begin{scope}[shift={(3,0)}]
\draw[very thick, black, fill = black] (0,0) [partial ellipse=0:360:0.04 and 0.04];
\draw (0,0)
to [out=70, in=down] (0,1)
to [out=up, in =-70] (-0.3,1.7);
\draw[very thick, black, fill = black] (0,0) [partial ellipse=0:360:0.04 and 0.04];
\draw (0,0)
to [out=10, in=200] (1,0.3)
to [out=20, in =160] (1.5,0);
\draw[very thick, black, fill = black] (0,0) [partial ellipse=0:360:0.04 and 0.04];
\draw (0,0)
to [out=-50, in=left] (1,-0.5)
to [out=right, in =left] (1.5,-1);

\draw (0,0) -- (-1.6, 0);
\draw[opacity=0.2] (-1.77,0) -- (-2.1, 0);
\draw[opacity=0.2] (-2.3,0) -- (-3, 0);
\end{scope}

\begin{scope}[shift={(1.4,0)}]
\draw[draw opacity=0, fill=blue, fill opacity=0.1] (0,0) [partial ellipse=110:250:0.08 and 0.4]
to [out=110, in=right] (-0.2, -0.25)
to [out=left, in=right] (-1.5, -0.25)
to [out=left, in=left, looseness=2] (-1.5, 0.25)
to [out=right, in=left] (-0.2, 0.25)
to [out=right, in=-100] (-0.09, 0.3);
\draw[blue, thick, opacity=0.3] (-0.04, 0.35)
to [out=-110, in=right] (-0.2, 0.25)
to [out=left, in=right] (-0.5, 0.25);
\draw[blue, thick] (-0.7, 0.25)
to [out=left, in=right] (-1.3, 0.25);
\draw[blue, thick] (-1.5,0.25)
to [out=left, in=left, looseness=2] (-1.5, -0.25)
to [out=right, in=left] (-0.7, -0.25);
\draw[out=right, in=left, opacity=0.3, blue,thick] (-0.5,-0.25)
to [out=right, in=left] (-0.2, -0.25)
to [out=right, in=up] (-0.04, -0.35);
\draw[thick, blue, fill=white, fill opacity=0.2] (0,0) [partial ellipse=15:345:0.08 and 0.4];

\draw (1.6,0) -- (0, 0);
\end{scope}

\end{scope}

\end{tikzpicture}
\caption{Left:  A point of intersection between $\Sigma_j$ a component of a C-complex and as basing arc $\beta_i$ with $i\neq j$.  Right:  A finger move replaces this point of intersection with a ribbon intersection between $\Sigma_i$ and $\Sigma_j$.}\label{fig:subordinate2}
\end{figure}
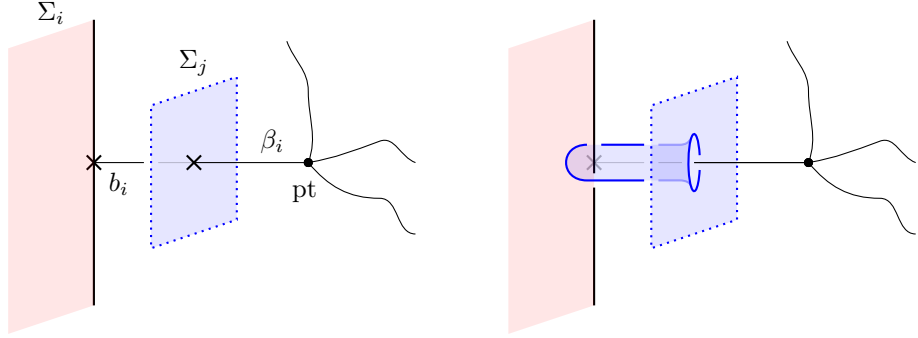

\end{proof}

Our final theorem uses Proposition~\ref{prop:mijk = magnus} to prove the
remaining implication: Theorem~\ref{thm:Main-intro}~(\ref{item:main-thm-4}) $\implies$ (\ref{item:main-thm-3}).

\begin{theorem}\label{thm:LCS converse}
Let $L$ and $L'$ be $n$-component ordered, oriented links with basings~$\beta$ and
$\beta'$ respectively.  Suppose that for each $k=1,\dots, n$, there exists a word~$\ell_k\in F$ such that
\[
\ell_k(\mu_1, \ldots, \mu_n) = \lambda_k \text{ mod } \pi(L)_3 \text{ and }
\ell_k(\mu'_1, \ldots, \mu'_n) = \lambda'_k \text{ mod } \pi(L')_3 \]
Then the total Milnor invariants $\mu(L)$ and $\mu(L')$ agree.
\end{theorem}

\begin{proof}
In order to emphasise the dependence of $m_{ijk}$ on the C-complex we will write $m_{ijk}(\Sigma)$ and $m_{ijk}(\Sigma')$.
For each $k=1,\dots, n$, pick a word~$\ell_k\in F$ with $\ell_k(\mu_1, \ldots, \mu_n) = \lambda_k$ and $\ell_k(\mu'_1, \ldots, \mu'_n) = \lambda_k'$ as in the statement of the theorem.
Since the linking numbers can be computed in terms of~$\ell_k\in F/F_2$, we conclude that the links $L$ and $L'$ have identical pairwise linking numbers.

Pick C-complexes~$\Sigma$ and $\Sigma'$ such that $\beta$ and $\beta'$
are subordinate to them. The existence of such~C-complexes is guaranteed by Lemma~\ref{lem:Subordinate}.
Since the links~$L$ and~$L'$ have identical pairwise linking numbers,
Proposition~\ref{prop:mijk = magnus} implies, for any~$i,j,k$,  that
\[ m_{ijk}(\Sigma) = e_{ij}(\ell_{k}) + \lk(L_k, L_j)\lk(L_i,L_j) =  m_{ijk}(\Sigma'). \] Since these are C-complexes, they have no triple intersections.  Thus, even as elements of $W = \bigwedge^3  \Z^n$, the total Milnor invariants $\mu(L)$ and $\mu(L')$ agree. Thus they also agree in the quotient~$\Pow$.
\end{proof}

\bibliographystyle{alpha}
\bibliography{Milnor}
\end{document}